\documentclass[final]{siamltex}
\usepackage{amsmath}
\usepackage{amssymb}
\usepackage{enumerate}
\usepackage{color}
\usepackage{booktabs}
\usepackage{graphicx,epsfig}
\usepackage{subfigure}
\usepackage{epstopdf}
\usepackage{multirow}
\usepackage{multicol}
\usepackage{arydshln}
\usepackage{bm}

\usepackage{cancel}                
\usepackage[normalem]{ulem}

\usepackage{algorithm}
\usepackage{algpseudocode}

\allowdisplaybreaks

\def\R{{\Bbb R}}

\def\Om{\Omega}

\def\ep{\epsilon}
\def\s{\sigma}
\def\f{\frac}
\def\p{\partial}

\def\na{\nabla}

\def\n{\mathbf{n}}

\def\r{\mathbf{r}}
\def\x{\mathbf{x}}
\def\y{\mathbf{y}}
\def\z{\mathbf{z}}

\def\B{\mathbf{B}}

\def\E{\mathbf{E}}
\def\J{\mathbf{J}}

\def\mE{{\mathcal E}}

\def\naxy{\na_{\mbox{\tiny xy}}}

\def\Deltaxy{\Delta_{\mbox{\tiny xy}}}

\def\nax1x2{\na_{x_1x_2}}
\def\nay1y2{\na_{y_1y_2}}
\def\nax{\na_{\tiny\x}}
\def\nay{\na_{\tiny\y}}

\def\eref#1{(\ref{#1})}

\newtheorem{remark}{Remark}[section]

\title{Convergence analysis of the harmonic $B_z$ algorithm with single injection current in MREIT
\thanks{Y Song was supported by Shandong Provincial Outstanding Youth Fund (No. ZR2018JL002), NSFC(No.11501336) and the China Postdoctoral Science Foundation (2019T120604, 2018M630795). R Sadleir was supported by the National Institute of Mental Health under grant RF1-114290. J Liu was supported by NSFC (No.11971104). }
}

\author{Yizhuang Song\thanks{School of Mathematics and Statistics, Center for Post-doctoral Studies of Management Science and Engineering, Shandong Normal University, Jinan, 250014, P.R.China.}
\and Rosalind Sadleir\thanks{School of Biological and Health Systems Engineering, Arizona State University, Tempe, AZ, 85287-9709, USA.}
\and Jijun Liu\thanks{Corresponding author. School of Mathematics, Southeast University, Nanjing, 210096, P.R.China. (e-mail: jjliu@seu.edu.cn).\quad Nanjing Center for Applied Mathematics, Nanjing, 211135, P.R.China.}
}

\begin{document}
\maketitle

\begin{abstract}
Magnetic resonance electrical impedance tomography (MREIT) aims to recover the electrical conductivity distribution of an object using  partial information of magnetic flux densities inside the tissue which can be measured using an MRI scanner, with the advantage that a higher spatial resolution of conductivity image can be provided than existing EIT techniques involving surface measurements. Traditional MREIT reconstruction algorithms use two data sets obtained with two linearly independent injected currents. However, injection of two currents is often not possible in applications such as transcranial electrical stimulation. Recently, we proposed an iterative conductivity reconstruction algorithm called the single current harmonic $B_z$ algorithm that demonstrated satisfactory performance in numerical and phantom tests. In this paper, we provide a rigorous mathematical analysis of the convergence of the iterative sequence for realizing this algorithm. We prove that, applying some mild conditions on the exact conductivity, the iterative sequence converges to the true solution within an explicit error bound. Such theoretical results substantiate the reasonability and efficiency of the proposed algorithm. We also provide more numerical evidence to validate these theoretical results.

\end{abstract}
\begin{keywords}
Inverse problems, biomedical imaging, MREIT, single injection current, harmonic $B_z$ algorithm, iteration, convergence, numerics.
\end{keywords}

\begin{AMS}
35R30, 35J61, 35Q61
\end{AMS}

\pagestyle{myheadings}
\thispagestyle{plain}
\markboth{Y. Song, R. Sadleir and J. Liu}{ Convergence of the single current MREIT}

\section{Introduction}

The electrical conductivity $\sigma$ and permittivity $\epsilon$ of biological tissues are fundamental indices of tissue state, being influenced by molecular composition, intra- and extra-cellular fluid balance, ionic composition and frequency, amongst other factors. Tissue electrical properties are significantly different in different pathological and physiological states including ischemia, hemorrhage, edema, inflammation, cancer and neural activity \cite{Liu2017,SeoWoo2011,SeoWoo2014,Widlak2012,WooSeo2008}.  Therefore, abnormal distributions of conductivity and permittivity can reveal early pathological changes in biological tissue that are of potential importance for medical diagnoses. Electrical property imaging aims to extract the tomographic conductivity and permittivity distributions of biological tissue by measuring magnetic fields resulting from an external current field applied to it. Detailed images of conductivity distributions may be obtained and can be examined against standard MRI images.

Depending on the method used to generating the external electrical field $\E$ and measuring the corresponding magnetic responses, there are several modalities to visualize the electrical tissue properties \cite{SeoWoo2011,Song2013a}.
Magnetic resonance electrical impedance tomography (MREIT) is a recently developed imaging technique which is capable of providing us a higher spatial resolution conductivity image at low frequency. In MREIT, pairs of electrodes are typically attached to the surface of the imaged object and the object with electrodes are placed into the bore of an MRI scanner. To reconstruct the conductivity distribution, we assume a sinusoidal current $I\sin(\omega t)$  ($\omega/2\pi \leq 1$ kHz) is injected into the object through the electrodes. The injected current will induce a current density $\J=(J_x,J_y,J_z)$ and a magnetic flux density $\B=(B_x,B_y,B_z)$ inside the object. If the direction of the main magnetic field is parallel to the $z$-axis, $B_z$, the $z$-component of $\B$, can be measured from MRI phase data \cite{SeoWoo2011,WooSeo2008}. The inverse problem for MREIT is to reconstruct the conductivity distribution from the measured $B_z$ data.

This inverse problem is ill-posed in the sense of non-uniqueness of the reconstruction and non-stability of the reconstruction process, if no restriction on the configuration is specified. More precisely, two different distributions of tissue conductivity could produce the same $B_z$ if there are no restrictions on the conductivity. To handle this non-uniqueness, existing reconstruction methods involve obtaining two sets of magnetic field data, using two independent injected currents delivered through two pairs of surface electrodes. These algorithms include the harmonic \cite{OhLeeWooetal2003,Seo2003} and non-iterative harmonic $B_z$ algorithms \cite{JeonLeeWoo2017,SeoJeonLeeWoo2011}, and algorithms involving an intermediate step approximating current densities from the $B_z$ data before reconstruction \cite{GaoHe2008,JeonLeeWoo2017,NamParkKwon2008,OranIder2012,ParkLeeKwon2007a}. The readers are refereed to \cite{SeoWoo2011,SeoWoo2014} for a review of existing reconstruction algorithms using two injected currents.

Since it takes a long time to measure data for two magnetic fields, the temporal resolution of biological tissue imaging using this configuration is severely affected by using two-current methods. In addition, it is cumbersome and impractical to attach two pairs of electrodes in some clinical applications including transcranial electrical stimulation \cite{Bikson2019}. For these reasons, although we could accelerate the data acquisition through sub-sampling the time-consuming phase encoding process to improve the temporal resolution \cite{Song2016,SongAmmariSeo2017,SongSadleir2018}, the most efficient reconstruction algorithm is to exploit the $B_z$ set produced by only one injection current. This approach has been used in \cite{Lee2010,SongSadleir2020}, to develop MREIT reconstruction methods that may be more suitable for practical clinical implementation.

One of the crucial issues in MREIT imaging algorithm using single injection currents is to handle the non-uniquess of the conductivity imaging model. Fortunately, the uniqueness for two-dimensional simplified model can still be ensured under an {\it a-priori} assumption  that the conductivity in the object boundary is known \cite{ParkLeeKwon2007a}. The proof is based on the uniqueness of a linear boundary value problem with respect to $\ln\sigma$ for the hyperbolic partial differential equation (PDE)
\begin{equation}\label{linear hyperbolic0}
  \widetilde \J^\bot \cdot \naxy \ln \sigma = - \f{1}{\mu_0}\Deltaxy B_z
\end{equation}
in two-dimensional imaging object for known internal current $\widetilde \J=(J_x,J_y)$, which can be determined directly from $B_z$.
Here, $\naxy=(\partial_x,\partial_y)$ and $\Deltaxy = \f{\p^2}{\p x^2}+\f{\p^2}{\p y^2}$ represent the two-dimensional gradient and Laplacian operators, respectively and $\cdot^\bot$ represents the anticlockwise right-angle rotation of a 2-dimensional vector, i.e., $\widetilde \J^\bot = (J_y,-J_x)$.

However, a stable reconstruction algorithm based on solving this first order PDE remains to be determined, due to the numerical instability of computing  the Laplacian operator on $B_z$, given that our practical inversion input data are $B_z$ instead of $\Deltaxy B_z$. The treatments of this instability for MREIT models using two injection currents can be found in \cite{XuLiu2010,XuLiu2012}.
A plausible numerical way to the solution of \eref{linear hyperbolic0} could be via the finite element or finite difference methods. Values of $B_z$ must be determined on an extremely fine mesh to obtain accurate estimates of $\widetilde\J^\bot$ and $\Deltaxy B_z$ using differential computations. However, the measured $B_z$ input data  are only available at a relatively coarse resolution. Without using adaptive refining meshes, once there exist some discontinuities in $\ln \sigma$ or there is a mismatch with the boundary values, the numerical solution for $\sigma$  from \eref{linear hyperbolic0} will be severely degraded because of the Gibbs phenomenon \cite{Lax2006,OranIder2012}. The other implementable way to the solution of \eref{linear hyperbolic0} is the method of characteristic lines for PDEs. However, noise and numerical errors will propagate along characteristic lines, and severe artifacts could occur near their ends \cite{Nachman2007}, which prevents us from recovering accurate $\sigma$ inside the object using $B_z$ and known values of $\sigma$ in the boundary.


In \cite{SongSadleir2020}, we proposed a reconstruction algorithm, called the single current harmonic $B_z$ algorithm, to solve the two-dimensional first order linear hyperbolic PDE \eref{linear hyperbolic0} with respect to $\ln\sigma$ to determine the conductivity using a single $B_z$ data set. In this novel algorithm, we take advantage of the forward model $\naxy\cdot\widetilde\J=0$ to track the change of conductivity along the direction of the current density. The numerical iteration scheme implemented and phantom experiments showed this algorithm was successful.  However, a strict mathematical theory regarding the convergence property as well as the convergence rate of the iterative process for approximating $\sigma$ has not yet been given. This is essential to develop efficient implementations and better quantitative evaluations for this novel algorithm.

In this paper, we provide a rigorous mathematical analysis for the convergence of the harmonic $B_z$ algorithm with single injection current for a simplified two-dimensional model, together with the error estimates on the iteration process. To this end, we firstly show that a cylindrical three-dimensional object with infinite length under some physical configurations can be transformed into a two-dimensional model for which we will consider the uniqueness of the inverse problem and the convergence of iterative scheme for the conductivity reconstruction. It should be emphasized that the corresponding results for the MREIT model with one injection current and a general three-dimensional object remain to be found. We prove that, under some mild condition on the true conductivity, the iterative sequence for our single injection current imaging model converges to the true conductivity in the space of $C^1(\Om)$, where $\Om$ represents the 2-dimensional imaging object. Numerical simulations are also presented to validate the convergence findings.

We arrange this paper as follows. In section 2, we briefly introduce the inverse problem model for MREIT with one-current injection, and then derive a 2-dimensional model for which we state the uniqueness of the conductivity reconstruction. In section 3, we review the single current harmonic $B_z$ algorithm proposed in \cite{SongSadleir2020} and establish our main result, the convergence analysis of this algorithm together with the error estimate, which gives a quantitative description of the inversion algorithm. In section 4, we validate the proposed theory for a two-dimensional toy model, the Shepp-Logan model and a more practical CT model. Some conclusions and  possible future research topics are finally stated in section 5.

\section{Two-dimensional MREIT model from one injected current}

In this section, we introduce the mathematical model of MREIT using $B_z$ data resulting from injection of one external current from the object boundary, and then state the recently developed iterative algorithm for a two-dimensional MREIT model, for which we will prove the convergence property in the next section.

\subsection{Problem formulation: from a three-dimensional to a two-dimensional model}
Assume that the imaged conductive object occupies a bounded three-dimensional domain $\Om\subset \R^3$ with a smooth boundary $\partial\Omega$. The parameter of the object to be reconstructed is the conductivity $\sigma$. To this end, MREIT technique excites the conductive object using an externally injected current, and then measures the corresponding response $B_z$ inside the medium \cite{SeoWoo2014}.

In MREIT, we attach a pair of surface electrodes $\mE^\pm$ to $\p\Om$ and place the object and electrodes into the bore of an MRI scanner as shown in Figure \ref{Fig:attach_electrode}.
We inject a sinusoidal current $i(t) = I\sin \omega t$ through the electrodes, where the angular frequency $\omega$ satisfies $0 \leq \frac{\omega}{2\pi}\leq $a few kilohertz \cite{SeoWoo2011,WooSeo2008}.
\begin{figure}[h]
\onecolumn
\centering
\subfigure[]{
\includegraphics[width=4.2cm,height=3.7cm]{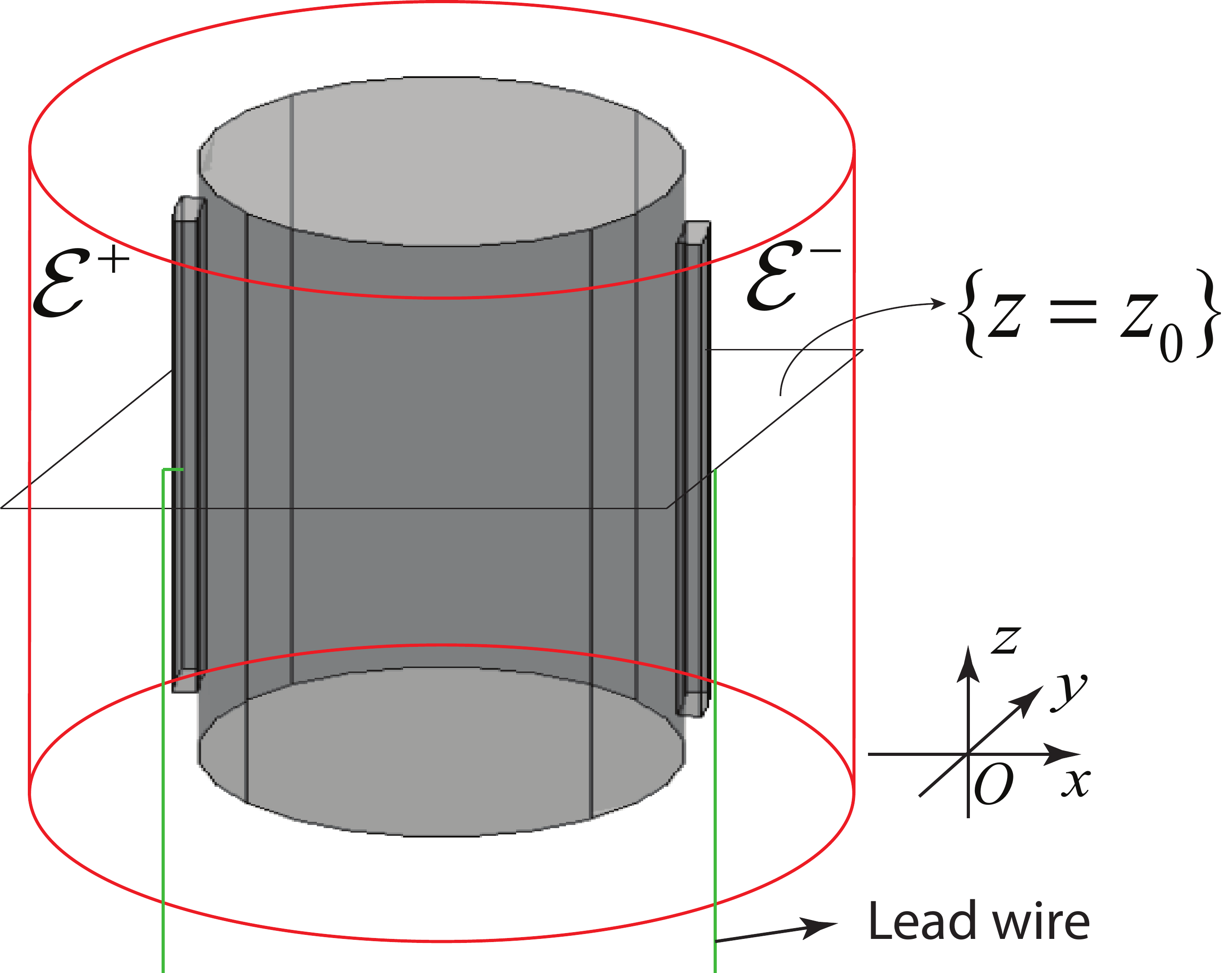}}
~\subfigure[]{
\includegraphics[width=6.5cm,height=3cm]{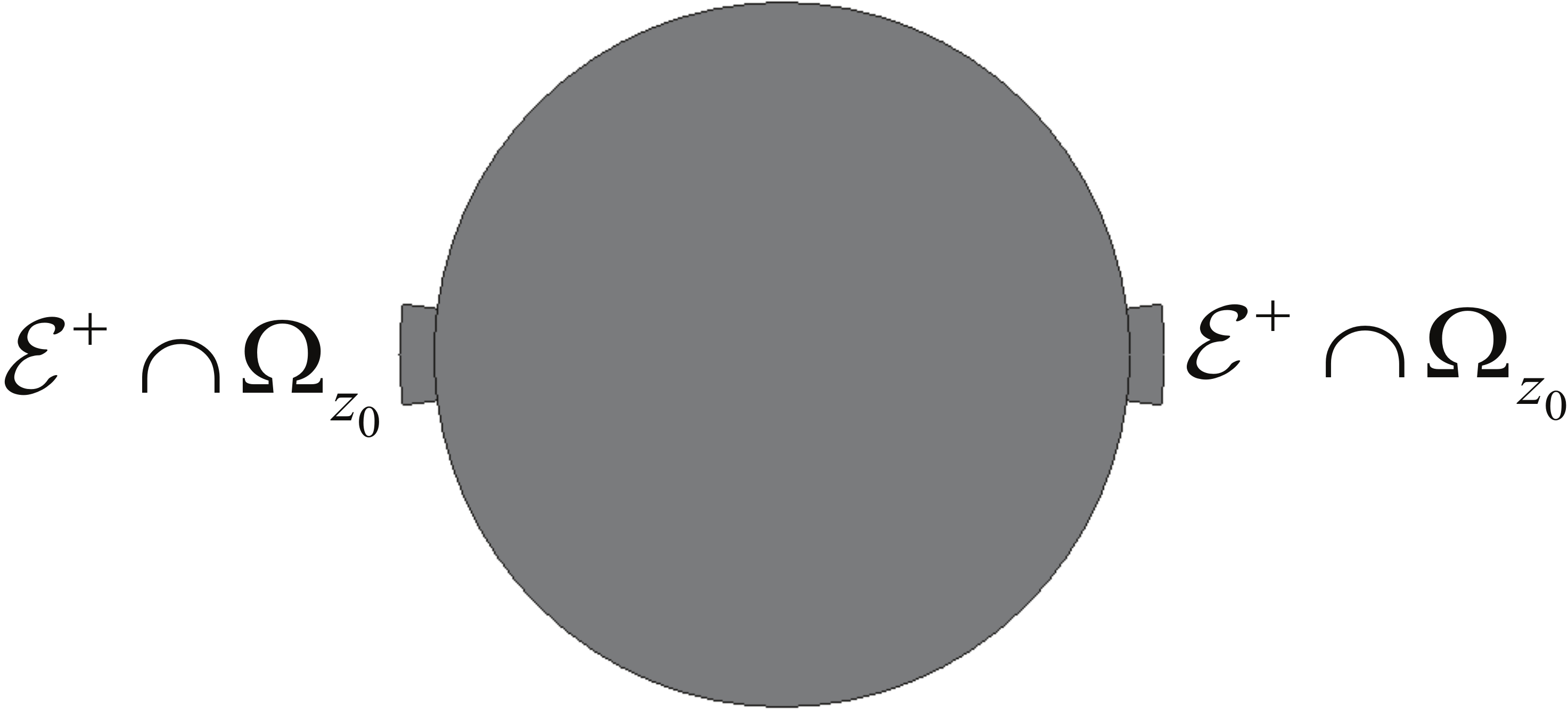}}
\caption{(a) Illustration of three-dimensional cylindrical model with surface electrodes $\mE^\pm$ on the boundary of $\Om$. (b) two-dimensional cross section $\Om_{z_0}$ of the three-dimensional cylindrical model $\Omega$ in (a).}\label{Fig:attach_electrode}
\end{figure}
Then the injected current will induce an electrical flux density $\J=(J_x,J_y,J_z)$, an electrical potential $u$ and a magnetic flux density $\B=(B_x,B_y,B_z)$ within the conductive medium $\Om$. The voltage potential $u = u[\sigma](\r)$ for $\r:=(x,y,z)$ is governed by the following partial differential equation with mixed nonlocal boundary conditions
\begin{eqnarray}\label{eq:forwardPDE}
    \begin{cases}
      \na\cdot (\sigma \na u) = 0 \qquad \mbox{in }\Om, \\
      \int_{\mE^+}\sigma \frac{\p u}{\p \n }dS = I = -\int_{\mE^-}\sigma \frac{\p u}{\p \n }dS, \\
      \na u\times \n|_{\mE^\pm} = {\bf 0}, \\
      \left.\sigma \frac{\p u}{\p \n }\right|_{\p \Om \setminus \overline{\mE^+\cup \mE^-}} = 0,
    \end{cases}
\end{eqnarray}
where $\n$ is the unit outer normal vector on $\p \Om$ and $dS$ is a surface area element. It has been proven that there exists a unique solution to \eref{eq:forwardPDE} under the extra restriction
\begin{equation}\label{liu21}
  u|_{\mE^-} = 0,
\end{equation}
see  \cite{Somersalo1992}.
Note that condition (\ref{liu21}) can be easily achieved in practical experiment by connecting $\mE^-$ to ground. Moreover, the unique solution can be represented by the solution to a PDE with mixed local boundary conditions \cite{Liu2010}.

Assume that the main magnetic field direction of the MRI scanner is parallel to the $z$-axis. Only the $z$-component of the magnetic flux density $\B$, $B_z$, can be measured from the MRI scanner. This partial magnetic field information will be our inversion input data for recovering $\sigma$.
The inverse problem for MREIT imaging is to reconstruct $\sigma$ based on the implicit relation between $\sigma$ and $B_z$ described by the Biot-Savart law \cite{Stratton1941}, that is,
\begin{eqnarray}\label{biot-savart law}
  B_z(\r) &=& \frac{\mu_0}{4\pi}\int_\Om \frac{\langle \r-\r',-\sigma(\r')\na u(\r')\times \hat \z \rangle}{|\r-\r'|^3}d\r'+\mathcal{H}(\r)
  \nonumber\\&=&
  \mu_0\int_\Om\langle\nabla_{\r'}\Phi(\r-\r'), -\sigma(\r')\na u(\r')\times \hat \z \rangle d\r'+\mathcal{H}(\r)
  \nonumber\\&=&
  \mu_0\int_{\mathbb{R}^3}\Phi(\r-\r')\nabla_{\r'}\cdot(\sigma(\r')\na u(\r')\times \hat \z) d\r'+\mathcal{H}(\r)
\end{eqnarray}
for $\r= (x,y,z)\in \Om$. The last identity in \eref{biot-savart law} is due to the fact that the electrical current ${\bf J}=-\sigma(\r)\na u(\r)$ is compactly supported in $\Om$,
$\Phi(\r-\r'):=-\frac{1}{4\pi|\r-\r'|}$ represents the fundamental solution of Laplacian operator in three-dimensional space, $\hat \z = (0,0,1)$
and $\mu_0 = 4\pi\times 10^{-7} $ H/m is the magnetic permeability of free space.
Note that the convolution is interpreted in the sense of distributions, since $\na\times{\bf J}$ is in general a distribution supported in $\overline{\Om}$.

Effects of stray magnetic fields $\mathcal{H}$ will always exist in the measured $B_z$ data from the currents in the lead wires and electrodes. Hence, it is difficult to reconstruct $\sigma$ directly from $B_z$ data because of the unknown field $\mathcal{H}$.
Fortunately, since $\mathcal{H}$ is harmonic, if the Laplacian operator is applied to both sides of
\eref{biot-savart law}, the uncertain effects of $\mathcal{H}$ will be removed.
Consequently,
most existing MREIT reconstruction methods determine $\sigma$ by applying a three-dimensional Laplacian operator on both sides of \eref{biot-savart law} to obtain
$$\Delta B_z(\r)=\mu_0\int_{\mathbb{R}^3}\delta(\r-\r')\nabla_{\r'}\cdot(\sigma(\r')\na u(\r')\times \hat \z) d\r'=\mu_0\nabla_{\r}\cdot(\sigma(\r)\na u(\r)\times \hat \z),$$
where $\delta$ is the Dirac delta function. Hence, we obtain the following first order hyperbolic equation
\begin{equation}\label{nonlinear hyperbolic}
  \naxy^\bot u\cdot \naxy\s \equiv\p_x(\sigma u_y)-\p_y(\sigma u_x)= \f{1}{\mu_0}\Delta B_z
\end{equation}
in each two-dimensional slice $\Om_{z_0} = \Om\cap \{z=z_0\}$, where $\naxy^\bot:=(\p_y, -\p_x)$.
Obviously, (\ref{nonlinear hyperbolic}) can be written as
\begin{equation}\label{liu15}
(\s\p_y u,-\s\p_xu)\cdot\naxy\ln\s=\f{1}{\mu_0}\Delta B_z
\end{equation}
in each two-dimensional slice $\Om_{z_0}$.
We note that as $u$ depends on $\sigma$ nonlinearly, equation \eref{liu15} is a nonlinear differential equation with respect to $\sigma$. Using Ohm's law $\J = -\sigma\na u$, we can transform the equation \eref{nonlinear hyperbolic} to be a first order linear hyperbolic equation \eref{linear hyperbolic0} with respect to $\ln\sigma$ in two-dimensional or a special three-dimensional cylindrical case. This is because $\J$ can be directly determined from $B_z$ by an explicit process in a two-dimensional MREIT model without knowledge of $\sigma$. We will state the determination of $\J$ from a given data $B_z$ in Theorem \ref{SolvingJ}.

Next we will show that the reconstruction problem in a special three-dimensional cylinder can essentially be simplified into a two-dimensional case.
Consider an infinite three-dimensional cylinder aligned with the $z$-direction namely
$$\Omega:=\{{\bf r}=(x,y,z): (x,y)=r(\theta)(\cos\theta,\sin\theta)\hbox{ for }\theta\in [0,2\pi],\; z\in \mathbb{R}^1\}\subset \mathbb{R}^3$$
with $2\pi-$periodic polar radius $r(\theta)>0$, and we set the electrode pairs $\mE^{\pm}$ to have infinite length along the $z-$direction. Assume that the conductivity in $\Om$ is uniform along $z$-direction, i.e., $\s({\bf r})\equiv\s(x,y)$, and the injection current $g := \sigma\nabla u\cdot {\bf n}$ imposed on $\p\Om$ is also independent of $z$.  Then both the solution $u(\bf r)$ to (\ref{eq:forwardPDE}) and the corresponding $\B(\bf r)$ is independent of $z$.  In this case, for any $g(x,y)\in H^{-1/2}_\diamond(\mE^{\pm})$ with a two-dimensional slice defined as
$$\Omega_{z_0}:=\{(x,y,z_0):(x,y)=r(\theta)(\cos\theta,\sin\theta)\hbox{ for }\theta\in [0,2\pi]\}\subset \mathbb{R}^2$$
and
$$H_\diamond^{-1/2}(\mE^{\pm}):=\left\{\phi(x,y):\;\phi\in H^{-1/2}(\mE^{\pm})
\mbox{ and }\int_{\mE^{\pm}}\phi(x,y) dl = 0\right\},$$
the system (\ref{eq:forwardPDE}) defines a map
\begin{equation}\label{liu24}
    \Lambda_\sigma~:g\in H^{-1/2}_{\diamond}(\mE^{\pm}) \mapsto \frac{\mu_0}{4\pi}\int_\Om \frac{\langle \r-\r',-\sigma(\r')\na u(\r')\times \hat \z \rangle}{|\r-\r'|^3}d\r'
\end{equation}
from a boundary current injection $(g,\mE^{\pm})$ to the internal magnetic field component $B_z$.
Since $\Om$ is a cylinder along the $z$ direction, and $\sigma(\r), u(\r)$ are independent of $z$, the right hand side of \eref{liu24} is
\begin{eqnarray}\label{liu14}
& &\frac{\mu_0}{4\pi}\int_\Om \frac{\langle \r-\r',-\sigma(\r')\na u(\r')\times \hat \z \rangle}{|\r-\r'|^3}d\r'
\nonumber\\&=&
\frac{\mu_0}{4\pi}\int_{\mathbb{R}^1}dz'\int_{\Om_{z'}}
\frac{(x-x',y-y')\cdot\s(x',y')(-u_{y'}(x',y'),u_{x'}(x',y'))}
{\sqrt{((x-x')^2+(y-y')^2+(z-z')^2)^3}}dx'dy'
\nonumber\\&\equiv&
\frac{\mu_0}{2\pi}\int_{\Om_{z_0}}
\frac{(x-x',y-y')\cdot\s(x',y')(-u_{y'}(x',y'),u_{x'}(x',y'))}
{(x-x')^2+(y-y')^2}dx'dy'
\end{eqnarray}
for $(x,y,z_0)\in\Om_{z_0}$, where the last equality comes from the fact that
$\int_0^{+\infty}\frac{1}{(\sqrt{1+s^2})^3}ds=1$.

Therefore, to reconstruct $\sigma$ for a two-dimensional or a special three-dimensional cylinder MREIT model from inversion input data $B_z$, we need only to solve the linear equation \eref{linear hyperbolic0} with respect to $\ln\sigma$ from known boundary values of $\sigma$. Note that the equation \eref{linear hyperbolic0} is linear since in these situations $\widetilde\J$ can be determined by $\Deltaxy B_z$.  However, this process needs a uniform lower bound of $|\J|$ in $\Om$ to obtain a stable solution of $\sigma$. Unfortunately, such a uniform lower bound cannot be ensured for a general three-dimensional object $\Omega$. For details, see Corollary \ref{corollary:uniform_lower_bound} and \cite{Liu2010}.

Noticing that both $g$ and the right-hand side of \eref{liu14} are independent of $z$, we will write
the two-dimensional slice $\Om_{z_0}$ as
\begin{equation*}\label{liu13}
\Om:=\{(x,y): x=r(\theta)(\cos\theta, \sin\theta), \theta\in [0,2\pi]\}
\end{equation*}
again with $2\pi-$periodic polar radius $r(\theta)>0$, to simplify notation.

Subsequently, we assume that $\Om\subset \R^2$ and always consider the case that $\sigma(x,y)\in C^1(\overline \Om)$.  For any $g\in H_\diamond^{-1/2}(\mE^{\pm})$, \eref{liu24}-\eref{liu14} define a Neumann-to-$B_z$ map from $H_\diamond^{-1/2}(\mE^{\pm}) \to H^1(\Om)$, by $u$ solving \eref{eq:forwardPDE} in a two-dimensional domain $\Om$ with $g$ given by the Neumann data. Based on the above relationship between the two- and three-dimensional problems, notations $\na$, $\Delta$, $\n$ and $\J\equiv \widetilde \J$  are always assumed to be those for two-dimensional cases. Moreover, for the two-dimensional case, we also denote $\r = (x,y)$.



Let $\widetilde \Om \subset\subset \Om\subset {\mathbb{R}^2}$ be a connected domain such that $\Om\setminus \widetilde \Om$ is a doubly connected domain and $\p\Om$ is the outer boundary of $\Om\setminus \widetilde \Om$.
We will prove, if $\sigma|_{\Om\setminus \widetilde \Om}$ is known, that we can uniquely reconstruct $\sigma$ from the measured data $B_z[\sigma]$ corresponding to a single injected current.
To this end, we firstly need to uniquely determine ${\bf J}[\sigma]$ from inversion input data $B_z[\sigma]$ directly for the two-dimensional case, without using the values of $\sigma$. To simplify our explanation, we denote $\overline{\p\Om\setminus \mE^\pm} = \Gamma^+\cup \Gamma^-$, where $\mE^+, \Gamma^+,  \mE^-, \Gamma^- \subset\partial\Om$ are in the counterclockwise direction, see Fig.~\ref{Fig:attach_electrode}(b).

The following theorem  provides a way of recovering $\J$  using only the given data $B_z$.

\begin{theorem}\label{SolvingJ}
For a two-dimensional MREIT model, the internal current $\J (x,y)$ from the nonlocal model \eref{eq:forwardPDE} can be determined from $B_z(x,y)$ directly by
\begin{equation}\label{eq:Jmain}
\J(x,y) = (J_x, J_y) = \na^\bot\left(\phi-\f{I}{2}\psi\right),
\end{equation}
where $\phi$ and $\psi$ are the solutions to the boundary value problems
\begin{eqnarray}\label{eq:Jmain_phi}
\begin{cases}
    \Delta \phi = \f{1}{\mu_0}\Delta B_z, &\mbox{in }\Om \\
    {\bf n}\cdot\na \phi = 0, &\mbox{on } \mE^+\cup \mE^- \\
    \phi|_{\Gamma^{\pm}} = 0
\end{cases}
\end{eqnarray}
and
\begin{eqnarray}\label{eq:Jmain_psi}
\begin{cases}
    \Delta \psi = 0, &\mbox{in }\Om \\
    {\bf n}\cdot \na\psi = 0, &\mbox{on } \mE^+\cup \mE^- \\
    \psi|_{\Gamma^{\pm}} = \pm 1,
\end{cases}
\end{eqnarray}
respectively.
\end{theorem}

\begin{proof}
From \cite{SeoJeonLeeWoo2011}, for a known $\phi,\psi$ defined by \eref{eq:Jmain_phi} and  \eref{eq:Jmain_psi} respectively, the internal current has the representation $\J(\r) = \na^\bot(\phi+\beta \psi)$, where
$\beta$ is a scaling factor defined as
\begin{equation}\label{eq:beta}
  \beta = \f{\int_{\mE^+}\sigma\na u\cdot \n d\ell - \int_{\mE^+}\na^\bot \phi\cdot \n d\ell}{\int_{\mE^+}\na^\bot \psi \cdot \n  d\ell},
\end{equation}
where $d\ell$ is the arc length element. It remains to prove $\beta = -\f{I}{2}$.

Indeed, from the boundary condition of the forward equation \eref{eq:forwardPDE} we have
\begin{equation}\label{beta_1}
\int_{\mE^+}\sigma \na u\cdot \n d\ell = I.
\end{equation}
Let us calculate $\int_{\mE^+} \na^\bot \phi\cdot \n d\ell$ and $\int_{\mE^+} \na^\bot \psi\cdot \n d\ell$.
We firstly parameterize the curve $\mE^+$ by ${\bm\gamma}(\ell)=(x(\ell),y(\ell))$, where $0\leq \ell\leq |\mE^+|$ is the arc length parameter with $|\mE^+|$ representing the measure of the electrode $\mE^+$. Then from the boundary condition in \eref{eq:Jmain_phi}, we obtain that $\phi({\bm\gamma}(0)) = \phi({\bm\gamma}(|\mE^+|))=0$. Moreover, the straightforward calculation using the chain rule yields
\begin{equation*}
  \f{d}{d\ell}\phi({\bm\gamma}(\ell)) = \na\phi({\bm\gamma}(\ell))\cdot {\bm\gamma}'(\ell) = \na^\bot \phi({\bm \gamma}(\ell))\cdot \n,
\end{equation*}
where the last equality comes from the fact that ${\bm \gamma}'(\ell)$ is the unit tangential vector to $\mE^+$.
From the Newton-Leibniz formula, we obtain
\begin{equation}\label{beta_2}
  \int_{\mE^+} \na^\bot \phi\cdot \n d\ell = \int_{\mE^+}\f{d}{d\ell}\phi({\bm\gamma}(\ell))d\ell = \phi({\bm\gamma})(|\mE^+|) - \phi({\bm\gamma})(0)= 0-0 = 0.
\end{equation}
Using the same argument, we can prove that
\begin{equation}\label{beta_3}
  \int_{\mE^+} \na^\bot \psi\cdot \n d\ell = -2.
\end{equation}

Combining the identities \eref{beta_1}, \eref{beta_2} and \eref{beta_3} we obtain $\beta = -\f{I}{2}$. This completes the proof.
\end{proof}

Due to the requirement to calculate three line integrals with integrands being normal derivatives in a small area of $\mE^+$ to obtaining $\beta$ as defined in \eref{eq:beta}, reconstruction of the current density $\J$ could be sensitive to numerical error using the method described in \cite{SeoJeonLeeWoo2011}. The formula \eref{eq:Jmain} enables us  to recover $\J$ from the magnetic field directly, which can be realized in an efficient way. In fact, noticing $\J^\bot=\na\phi-\f{I}{2}\na\psi$, an efficient computation scheme for $\J$ is essentially to compute $\na\phi$ and $\na\psi$ in $\widetilde\Om$.

Mathematically,  we obtain $\Delta(\phi - \f{1}{\mu_0}B_z) = 0$  from \eref{eq:Jmain_phi} for known $B_z$. Hence the systems \eref{eq:Jmain_phi} and \eref{eq:Jmain_psi} for $\phi,\psi$ can be unified in the form
\begin{eqnarray}\label{liu51}
\begin{cases}
    \Delta W = 0 &\mbox{in }\Om \\
    {\bf n}\cdot\na W = b_n &\mbox{on } \mE^+\cup \mE^- \\
    W = b_d &\mbox{on }\Gamma^\pm
\end{cases}
\end{eqnarray}
for known boundary data $(b_n, b_d)$. This Laplacian equation with mixed boundary condition can be solved using the boundary equation method, with the representation
\begin{eqnarray*}
W({\bf r})&=&-\int_{\mE^\pm}[b_n({\bf r'})\Psi({\bf r'}-{\bf r})-\p_{\bf n}\Psi({\bf r'}-{\bf r})W({\bf r'})]d\ell(\r')-
\nonumber\\& &\int_{\p\Om\setminus\mE^{\pm}}[\Psi({\bf r'}-{\bf r})\p_{\bf n}W({\bf r'})-\p_{\bf n}\Psi({\bf r'}-{\bf r})b_d({\bf r'})]d\ell(\r'), \quad {\bf r}\in\widetilde\Om,
\end{eqnarray*}
where $\Psi({\bf r'}-{\bf r})=\frac{1}{2\pi}\ln\frac{1}{|{\bf r'}-{\bf r}|}$ is the fundamental solution to the Laplacian operator in two-dimensional cases, and the density function $(W|_{\mE^{\pm}}, \p_{\bf n}W|_{\p\Om\setminus\mE^{\pm}})$ can be determined from the jump relations on $\p\Om$ for single and double layer potentials \cite{Folland1995}. Finally we have
\begin{eqnarray}\label{gradientW}
\na W({\bf r})&=&-\int_{\mE^\pm}[b_n({\bf r'})\na\Psi({\bf r'}-{\bf r})-\p_{\bf n}\na\Psi({\bf r'}-{\bf r})W({\bf r'})]d\ell(\r')-
\nonumber\\& &\int_{\p\Om\setminus\mE^{\pm}}[\na\Psi({\bf r'}-{\bf r})\p_{\bf n}W({\bf r'})-\na\p_{\bf n}\Psi({\bf r'}-{\bf r})b_d({\bf r'})]d\ell(\r')
\end{eqnarray}
for $\r\in\widetilde\Om$, which avoids the need for numerical differentiation to obtain $W$ by computing the right hand side directly. Moreover, we noticed that the integrands in the right hand side are smooth for ${\bf r}\in\widetilde\Omega$.

\begin{remark}
  The formula \eref{gradientW} provides us a convenient way to calculate $\J$ mathematically since we do not need to numerically calculate $\na W$ which could amplify the noise in the data $B_z$. However, it should be pointed out that physically, it could be difficult to calculate $\J$ in such a way due to the existence of stray magnetic fields $\mathcal{H}$ in the measured magnetic field, i.e., the practical measured data in physical configuration are $B_z+\mathcal{H}$, see also \eqref{biot-savart law}. Such an unknown $\mathcal{H}$ can be considered as an artifact in $B_z$.
\end{remark}

\subsection{Uniqueness of 2D MREIT reconstructions using one injection current}

Based on Theorem \ref{SolvingJ}, now we prove the uniqueness of recovering $\sigma$ with one injection current $(\mE^{\pm}, g)$ for the two-dimensional MREIT model.
Due to the integrable singularity of $\na\Psi$, we define a constant
$$K:=\left\|\int_{\widetilde \Om} |\na\Psi({\bf r}'-\cdot)|\;d{\bf r}'\right\|_{C(\overline\Om)}<+\infty.$$ Furthermore, we introduce the admissible set
\begin{equation*}
\mathcal{A}[\epsilon_0,\sigma_\pm^0,\sigma_b]:=\{\sigma\in C^{1}(\overline\Om)~:~\sigma_-^0\leq \sigma\leq \sigma_+^0,~\|\na\ln\sigma\|_{C(\overline\Om)}\leq \epsilon_0,~\sigma|_{\Om\setminus \widetilde \Om} = \sigma_b\},
\end{equation*}
where $\epsilon_0(<\f{1}{4K})$ and $\sigma_\pm^0$ are known positive constants, and $\sigma_b$ is a known function satisfying
$$\sigma_b\in C^1(\overline{\Om\setminus \widetilde \Om}),\quad \sigma_-^0\leq \sigma_b \leq \sigma_+^0,\quad \|\na\ln\sigma_b\|_{C(\overline{\Om\setminus \widetilde \Om})}\leq \epsilon_0.$$

The following theorem states that if the conductivity distribution in the subregion $\Om\setminus \widetilde{\Om}$ is known (say, $\sigma_b$), we can uniquely reconstruct $\sigma$ from the measured data $B_z$ corresponding to one-injected current.
\begin{theorem}\label{thm:uniqueness}
For $\sigma,~\widehat\sigma \in \mathcal{A}[\epsilon_0,\sigma^0_\pm,\sigma_b]$, if $\Lambda_{\sigma}[g] = \Lambda_{\widehat\sigma}[g]$ for one injected current $(I,\mE^\pm)$ with $I>0$, then $\sigma = \widehat\sigma$ in $\Om$.
\end{theorem}

\begin{proof}
Since we assume the conductivity is known in $\Om\setminus \widetilde\Om$, it is enough to prove $\sigma = \widehat\sigma$ in $\widetilde\Om$. From Proposition 2.10 of \cite{Alessandrini2004}, there exists a positive constant $C_1>0$ depending only on $(\Om,\widetilde \Om,\mE^\pm,\sigma^0_\pm,\sigma_b,\ep_0)$ such that
\begin{equation}\label{infimum_nablau}
  \xi_{\widehat\sigma}: = \inf_{\widetilde \Om} |\na^\bot u[\widehat\sigma]|  \geq C_1 I(u[\widehat\sigma]|_{\mE^+}-u[\widehat\sigma]|_{\mE^-})^{1/2} = C_1 I(u[\widehat\sigma]|_{\mE_+})^{1/2}>0
\end{equation}
for $\widehat\sigma\in \mathcal{A}[\epsilon_0,\sigma^0_\pm,\sigma_b]$, with
$u[\widehat\sigma]$ solving \eref{eq:forwardPDE}. The last equality in \eref{infimum_nablau} comes from \eref{liu21}, while the last inequality in \eref{infimum_nablau} comes from the fact that the constant function $u[\widehat\sigma]|_{\mE^+}$ cannot be zero, otherwise $u[\widehat\sigma]|_{\mE^\pm} = 0$ together with $\frac{\p u[\widehat\sigma]}{\p{\bf n}}|_{\p\Om\setminus\mE^\pm} = 0$ will lead to $u[\widehat\sigma]\equiv 0$ in $\Om$, which contradicts the requirement that  $\int_{\mE^+}\widehat\sigma\f{\p u[\widehat\sigma]}{\p \n}dS = I>0$. Hence $|\J[\widehat\sigma]^\bot | := |\widehat\sigma \na^\bot u[\widehat\sigma]|>\sigma_-^0 \xi_{\widehat\sigma}>0$, and $\widehat\sigma$ satisfies
\begin{eqnarray}\label{eq-inverse}
    \begin{cases}
      \J[\widehat\sigma]^\bot \cdot \na \ln\widehat\sigma = -\f{1}{\mu_0} \Deltaxy B_z[\hat\sigma]  &\mbox{in }\widetilde \Om, \\
      \ln\widehat\sigma = \ln\sigma  &\mbox{on } \p\widetilde \Om
    \end{cases}
\end{eqnarray}
from \eref{liu15} and $\p_z^2B_z\equiv 0$ with positive $|\J[\widehat\sigma]^\bot|$ in $\widetilde\Om$.
The condition $\ln\widehat\sigma =\ln \sigma$ on $\partial\widetilde\Om$ in \eref{eq-inverse} comes from $\widehat\sigma =\sigma$ in $\Om\setminus \widetilde \Om$. By Theorem \ref{SolvingJ},  $\J[\widehat\sigma]$ is determined by $B_z[\widehat\sigma]$, so $B_z[\widehat\sigma]=\Lambda_{\widehat\sigma}[g]=\Lambda_{\sigma}[g]=B_z[\sigma]$ yields $\J[\widehat\sigma]=\J[\sigma]$. From \cite{Richter1981}, there exists a unique solution to \eref{eq-inverse} for known $B_z$ and $\bf J$. Therefore we conclude from \eref{eq-inverse} that $\ln\widehat\sigma \equiv \ln\sigma$ in $\widetilde\Om$. That is, $\widehat\sigma \equiv \sigma$ in $\overline\Om$, since we already have
$\widehat\sigma \equiv \sigma=\sigma_b$ in $\Om\setminus\widetilde \Om$. This completes the proof.
%
%
%
%
\end{proof}

\section{Single current harmonic $B_z$ algorithm for two-dimensional MREIT model, and convergence analysis}

To reconstruct the conductivity distribution from a single $B_z$ dataset, or to solve the boundary value problem \eref{eq-inverse} stably using the given $B_z$ data, we previously proposed the single current harmonic $B_z$ algorithm in \cite{SongSadleir2020}. This algorithm can be divided into three steps. The first step is to reconstruct a current flux density $\J$ from the measured $B_z$ data directly, and the second step is to solve $\ln\sigma$ in terms of the data pair $(B_z,\J)$ from \eref{eq-inverse}. Finally the conductivity distribution $\sigma$ can be reconstructed from $\sigma=\exp(\ln\sigma)$.  For our two-dimensional MREIT model, $\J$ can be exactly recovered from $B_z$ directly due to Theorem \ref{SolvingJ}.

Here we just consider the second step and assume that the current density $\J$ has been precisely recovered from given exact $B_z$ data. Since
\begin{equation}\label{eq:Deltau_nablalnsigma}
-\sigma\Delta u+\J^T\cdot \nabla\ln\sigma\equiv 0
\end{equation}
for $\J=(J_x,J_y)^T$ from
$-\nabla\cdot(\sigma \nabla u)\equiv\nabla\cdot \J\equiv 0$, equation \eref{linear hyperbolic0} together with the relation \eref{eq:Deltau_nablalnsigma} yields the vector identity
\begin{eqnarray}\label{liu52}
\left[
\J^\bot(\J^\bot)^T
+\J\J^T
\right]\nabla\ln\sigma\equiv
\sigma \Delta u \J
-\frac{1}{\mu_0}\Delta B_z\J^\bot
\end{eqnarray}
with $\J^\bot:=(J_y,-J_x)^T$. Since $
\J^\bot(\J^\bot)^T
+\J\J^T=|\J|^2 \mathbb{I}$, where $\mathbb{I}$ represents the $2\times 2$ identity matrix, the conductivity $\sigma$ solving \eref{liu52} can be approximately recovered using the iterative process
\begin{equation}\label{eq:updating}
  \na \ln\sigma^{n+1}
  =\frac{1}{|\J|^2}
  \left[
    \begin{array}{c}
      \sigma^n J_x \Delta u^n -\frac{1}{\mu_0 }J_y \Delta B_z \\
      \sigma^n J_y \Delta u^n +\frac{1}{\mu_0}J_x \Delta B_z \\
    \end{array}
  \right]
  \quad
  \mbox{in }\widetilde \Om
\end{equation}
for $n=0,1,2,\cdots$ with an initial guess $\sigma^0\in C^1(\overline \Om)$, where $u^n=u[\sigma^n]$ is the solution to \eref{eq:forwardPDE} with $\sigma$ replaced by $\sigma^n$. We can rewrite \eref{eq:updating} as
\begin{equation}\label{eq-9}
  \na \ln \sigma^{n+1} =
  \left[
    \begin{array}{cc}
      \sigma^n \Delta u^n &  -\frac{1}{\mu_0 } \Delta B_z \\
      \frac{1}{\mu_0} \Delta B_z  & \sigma^n \Delta u^n \\
    \end{array}
  \right]
  \f{\J}{|\J|^2}:=\mathbf{s}[\sigma^n, B_z]:= \mathbf{s}^n,
\end{equation}
noticing that $\J$ has been obtained from $B_z$.
Finally, reconstruction of the conductivity distribution by the single-current harmonic $B_z$ algorithm can be realized iteratively by solving the linear elliptic equation
\begin{equation}\label{algorithm:SCHBz}
  \left\{
    \begin{split}
      & \Delta \ln \sigma^{n+1} = \na\cdot \mathbf{s}^n ~\qquad \mbox{in } \widetilde \Om \\
      & \ln \sigma^{n+1} = \ln\sigma_b  \qquad \qquad \mbox{on }\p \widetilde\Om,
    \end{split}
  \right.
\end{equation}
with respect to $\ln\sigma^{n+1}$ in $\widetilde\Omega$
for known $\sigma^n$
from a specified initial guess $\sigma^0\in \mathcal{A}[\epsilon_0,\sigma^0_{\pm},\sigma_b]$. We then set $\sigma^{n+1} := \sigma_b$ in $\Om\setminus \widetilde \Om$ to yield $\sigma^{n+1}$ in $\Om$. The iteration stops at
$\|\ln\f{\sigma^{n+1}}{\sigma^n}\|\leq \epsilon$ for some specified tolerance $\epsilon>0$.

\begin{remark}
For the reconstruction algorithm, it is better to solve \eref{algorithm:SCHBz} from
\begin{eqnarray}\label{liu001}
    \begin{cases}
      \Delta \ln \sigma^{n+1} = \na\cdot \mathbf{t}^n-\nabla\cdot\mathbf{t}^*  &\mbox{in } \widetilde\Om \\
      \ln \sigma^{n+1} = \ln \sigma_b  &\mbox{on }\p \widetilde\Om,
    \end{cases}
\end{eqnarray}
where
$\mathbf{t}^n:=
\frac{1}{|\J|^2}
  \sigma^n\Delta u^n\J^T, \; \mathbf{t}^*:=
\frac{1}{|\J|^2}\frac{1}{\mu_0 }\Delta B_z(\J^\bot)^T$.
Then we obtain the sequence
\begin{eqnarray}\label{sigma_n+1}
\sigma^{n+1}(\bf r):=
\begin{cases}
\sigma^{n+1}(\bf r), &{\bf r}\in\widetilde\Om,\\
\sigma_b(\bf r), &{\bf r}\in \Om\setminus\overline{\widetilde\Om}
\end{cases}
\end{eqnarray}
in terms of \eref{algorithm:SCHBz}.
\end{remark}

To establish our main result showing convergence of the iterative process \eref{liu001}-\eref{sigma_n+1} using $B_z$ from a single injection current, we need the following lemmas.

\begin{lemma}\label{lemma:regularity_forward}
  Suppose that $v$ is the unique solution to the boundary value problem
  \begin{eqnarray}\label{eq:forwardPDE_general}
    \begin{cases}
      \na\cdot (\sigma \na v) = \sigma f, \qquad \mbox{in }\Om \\
      \int_{\mE^+}\sigma \frac{\p v}{\p \n }dS =
        I = -\int_{\mE^-}\sigma \frac{\p v}{\p \n }dS \\
      \na v\times \n|_{\mE^\pm} = {\bf 0} \\
      \left.\sigma \frac{\p v}{\p \n }\right|_{\p \Om \setminus \overline{\mE^+\cup \mE^-}} = 0 \\
      v|_{\mE^-} = 0
    \end{cases}
\end{eqnarray}
for $\sigma\in C^1(\overline{\Om})$ and a specified function $f(\r)$. Suppose $0<\sigma_-:=\inf_\Om \sigma \leq \sigma \leq \sigma_+:=\sup_\Om \sigma<+\infty $. Denote a fixed domain $\widetilde{\widetilde\Om}$ satisfying $\widetilde\Om \subset\subset \widetilde{\widetilde\Om}\subset\subset\Om$. Then the following estimates hold:
  \begin{enumerate}[i)]
    \item If $f\in L^2(\Om)$, there exists a constant $C_2$ depending only on $\Om$ such that
  \begin{equation}\label{estimate:H1}
    \|v\|_{H^1(\Om)} \leq C_2 \left(\frac{\sigma_+}{\sigma_-} \|f\|_{L^2(\Om)} + \f{I}{\sigma_- |\mE^+|^{1/2}}\right).
  \end{equation}
 Moreover, there exists a positive $C_3 = F_3(\|\na\ln \sigma\|_{C(\Om)})$ depending only on $\widetilde{\widetilde\Om}$ such that
  \begin{equation}\label{estimate:H2}
    \|v\|_{H^2(\widetilde{\widetilde \Om})} \leq C_3\|v\|_{H^1(\Om)}.
  \end{equation}

  \item If $f\in C(\Om)$, then $v\in C^{1,\alpha}(\widetilde{\widetilde{\Om}})$ for $\alpha\in (0,1)$ and there exists positive $C_4 = F_4(\|\na\ln \sigma\|_{C(\Om)})$ depending only on $\widetilde\Om$ and $\widetilde{\widetilde\Om}$ such that
  \begin{equation}\label{estimate:Holder}
     \|\na v\|_{C^{0,\alpha}(\widetilde\Om)} \leq C_4(\|v\|_{C^{0,\alpha}(\widetilde{\widetilde\Omega})} + \|f\|_{L^2(\Omega)}).
  \end{equation}
  \end{enumerate}
Here $F_i$ ($i=3,4$) are bounded functions with respect to the argument.
\end{lemma}

\begin{proof}
Noting that $v|_{\mE^-}=\frac{\p v}{\p{\bf n}}|_{\p\Om\setminus\mE^{\pm}}=0$, by multiplying $v$ on both sides of the equation in \eref{eq:forwardPDE_general}, it follows that
\begin{equation}\label{process_2}
  \int_\Om \sigma(\r)|\na v(\r)|^2 d\r = -\int_\Om \sigma(\r)f(\r)v(\r)d\r + \int_{\mE^+}v(\r)\sigma(\r)\f{\p v(\r)}{\p \n}dS.
\end{equation}
From the boundary conditions in \eref{eq:forwardPDE_general}, we obtain
\begin{equation*}
   \left|\int_{\mE^+}v(\r)\sigma(\r)\f{\p v(\r)}{\p \n}dS\right| = I \left|v|_{\mE^+}\right| = \f{I}{|\mE^+|}\left|\int_{\mE^+}v(\r)dS\right|\le \frac{I}{|\mE^+|^{1/2}}\|v\|_{L^2(\p\Om)},
\end{equation*}
noticing that $v|_{\mE^+}$ is a constant function. Therefore, we have from \eref{process_2} that
$$\|\na v\|^2_{L^2(\Om)}\le \frac{\sigma_+}{\sigma_-}\|f\|_{L^2(\Om)}\|v\|_{L^2(\Om)}+\frac{I}{\sigma_-|\mE^+|^{1/2}}\|v\|_{H^1(\Om)}$$
due to the trace theorem. On the other hand, since $v|_{\mE^-}=0$, the Poincar\'e inequality says that there exists a constant $C_p$ depending only on $\Om$ such that $\|v\|_{H^1(\Om)}\le C_p\|\na v\|_{L^2(\Om)}$. So the above estimate immediately leads to
\eref{estimate:H1}.

The estimates \eref{estimate:H2} and \eref{estimate:Holder} come from \eref{estimate:H1} and the interior regularity results for elliptic PDEs \cite{Gilbarg2001} (see also \cite{Liu2007} and \cite{Liu2010}).
The proof is complete.
\end{proof}

\begin{lemma}\label{lemma:estimate_u-uhat}
Suppose that $\sigma$ satisfies the condition in Lemma \ref{lemma:regularity_forward} and $\widehat \sigma\in C^1(\overline \Om)$ satisfies $0<\widehat\sigma_-:=\inf_\Om \widehat\sigma \leq \widehat\sigma \leq \widehat\sigma_+:=\sup_\Om \widehat\sigma<+\infty $.
Let $u$ and $\widehat u$ be the solutions of the equation \eqref{eq:forwardPDE} for the conductivities $\sigma$ and $\widehat \sigma$ respectively with the injection current $(I,\mE^\pm)$. Then there exists a positive function $C_5$ depending on $(\sigma,\widehat \sigma)$ such that
\begin{equation}\label{estimate:na(un-u*)Holder}
\|\na(u - \widehat u)\|_{C^{0,\alpha}(\widetilde \Om)} \leq C_5\left\|\na\ln\f{\sigma}{\widehat \sigma} \right\|_{C(\widetilde\Om)},
\end{equation}
where $C_5$ has the following form
\begin{equation*}
\begin{split}
  C_5 = &C_4(\|\na\ln\widehat \sigma\|_{C(\Om)})(C_sC_3(\|\na\ln\widehat \sigma\|_{C(\Om)})C_2\f{\widehat\sigma_+}{\widehat\sigma_-}|\Om|^{1/2}+1)\times\\
  &C_4(\|\na\ln\sigma\|_{C(\Om)})C_s
  C_3(\|\na\ln\sigma\|_{C(\Om)})C_2\f{I}{\sigma_-|\mE^+|^{1/2}}.
\end{split}
\end{equation*}
\end{lemma}

\begin{proof}
Using $\na\cdot(\widehat\sigma\na\widehat u)=\na\cdot(\sigma\na u)=0$, straightforward calculations verify that $u - \widehat u$ satisfies
\begin{eqnarray*}\label{eq:un-u*}
    \begin{cases}
      \na\cdot(\widehat\sigma \na(u- \widehat u)) = \widehat\sigma \na\ln \f{\widehat\sigma}{\sigma}\cdot \na u \qquad \mbox{in }\Om\\
      \widehat\sigma \na (u- \widehat u)\cdot \n|_{\p \Om \setminus \mE^\pm} = 0 \\
      \int_{\mE^\pm} \widehat \sigma \na (u- \widehat u)\cdot\n dS = 0 \\
      \na (u- \widehat u)\times \n |_{\mE^\pm} = {\bf 0} \\
      (u- \widehat u)|_{\mE^-} = 0.
    \end{cases}
\end{eqnarray*}
Applying Lemma \ref{lemma:regularity_forward} to this problem, we obtain
\begin{eqnarray}\label{liu53}
\|\na(\widehat u - u)\|_{C^{0,\alpha}(\widetilde \Om)} &\leq&
  C_4(\|\na\ln\widehat \sigma\|_{C(\Om)})\left(\|\widehat u - u\|_{C^{0,\alpha}(\widetilde{\widetilde\Om})}+\left\|\na\ln\f{\widehat \sigma}{\sigma}\cdot \na u\right\|_{C(\Omega)}\right)
  \nonumber\\&=&
  C_4(\|\na\ln\widehat \sigma\|_{C(\Om)})\left(\|\widehat u - u\|_{C^{0,\alpha}(\tilde{\tilde{\Om}})}+\left\|\na\ln\f{\widehat \sigma}{\sigma}\cdot \na u\right\|_{C(\widetilde\Omega)}\right),\qquad\;\qquad\;
\end{eqnarray}
where the last equality comes from the fact that $\sigma = \widehat\sigma$ in $\Omega\setminus \widetilde\Omega$.

Using the Sobolev embedding theorem and Lemma \ref{lemma:regularity_forward}, we have
\begin{eqnarray}\label{estimate:u-uhatHolder}
  \|\widehat u - u\|_{C^{0,\alpha}(\widetilde{\widetilde{\Om}})} &\leq& C_s \|\widehat u - u\|_{H^2(\widetilde{\widetilde{\Om}})} \nonumber
  \\& \leq& C_sC_3(\|\na\ln\widehat\sigma\|_{C(\Om)}) C_2\f{\widehat\sigma_+}{\widehat\sigma_-}\left\|\na\ln \f{\widehat \sigma}{\sigma}\cdot\na u\right\|_{L^2(\Om)} \nonumber
  \\&\le& C_s C_3(\|\na\ln\widehat\sigma\|_{C(\Om)})C_2\f{\widehat\sigma_+}
  {\widehat\sigma_-}|\Om|^{1/2}\left\|\na\ln \f{\widehat \sigma}{\sigma}\cdot\na u\right\|_{C({\widetilde \Om})}
\end{eqnarray}
since $\widehat\sigma=\sigma$ in $\Om\setminus\widetilde\Om$.
Hence \eref{liu53} and \eref{estimate:u-uhatHolder} generate
\begin{eqnarray}\label{estimate:na(u-uhat)Holder1}
& &\|\na(\widehat u - u)\|_{C^{0,\alpha}(\widetilde{\Om})} \leq C_4(\|\na\ln\widehat\sigma\|_{C(\Om)})\times
\nonumber\\
 & & \qquad \qquad \left(C_sC_3(\|\na\ln\widehat\sigma\|_{C(\Om)})C_2\f{\widehat\sigma_+}
  {\widehat\sigma_-}|\Om|^{1/2}+1\right)\left\|\na\ln\f{\widehat \sigma}{\sigma}\cdot \na u\right\|_{C(\widetilde\Omega)}.
\end{eqnarray}
On the other hand, by the Sobolev embedding theorem and \eref{estimate:H1}-\eref{estimate:Holder}, we obtain
\begin{eqnarray}\label{estimate:nauC}
  \|\na u\|_{C(\widetilde \Om)} &\leq& \|\na u\|_{C^{0,\alpha}(\widetilde \Om)} \nonumber \\
  &\leq& C_4(\|\na\ln\sigma\|_{C(\Om)})\|u\|_{C^{0,\alpha}(\widetilde{\widetilde \Om})} \nonumber\\&\leq&
  C_4(\|\na\ln\sigma\|_{C(\Om)})C_s\|u\|_{H^2(\widetilde{\widetilde \Om})}
  \nonumber\\&\leq& C_4(\|\na\ln\sigma\|_{C(\Om)})C_sC_3(\|\na\ln\sigma\|_{C(\Om)})\|u\|_{H^1(\Om)}
  \nonumber\\&\leq& C_4(\|\na\ln\sigma\|_{C(\Om)})C_sC_3(\|\na\ln\sigma\|_{C(\Om)})C_2\f{I}{\sigma_-|\mE^+|^{1/2}}.
\end{eqnarray}
Combining \eref{estimate:na(u-uhat)Holder1} and \eref{estimate:nauC}, we obtain  \eref{estimate:na(un-u*)Holder}.
\end{proof}

We will subsequently denote $\sigma^*$ as the exact conductivity to be reconstructed, $\J^*$ to be the corresponding current density, $u^* = u[\sigma^*]$ as the corresponding voltage potential and $B_z = B_z[\sigma^*]$ as the corresponding measurable $z$-component of magnetic flux density. The next result ensures the regularity of our iterative sequence.
\begin{lemma}\label{lemma:C1_sigma_n}
For the inversion input $B_z=B_z[\sigma^*]$ corresponding to the exact conductivity $\sigma^*\in \mathcal{A}[\epsilon_0,\sigma^0_{\pm},\sigma_b]$, if the initial value $\sigma^0\in \mathcal{A}[\epsilon_0,\sigma^0_{\pm},\sigma_b]$, the iterative sequence $\{\sigma^n: n=1,2,\cdots\}$ obtained by \eref{algorithm:SCHBz} and \eref{sigma_n+1} has the regularity $\sigma^n\in C(\overline{\Om})\bigcap C^1(\widetilde\Om)\bigcap C^1(\Om\setminus\overline{\widetilde\Om})$. Moreover, if $\sigma_b$ is a constant in $\Om\setminus\overline{\widetilde\Om}$,
$\sigma^n\in C^1(\overline{\Om})$.
\end{lemma}

\begin{proof}
For $\sigma^*\in \mathcal{A}[\epsilon_0,\sigma^0_{\pm},\sigma_b]\subset C^1(\overline{\Om})$, it follows from \eref{liu15} and Lemma \ref{lemma:regularity_forward} that $\Delta B_z[\sigma^*]\in C(\widetilde{\widetilde{\Om}})$ and then ${\bf J}\in C(\widetilde \Om)$ by Theorem \ref{SolvingJ}.

For the initial guess $\sigma^0\in \mathcal{A}[\epsilon_0,\sigma^0_{\pm},\sigma_b]\subset C^1(\overline{\Om})\subset C(\overline{\Om})\bigcap C^1(\widetilde\Om)\bigcap C^1(\Om\setminus\overline{\widetilde\Om})$, interior regularity for elliptic equations implies that $u^0=u[\sigma^0]\in C^{1,\alpha}(\overline{\widetilde{\Omega}})$ for $\alpha>0$ \cite{Gilbarg2001,Kim2002}. Then \eref{eq-9} and the identity $\sigma^1 \Delta u^1 \equiv \J[\sigma^1]\cdot\na\ln\sigma^1$ imply $\na\ln\sigma^1\in C(\overline{\widetilde\Om})$, i.e, $\ln\sigma^1\in C^1(\overline{\widetilde\Om})$. Since $\sigma^1\in C(\overline\Om)$ due to $\sigma^1=\sigma_b$ in $\p\widetilde\Om$, we know
that $\sigma^1\in C(\overline{\Om})\bigcap C^1(\widetilde\Om)$. On the other hand, since we define $\sigma^1=\sigma_b\in C^1(\Om\setminus\widetilde\Om)$ in $\Om\setminus\widetilde\Om$, we have proven that $\sigma^1\in C(\overline{\Om})\bigcap C^1(\widetilde\Om)\bigcap C^1(\Om\setminus\overline{\widetilde\Om})$.

Now, for $\sigma^1 \in C(\overline{\Om})\bigcap C^1(\widetilde\Om)\bigcap C^1(\Om\setminus\overline{\widetilde\Om})$, the same process in terms of \eref{eq-9}
ensures $\sigma^2\in C(\overline{\Om})\bigcap C^1(\widetilde\Om)\bigcap C^1(\Om\setminus\overline{\widetilde\Om})$. So induction arguments prove $\sigma^n\in C(\overline{\Om})\bigcap C^1(\widetilde\Om)$ for all $n=1,2,\cdots$.


Let us assume further that $\sigma_b$ is a constant in $\Om\setminus\widetilde \Om$. For the same initial value $\sigma^0\in \mathcal{A}[\epsilon_0,\sigma^0_{\pm},\sigma_b]$,  consider the iterative sequence $\{\widehat{\widetilde\sigma}^n: n=1,2,\cdots\}$ defined by
\begin{eqnarray}
  \na \ln \widehat{\widetilde\sigma}^{n+1}&=&
  \left[
    \begin{array}{cc}
      \widehat{\widetilde\sigma}^n \Delta u[\widehat{\widetilde\sigma}^n] &  -\frac{1}{\mu_0 } \Delta B_z \\
      \frac{1}{\mu_0} \Delta B_z  & \widehat{\widetilde\sigma}^n \Delta u[\widehat{\widetilde\sigma}^n] \\
    \end{array}
  \right]
  \f{\J^*}{|\J^*|^2}\nonumber\\
  &\equiv&\left[
    \begin{array}{cc}
      -\na\widehat{\widetilde\sigma}^n\cdot \na u[\widehat{\widetilde\sigma}^n] &  -{(\J^*)}^\bot\cdot\na\ln\sigma^*\\
      {(\J^*)}^\bot\cdot\na\ln\sigma^*  &-\na\widehat{\widetilde\sigma}^n\cdot \na u[\widehat{\widetilde\sigma}^n]  \\
    \end{array}
  \right]
  \f{\J^*}{|\J^*|^2}\quad \mbox{ in }\widehat{\widetilde\Om},\label{liu61}\\
  \ln \widehat{\widetilde\sigma}^{n+1}&=&\ln\sigma_b, \quad \mbox{ on }\partial\widehat{\widetilde\Om}\label{liu62}
\end{eqnarray}
in a larger domain $\widehat{\widetilde\Om}$ satisfying
$\Om\supset\widetilde{\widetilde\Om}\supset\widehat{\widetilde\Om}\supset\widetilde\Om$, and
\begin{eqnarray*}
\widehat{\widetilde\sigma}^{n+1}({\bf r}):=
\begin{cases}
\widehat{\widetilde\sigma}^{n+1}({\bf r}), &{\bf r}\in \widehat{\widetilde\Om},\\
\sigma_b, &{\bf r}\in\Om\setminus \widehat{\widetilde\Om}.
\end{cases}
\end{eqnarray*}
Using the same arguments for $\sigma^n$ in $\widetilde\Om$, we know that $\widehat{\widetilde\sigma}^{n}\in C^1(\widehat{\widetilde\Om})$.
For $\widehat{\widetilde\sigma}^0=\sigma^0\in \mathcal{A}[\epsilon_0,\sigma^0_{\pm},\sigma_b]$, since $\widehat{\widetilde\sigma}^0\equiv\sigma^*\equiv\sigma_b (constant)$ in $\widehat{\widetilde\Om}\setminus\widetilde\Om$, we have $\na \ln \widehat{\widetilde\sigma}^{1}\equiv 0$ in $\widehat{\widetilde\Om}\setminus\widetilde\Om$ from \eref{liu61}. Therefore the boundary condition \eref{liu62} yields $\widehat{\widetilde\sigma}^{1}\equiv \sigma_b$ in $\widehat{\widetilde\Om}\setminus\widetilde\Om$.
So, $\widehat{\widetilde\sigma}^{1}$ satisfies the same equation in $\widetilde\Om$ and has the same boundary value $\sigma_b$ on $\p\widetilde\Om$ as $\sigma^{1}$. Therefore $\widehat{\widetilde\sigma}^{1}\equiv \sigma^{1}$ in $\overline{\widetilde\Om}$. On the other hand, we have $\widehat{\widetilde\sigma}^{1}\equiv \sigma_b\equiv \sigma^{1}$ in $\widehat{\widetilde\Om}\setminus\widetilde\Om$, so we finally have
$\widehat{\widetilde\sigma}^{1}\equiv \sigma^{1}$ in $\widehat{\widetilde\Om}$. By the regularity of $\widehat{\widetilde\sigma}^{1}$, we know that $\sigma^{1}\in C^1(\widehat{\widetilde\Om})$.

Now, by induction arguments, we know $\sigma^{n}\in C^1(\widehat{\widetilde\Om})$. Noticing
$\sigma^{n}\equiv \sigma_b $(constant) in $\Om\setminus\widetilde\Om$, we finally have $\sigma^n\in C^1(\overline{\Om})$.
\end{proof}

From Lemma \ref{lemma:C1_sigma_n}, if $\sigma_b$ is a constant, $\sigma^n$ is of $C^1$ smoothness in $\overline{\Om}$. In the following, we will only consider this special case. Moreover, under the assumption that $\sigma_b$ is a constant we can further derive that $|\na u[\sigma]|$ has a uniform positive lower bound $\xi_0>0$ in $\widetilde \Om$ for $\sigma\in \mathcal{A}[\ep_0,\sigma^0_\pm,\sigma_b]$ if $\ep_0$ is smaller than a given constant. This result is stated in the following Corollary \ref{corollary:uniform_lower_bound}.

\begin{corollary}\label{corollary:uniform_lower_bound}
Let $\sigma_b$ be a positive constant. Then there exists a positive constant $\eta_0\in (0,\f{1}{4K})$ such that, if $\epsilon_0<\eta_0$, the solution $u[\sigma]$ to \eref{eq:forwardPDE} has the estimate
\begin{equation}\label{estimate:lowergradient}
  \inf_{\widetilde \Om} |\na u[\sigma]| \geq \xi_0 >0
\end{equation}
for $\sigma\in \mathcal{A}[\ep_0,\sigma^0_\pm,\sigma_b]$ uniformly, where the constant $\xi_0$ depends only on $\eta_0,\widetilde \Om,\mE^\pm,\sigma^0_\pm$ and $\sigma_b$, but independent of $\sigma$ itself.
\end{corollary}

\begin{proof}
From \eref{infimum_nablau} we obtain
\begin{equation*}
  \inf_{\widetilde \Om}|\na u[\sigma_b]| \geq \xi_{\sigma_b}:= C_1 I(u[\sigma_b]|_{\mE^+})^{1/2}.
\end{equation*}
From Lemma \ref{lemma:estimate_u-uhat} we obtain the estimate
\begin{equation}\label{estimate:nau-nau_b}
  \|\na u[\sigma] - \na u[\sigma_b]\|_{C(\widetilde \Om)} \leq \|\na u[\sigma] - \na u[\sigma_b]\|_{C^{0,\alpha}(\widetilde \Om)} \leq C_b \|\na\ln\sigma\|_{C(\Om)}.
\end{equation}
Here $C_b$ is a constant defined as $$C_b=C_4(0)(C_sC_3(0)C_2\f{\sigma^0_+}{\sigma^0_-}|\Om|^{1/2}+1)
\overline{C_4}C_s\overline{C_3}\f{I C_2}{\sigma_-^0|\mE^+|^{1/2}},$$
where $\overline{C_i} = \sup\limits_{s\in (0,\f{1}{4K})}F_i(s)$ for $i=3,4$.
Hence, for $\r\in \widetilde \Om$ we have
\begin{equation*}
  |\na u[\sigma](\r)| \geq |\na u[\sigma_b](\r)| - \|\na u[\sigma_b] - \na u[\sigma]\|_{C(\Om)} \geq \xi_{\sigma_b} - C_b\|\na \ln\sigma\|_{C(\Om)},
\end{equation*}
where $\xi_{\sigma_b} = \inf_{\widetilde \Om}|\na u[\sigma_b]|$.
Let $\eta_0 = \min\{\f{1}{4K},\f{1}{2C_b}\xi_{\sigma_b}\}$ and $\xi_0 = \f{1}{2}\xi_{\sigma_b}$, then if $\ep_0 <\eta_0$ we have
\begin{equation*}
  |\na u[\sigma]| \geq \xi_0\quad  \hbox{in }\widetilde \Omega
\end{equation*}
for $\sigma\in \mathcal{A}[\ep_0,\sigma^0_\pm,\sigma_b]$ uniformly. Hence we have the estimate \eref{estimate:lowergradient}.
This completes the proof.
\end{proof}

Now it is possible to present our main result, the convergence property of the iteration process \eref{algorithm:SCHBz}.

\begin{theorem}\label{theorem:convergence}
Assume that $\sigma^*\in \mathcal{A}[\epsilon_0,\sigma^0_\pm,\sigma_b]$ with $\sigma_b$ a positive constant in $\Om\setminus\widetilde\Om$, while $\{\sigma^n: n=1,2,\cdots\}$ is the sequence generated by \eref{algorithm:SCHBz} and \eref{sigma_n+1} for inversion input $B_z^*$ with constant initial value $\sigma^0 \equiv \sigma_b$ in $\Omega$. Then there exists a constant $\eta\in (0,\eta_0)$ such that for $\ep_0<\eta$, it holds that
  \begin{equation}\label{estimate:main}
  \left\|\ln\f{\sigma^{n}}{\sigma^*}\right\|_{C^{1}(\widetilde \Om)} \leq (K+1)\epsilon_0\theta^{n-1},\quad n=1,2,\cdots,
\end{equation}
where the constant $\theta \in (\f{\sqrt{2}}{2},1)$ depends on $\ep_0$.
\end{theorem}

\begin{proof}
By Lemma 3.3, we know $\sigma^n \in C^1(\overline\Om)$ for $n=1,2,\cdots$. Since
$\frac{1}{\mu_0}\Delta B_z[\sigma^*]\equiv {\J}^\bot\cdot\na\ln\sigma^*$,
the iteration scheme \eref{eq-9} is
\begin{equation*}
  \na \ln \sigma^{n+1} =
  \left[
    \begin{array}{cc}
      \sigma^n \Delta u^n &  -({\J^*})^\bot\cdot\na\ln\sigma^* \\
      ({\J^*})^\bot\cdot\na\ln\sigma^*  & \sigma^n \Delta u^n \\
    \end{array}
  \right]
  \f{\J^*}{|\J^*|^2}.
\end{equation*}
Hence
\begin{equation}\label{liu54}
 \na\ln\f{\sigma^{n+1}}{\sigma^*} = \f{1}{\sigma^*|\na u^*|}(\sigma^n \Delta u^n -\sigma^* \Delta u^*) \f{\J^*}{|\J^*|} \qquad \mbox{in } \widetilde\Om.
\end{equation}
By decomposing $\na\ln\f{\sigma^{n+1}}{\sigma^*}$ in $\widetilde \Om$ along the directions $\J^*$ and $(\J^*)^\bot$, we obtain
\begin{equation*}\label{inequ:pre_1}
  \left\|\na\ln\f{\sigma^{n+1}}{\sigma^*}\right\|_{C(\widetilde \Om)} = \left\|\f{1}{\sqrt{2}|\J^*|}(\J^*+(\J^*)^\bot)\cdot\na\ln\f{\sigma^{n+1}}{\sigma^*}  \right\|_{C(\widetilde \Om)}.
\end{equation*}
Since $\sigma^n \Delta u^n = -\na u^n \cdot \na \sigma^n$, we obtain from \eref{liu54} that
\begin{equation}\label{inequ:pre_2}
  \left\|\na\ln\f{\sigma^{n+1}}{\sigma^*}\right\|_{C(\widetilde \Om)} =\left\|\f{1}{\sqrt{2}\sigma^*|\na u^*|}\left(\na u^n \cdot \na \sigma^n - \na u^* \cdot \na \sigma^*\right) \right\|_{C(\widetilde \Om)}.
\end{equation}

Applying the decomposition
\begin{equation*}\label{eq:pre-1}
  \na u^n \cdot \na \sigma^n - \na u^* \cdot \na \sigma^* = \na \sigma^*\cdot \na(u^n-u^*)+ \na u^n\cdot \na(\sigma^n-\sigma^*),
\end{equation*}
\eref{inequ:pre_2} generates
$$
  \left\|\na\ln\f{\sigma^{n+1}}{\sigma^*}\right\|_{C(\widetilde \Om)} \leq \f{1}{\sqrt{2}\xi_0}\|\na\ln \sigma^*\cdot \na (u^n-u^*)\|_{C(\widetilde \Om)} + \f{1}{\sqrt{2}}\left\|\f{\na u^n}{\sigma^* |\na u^*|}\cdot \na (\sigma^n -\sigma^*)\right\|_{C(\widetilde \Om)}.
$$
We further decompose the second term in the right hand side to generate
\begin{equation*}\label{inequ-20}
\left\|\na\ln\f{\sigma^{n+1}}{\sigma^*}\right\|_{C(\widetilde \Om)} \leq \mathcal{I}_1^n + \mathcal{I}_2^n + \mathcal{I}_3^n,  \qquad n=0,1,2,\cdots,
\end{equation*}
where
\begin{eqnarray}\label{def:I1-I3}
   \mathcal{I}_1^n &:=& \f{1}{\sqrt{2}\xi_0}\|\na\ln \sigma^*\|_{C(\widetilde \Om)}\| \na (u^n-u^*)\|_{C(\widetilde \Om)},  \label{I_1n}\\
  \mathcal{I}_2^n &:=& \f{1}{\sqrt{2}\sigma_-^*}\left\|\na (\sigma^n -\sigma^*)\right\|_{C(\widetilde \Om)},  \label{I_2n} \\
  \mathcal{I}_3^n &:=& \f{1}{\sqrt{2}\sigma_-^*\xi_0}\|\na(u^n - u^*)\|_{C(\widetilde \Om)} \|\na(\sigma^n-\sigma^*)\|_{C(\widetilde \Om)}, \label{I_3n}
\end{eqnarray}
with $\sigma_-^*:= \inf_\Om \sigma^*\ge \sigma_-^0>0$. We also define $\sigma_+^* = \sup_\Om \sigma^*\le \sigma_+^0$ for later use.

Define $e^n({\bf r}):=\sigma^n({\bf r}) -\sigma^*({\bf r})$ in $\Om$. Since $e_n|_{\Om\setminus \widetilde \Om}\equiv 0$ from $\sigma^n({\bf r})=\sigma^*({\bf r})=\sigma_b({\bf r})$, we obtain for ${\bf r}\in\widetilde \Om\subset \mathbb{R}^2$ in
$$
e_n({\bf r})\equiv \int_{\widetilde \Om} \Delta\Psi({\bf r}'-{\bf r})e_n({\bf r}')\;d{\bf r}'=-
\int_{\widetilde \Om} \na\Psi({\bf r}'-{\bf r})\cdot\na e_n({\bf r}')\;d{\bf r}'
$$
that
\begin{equation*}\label{estimate:e_n}
|e_n({\bf r})|\le \|\na e_n\|_{C({\widetilde \Om})}\int_{\widetilde \Om} |\na\Psi({\bf r}'-{\bf r})|\;d{\bf r}'\le K\|\na e_n\|_{C({\widetilde \Om})}.
\end{equation*}
Hence, we obtain
\begin{eqnarray}\label{relation:sigma_logsigma}
  \|e_n\|_{C(\widetilde \Om)} \leq K \|\na e_n\|_{C(\widetilde \Om)}
\end{eqnarray}
and
\begin{equation*}
  \sigma^n \leq \sigma^* + K\|\na e^n\|_{C(\Om)}.
\end{equation*}

Similarly to the above estimates, we have
\begin{equation}\label{sigma*=int}
  \sigma^*(\r) = \int_\Om \Delta\Psi(\r'-\r)\sigma^*(\r')d\r' = \sigma_b - \int_\Om \na\Psi(\r'-\r)\cdot \na\sigma^*(\r')d\r',\quad \r\in \Om.
\end{equation}
Hence, we have
\begin{equation*}\label{sigma*+control_1}
  \sigma^*_+ \leq \sigma_b + K\|\na \sigma^*\|_{C(\Om)} \leq \sigma_b + K\sigma_+^*\|\na\ln\sigma^*\|_{C(\Om)}.
\end{equation*}
From the definition of  $\mathcal{A}[\ep_0,\sigma_\pm^0,\sigma_b]$ and the fact that $\sigma^*\in \mathcal{A}[\ep_0,\sigma_\pm^0,\sigma_b]$, we obtain
\begin{equation}\label{sigma*+control_2}
  \sigma_+^* \leq \frac{\sigma_b}{1-K\|\na \ln\sigma^*\|_{C(\Om)}}.
\end{equation}
From \eref{sigma*=int} we also obtain
\begin{equation}\label{sigma*-control_1}
  \sigma_-^* \geq \sigma_b - K\|\na\sigma^*\|_{C(\Om)} \geq \sigma_b - K\sigma_+^*\|\na\ln\sigma^*\|_{C(\Om)}.
\end{equation}
Substituting \eref{sigma*+control_2} into \eref{sigma*-control_1} we obtain
\begin{equation}\label{sigma*-control_2}
  \sigma_-^* \geq \sigma_b -K\f{\sigma_b \|\na\ln\sigma^*\|_{C(\Om)}}{1-K\|\na\ln \sigma^*\|_{C(\Om)}} \geq \left(1-2K\|\na\ln\sigma^*\|_{C(\Om)}\right)\sigma_b.
\end{equation}
From \eref{sigma*+control_2} and \eref{sigma*-control_2} we obtain that
\begin{equation}\label{fraction_sigma_pm}
  1\leq \f{\sigma^*_+}{\sigma^*_-} \leq \f{1}{[1-K\|\na\ln\sigma^*\|_{C(\Om)}][1-2K\|\na\ln\sigma^*\|_{C(\Om)}]}.
\end{equation}

On the other hand, from the identity
\begin{equation}\label{relation:nasigma_n-sigma}
   \na e_n=\sigma^n\na\ln\sigma^n - \sigma^*\na\ln\sigma^*\\
   =\sigma^* \na\ln\f{\sigma^n}{\sigma^*}+e_n\na\ln\sigma^n,
\end{equation}
we obtain
\begin{equation*}\label{ineq-22}
\|\na e_n\|_{C(\widetilde\Om)}
\leq \|\sigma^*\|_{C(\widetilde\Om)} \left\|\na\ln\f{\sigma^n}{\sigma^*}\right\|_{C(\widetilde\Om)} + \|\na\ln\sigma^n\|_{C(\widetilde\Om)}\|e_n\|_{C(\widetilde\Om)},
\end{equation*}
which leads by \eref{relation:sigma_logsigma} to
$$
\|\na e_n\|_{C(\widetilde\Om)}\leq \|\sigma^*\|_{C(\widetilde\Om)} \left\|\na\ln\f{\sigma^n}{\sigma^*}\right\|_{C(\widetilde\Om)} + K \|\na \ln\sigma^n\|_{C(\widetilde \Om)}\|\na e_n\|_{C(\widetilde \Om)},
$$
that is,
\begin{equation}\label{ineq-50}
  (1-K\|\na\ln\sigma^n\|_{C(\widetilde \Om)})\|\na e_n\|_{C(\widetilde \Om)} \leq \|\sigma^*\|_{C(\widetilde \Om)} \left\|\na\ln\f{\sigma^n}{\sigma^*}\right\|_{C(\widetilde \Om)}.
\end{equation}

Next we will estimate $\mathcal{I}_i^n$ ($i=1,2,3$) using the method of mathematical induction.
Defining $\sigma^n_- = \inf_{\widetilde\Om}\sigma^n$ and $\sigma^n_+ = \sup_{\widetilde\Om}\sigma^n$, from Lemma \ref{lemma:estimate_u-uhat} we have
\begin{equation*}\label{estimate:un-u*}
  \|\na(u^n -u^*)\|_{C(\widetilde \Om)} \leq C_6^n\left\|\na\ln\f{\sigma^n}{\sigma^*}\right\|_{C(\widetilde \Om)} \quad \mbox{for } n=1,2,\cdots,
\end{equation*}
where $C_6^n~~(n=1,2,\cdots)$ is defined as
\begin{equation*}\label{def:C6n}
\begin{split}
  C_6^n = & C_4(\|\na\ln \sigma^n\|_{C(\Om)})(C_sC_3(\|\na\ln\sigma^n\|_{C(\Om)})
  C_2\f{\sigma_+^n}{\sigma_-^n}|\Om|^{1/2}+1)\times \\
  &\qquad C_4(\|\na\ln\sigma^*\|_{C(\Om)})C_sC_3(\|\na\ln\sigma^*\|_{C(\Om)})C_2\f{I}{\sigma_-^0|\mE^+|^{1/2}}.
\end{split}
\end{equation*}

Step 1. For $n=0$, we have $\|\na \ln \sigma^0\| =0$ and $\sup_\Om\sigma^0= \inf_\Om\sigma^0 = \sigma^0$. Hence from Lemma \ref{lemma:estimate_u-uhat} and \eqref{estimate:nau-nau_b} we obtain
\begin{equation*}\label{estimate:I_10}
  \|\mathcal{I}_1^0\|_{C(\widetilde \Om)} \leq \f{ C_b}{\sqrt{2}\xi_0}\|\na\ln\sigma^*\|_{C(\Om)}\left\|\na\ln\f{\sigma^{0}}{\sigma^*}\right\|_{C(\widetilde \Om)}.
\end{equation*}

Next we will estimate $\mathcal{I}_2^0$. From \eref{ineq-50} we obtain
\begin{equation*}
  \|\na e_0\|\leq \|\sigma^*\|_{C(\widetilde \Om)}\left\|\na\ln\f{\sigma^0}{\sigma^*}\right\|_{C(\widetilde \Om)}.
\end{equation*}
Hence from the definition of $\mathcal{I}_2^0$ and \eref{relation:nasigma_n-sigma} we obtain
\begin{equation}\label{estimate_I2_01}
  \|\mathcal{I}_2^0\|_{C(\widetilde \Om)} \leq \f{\sigma_+^*}{\sqrt{2}\sigma_-^*}\left\|\na\ln\f{\sigma^0}{\sigma^*}\right\|_{C(\widetilde \Om)}.
\end{equation}
Substituting \eref{fraction_sigma_pm} into \eref{estimate_I2_01} we obtain
\begin{equation*}
  \|\mathcal{I}_2^0\|_{C(\widetilde \Om)} \leq \f{1}{\sqrt{2}} \f{1}{[1-K\|\na\ln\sigma^*\|_{C(\Om)}][1-2K\|\na\ln\sigma^*\|_{C(\Om)}]}\left\|\na\ln\f{\sigma^0}{\sigma^*}\right\|_{C(\widetilde \Om)}.
\end{equation*}

Next we will estimate $\mathcal{I}_3^0$. From Lemma \ref{lemma:estimate_u-uhat} and \eref{fraction_sigma_pm} we obtain
\begin{equation*}
\begin{split}
  \|\mathcal{I}_3^0\|_{C(\widetilde \Om)} &\leq C_b\f{\sigma^*_+}{\sqrt{2}\sigma^*_-\xi_0} \left\|\na\ln\f{\sigma^0}{\sigma^*}\right\|_{C(\widetilde\Om)}^2 \\
  & = C_b\f{\sigma^*_+}{\sqrt{2}\sigma^*_-\xi_0}
  (\left\|\na\ln\sigma^0\right\|_{C(\widetilde\Om)}+\left\|\na\ln \sigma^*\right\|_{C(\widetilde\Om)}) \left\|\na\ln\f{\sigma^0}{\sigma^*}\right\|_{C(\widetilde\Om)} \\
  & \leq \f{C_b }{\sqrt{2}\xi_0}\f{\left\|\na\ln \sigma^*\right\|_{C(\widetilde\Om)}}
  {[1-K\|\na\ln\sigma^*\|_{C(\Om)}][1-2K\|\na\ln\sigma^*\|_{C(\Om)}]} \left\|\na\ln\f{\sigma^0}{\sigma^*}\right\|_{C(\widetilde\Om)}
\end{split}
\end{equation*}
in which the last equality comes from the estimate \eref{fraction_sigma_pm} and the fact that $\sigma^0$ is a constant.

Let $G_0$ be a constant depending on $\|\na\ln\sigma^*\|_{C(\Om)}$ defined as
\begin{equation*}
\begin{split}
& G_0(\|\na\ln\sigma^*\|_{C(\Om)}) := \f{1}{\sqrt{2}} \f{1}{[1-K\|\na\ln\sigma^*\|_{C(\Om)}][1-2K\|\na\ln\sigma^*\|_{C(\Om)}]} + \\
&\qquad  \f{ C_b}{\sqrt{2}\xi_0}\left(1+\f{1}{[1-K\|\na\ln\sigma^*\|_{C(\Om)}][1-2K\|\na\ln\sigma^*\|_{C(\Om)}]}\right)\|\na\ln\sigma^*\|_{C(\Om)}.
\end{split}
\end{equation*}
Since $G_0$ is strictly monotonically increasing with respect to $\|\na\ln \sigma^*\|_{C(\Om)}$ and  $G_0\to \f{\sqrt{2}}{2}$ as $\|\na\ln \sigma^*\|_{C(\Om)}\to 0$,
there exists $\eta_1 \in (0,\eta_0)$ such that $\theta_0:=G(\ep_0)<1$ for $\epsilon_0\in (0,\eta_1)$. Hence, from the monotonicity property of $G_0$,  we have for $\ep_0<\eta_1$ that
\begin{equation}\label{estimate:step1}
  \left\|\na\ln\f{\sigma^1}{\sigma^*}\right\|_{C(\Om)} \leq \theta_0 \left\|\na\ln \f{\sigma^0}{\sigma^*}\right\|_{C(\Om)} = \theta_0 \|\na\ln \sigma^*\|_{C(\Om)}\le \theta_0\epsilon_0< \epsilon_0.
\end{equation}

Step 2. For $n=1$,  we have from \eref{estimate:step1} that
\begin{equation}\label{estimate:nablalnsigma1}
  \|\na\ln\sigma^1\|_{C(\Om)} \leq (\theta_0 +1)\|\na\ln \sigma^*\|_{C(\Om)} <2\|\na\ln \sigma^*\|_{C(\Om)}.
\end{equation}
Then similarly to \eref{fraction_sigma_pm}, we have
\begin{equation*}\label{estimate:fraction_sigma1+/-}
  1\leq \f{\sigma^1_+}{\sigma^1_-} \leq \f{1}{[1-2K\|\na\ln\sigma^*\|_{C(\Om)}][1-4K\|\na\ln\sigma^*\|_{C(\Om)}]}.
\end{equation*}

We will estimate $\mathcal{I}_1^1$ first. From Lemma \ref{lemma:estimate_u-uhat},  we have for any $\alpha\in (0,1)$ that
\begin{equation}\label{estimate:nau1-u*}
  \|\na (u^1 - u^*)\|_{C(\Om)} \leq \|\na (u^1 - u^*)\|_{C^{0,\alpha}(\Om)}\leq C_7\left\|\na\ln \f{\sigma^1}{\sigma^*}\right\|_{C(\Om)}.
\end{equation}
Here, the constant $C_7$ is defined as
\begin{equation*}
\begin{split}
  C_7 = & \overline{\overline{C_4}}\left(C_s\overline{\overline{C_3}}
  C_2\f{1}{[1-2K\|\na\ln\sigma^*\|_{C(\Om)}][1-4K\|\na\ln\sigma^*\|_{C(\Om)}]}
  |\Om|^{1/2}+1\right)\times \\
  &\quad C_4(\|\na\ln\sigma^*\|_{C(\Om)})C_sC_3(\|\na\ln\sigma^*\|_{C(\Om)})
  C_2\f{I}{\sigma_-^0|\mE^+|^{1/2}},
\end{split}
\end{equation*}
where $\overline{\overline{C_i}} = \sup\limits_{s\in (0,\f{1}{2K})}F_i(s)$ for $i=3,4$.
Then from \eref{I_1n} and \eref{estimate:nau1-u*}, we obtain
\begin{equation}\label{estimate_I1_1}
  \|\mathcal{I}_1^1\|_{C(\widetilde \Om)}  \leq \f{1}{\sqrt{2}\xi_0}\|\na\ln\sigma^*\|_{C(\Om)} C_7 \left\|\na\ln \f{\sigma^1}{\sigma^*}\right\|_{C(\Om)}.
\end{equation}

%
%

For $\mathcal{I}_2^1$, from \eref{ineq-50} and \eref{estimate:nablalnsigma1} we obtain
\begin{equation}\label{estimate:nabla_e1}
   \|\na e_1\|_{C(\widetilde \Om)} \leq \f{\|\sigma^*\|_{C(\widetilde \Om)}}{1-2K\|\na\ln\sigma^*\|_{C(\widetilde \Om)}} \left\|\na\ln\f{\sigma^1}{\sigma^*}\right\|_{C(\widetilde \Om)}.
\end{equation}
Substituting \eref{fraction_sigma_pm}, \eref{ineq-50} and \eref{estimate:nabla_e1} into \eref{I_2n}, we obtain
\begin{equation}\label{estimate_I2_1}
  \|\mathcal{I}_2^1\|_{C(\widetilde \Om)}  \leq \f{1}{\sqrt{2}[1-K\|\na\ln\sigma^*\|_{C(\Om)}][1-2K\|\na\ln\sigma^*\|_{C(\Om)}]^2}\left\|\na\ln\f{\sigma^1}{\sigma^*}\right\|_{C(\widetilde \Om)}.
\end{equation}

Next we will estimate $\mathcal{I}_3^1$. Substituting \eref{estimate:nau1-u*} and \eref{estimate:nabla_e1} into \eref{I_3n} we obtain
\begin{equation*}
\begin{split}
  \|\mathcal{I}_3^1\|_{C(\widetilde \Om)}  &\leq \f{C_7 }{\sqrt{2}\xi_0[1-K\|\na\ln\sigma^*\|_{C(\Om)}][1-2K\|\na\ln\sigma^*\|_{C(\Om)}]^2} \left\|\na\ln\f{\sigma^1}{\sigma^*}\right\|^2_{C(\widetilde\Om)}\\
  & \leq \f{C_7 (\|\na\ln\sigma^1\|_{C(\Om)}+\|\na\ln\sigma^*\|_{C(\Om)})}{\sqrt{2}\xi_0[1-K\|\na\ln\sigma^*\|_{C(\Om)}][1-2K\|\na\ln\sigma^*\|_{C(\Om)}]^2} \left\|\na\ln\f{\sigma^1}{\sigma^*}\right\|_{C(\widetilde\Om)}.
\end{split}
\end{equation*}
From \eref{estimate:nablalnsigma1} we obtain
\begin{equation}\label{estimate_I3_1}
  \|\mathcal{I}_3^1\|_{C(\widetilde \Om)} \leq \f{3C_7 \|\na\ln\sigma^*\|_{C(\Om)}}{\sqrt{2}\xi_0[1-K\|\na\ln\sigma^*\|_{C(\Om)}][1-2K\|\na\ln\sigma^*\|_{C(\Om)}]^2} \left\|\na\ln\f{\sigma^1}{\sigma^*}\right\|_{C(\widetilde\Om)}.
\end{equation}

Let $G_1$ be a constant depending on $\|\na\ln\sigma^*\|_{C(\Om)}$ defined as
\begin{equation*}\label{def:theta}
\begin{split}
 G_1(\|\na\ln\sigma^*\|_{C(\Om)}) & = \f{C_7 }{\sqrt{2}\xi_0}\left(1+\f{3}
 {[1-K\|\na\ln\sigma^*\|_{C(\Om)}][1-2K\|\na\ln\sigma^*\|_{C(\Om)}]^2}\right)\times \\
 &\quad \|\na\ln\sigma^*\|_{C(\Om)}
+\f{1}{\sqrt{2}[1-K\|\na\ln\sigma^*\|_{C(\Om)}][1-2K\|\na\ln\sigma^*\|_{C(\Om)}]^2}.
\end{split}
\end{equation*}
Similarly to $G_0(\|\na\ln\sigma\|_{C(\Om)})$, $G_1(\|\na\ln\sigma\|_{C(\Om)})$ is also strictly monotonically increasing with respect to $\|\na\ln \sigma^*\|_{C(\Om)}$ and  $G_1\to \f{\sqrt{2}}{2}$ as $\|\na\ln \sigma^*\|_{C(\Om)}\to 0$.
So there exists $\eta\in (0,\eta_1)$, such that if $\ep_0<\eta$, $\theta:=G_1(\ep_0)\in (\f{\sqrt{2}}{2},1)$. From the monotonicity property of $G_1$, we obtain that
\begin{equation*}\label{estimate:lnsigma2}
  \left\|\na\ln\f{\sigma^2}{\sigma^*}\right\|_{C(\Om)} \leq \theta \left\|\na\ln \f{\sigma^1}{\sigma^*}\right\|_{C(\Om)} \leq \theta\theta_0 \|\na\ln \sigma^*\|_{C(\Om)}\le \epsilon_0\theta.
\end{equation*}

Step 3.
Suppose for $k \leq n$ and $\|\na\ln\sigma^*\|_{C(\Om)}\in [0,\f{1}{4K})$, we have
the estimate
\begin{equation}\label{assumption}
  \left\|\na\ln\f{\sigma^k}{\sigma^*}\right\|_{C(\Om)} \leq \theta^{k-1} \theta_0 \|\na\ln \sigma^*\|_{C(\Om)}\le \epsilon_0\theta^{k-1}
\end{equation}
for $\theta$ defined in Step 2.
We will next estimate for the case that $k+1$.

From \eref{assumption}, we have
\begin{equation}\label{estimate:nalnsigma_k}
  \|\na\ln\sigma^k\|_{C(\Om)} \leq (\theta^{k-1}\theta_0 +1)\|\na\ln \sigma^*\|_{C(\Om)} <2\|\na\ln \sigma^*\|_{C(\Om)}.
\end{equation}
Since $\|\na\ln\sigma^*\|_{C(\Om)}<\f{1}{4K}$, using the same way as that estimating \eref{fraction_sigma_pm}, we obtain
\begin{equation*}
  1\leq \f{\sigma^k_+}{\sigma^k_-} \leq \f{1}{[1-2K\|\na\ln\sigma^*\|_{C(\Om)}][1-4K\|\na\ln\sigma^*\|_{C(\Om)}]}\quad \mbox{for }k=2,3,\cdots,n.
\end{equation*}

Noting that the right hand side of \eref{estimate:nalnsigma_k} is independent of $k$, using exactly the same method as that used to derive \eref{estimate_I1_1}-\eref{estimate_I3_1} we have the following estimates
\begin{eqnarray}
  &\|\mathcal{I}_1^k\|_{C(\widetilde \Om)}\leq  \f{1}{\sqrt{2}\xi_0}\|\na\ln\sigma^*\|_{C(\Om)} C_7 \left\|\na\ln \f{\sigma^k}{\sigma^*}\right\|_{C(\Om)} \label{I_1k},\\
  &\|\mathcal{I}_2^k\|_{C(\widetilde \Om)}\leq  \f{1}{\sqrt{2}[1-K\|\na\ln\sigma^*\|_{C(\Om)}][1-2K\|\na\ln\sigma^*\|_{C(\Om)}]^2}\left\|\na\ln\f{\sigma^k}{\sigma^*}\right\|_{C(\widetilde \Om)}, \quad \label{I_2k}\\
  &\|\mathcal{I}_3^k\|_{C(\widetilde \Om)}\leq  \f{3 C_7 \|\na\ln\sigma^*\|_{C(\Om)} }{\sqrt{2}\xi_0[1-K\|\na\ln\sigma^*\|_{C(\Om)}][1-2K\|\na\ln\sigma^*\|_{C(\Om)}]^2} \left\|\na\ln\f{\sigma^k}{\sigma^*}\right\|_{C(\widetilde\Om)}. \quad\quad \label{I_3k}
\end{eqnarray}
Note that all the coefficients in the right hand sides of \eref{I_1k}-\eref{I_3k} are independent of $k$ and exactly the same as the estimates for $\mathcal{I}^1_i$ for $i=1,2,3$ respectively. Hence, we have exactly the same $\theta$ as that defined in Step 2 such that
\begin{equation}\label{estimate:nalnsigma^k+1_0}
   \left\|\na\ln\f{\sigma^{k+1}}{\sigma^*}\right\|_{C(\Om)} \leq \|\mathcal{I}_1^k\|_{C(\widetilde \Om)} + \|\mathcal{I}_2^k\|_{C(\widetilde \Om)} + \|\mathcal{I}_3^k\|_{C(\widetilde \Om)} \leq \theta \left\|\na\ln\f{\sigma^k}{\sigma^*}\right\|_{C(\widetilde\Om)}.
\end{equation}
From \eref{assumption} and \eref{estimate:nalnsigma^k+1_0} we obtain
\begin{equation*}
  \left\|\na \ln\f{\sigma^{k+1}}{\sigma^*}\right\|_{C(\widetilde \Om)} \leq \theta_0 \|\na \ln\sigma^*\|_{C(\Om)}\theta^{k}\le \epsilon_0\theta^{k}
\end{equation*}
from induction arguments. From Lemma \ref{lemma:C1_sigma_n} and $\sigma^n = \sigma^* = \sigma_b$ in $\Om\setminus \widetilde \Om$ we obtain
\begin{equation}\label{estimate_nabla_ln_C}
  \left\|\na \ln\f{\sigma^{k+1}}{\sigma^*}\right\|_{C(\Om)} \leq \epsilon_0\theta^{k}.
\end{equation}

Similarly to \eref{relation:sigma_logsigma} we also have that
\begin{equation}\label{estimate_ln_C}
 \left\|\ln\f{\sigma^{k+1}}{\sigma^*}\right\|_{C(\Om)} \leq K \left\|\na \ln\f{\sigma^{k+1}}{\sigma^*}\right\|_{C(\Om)}.
\end{equation}
Combining \eref{estimate_nabla_ln_C} and \eref{estimate_ln_C}, we obtain \eref{estimate:main}.
The proof is complete.
\end{proof}

\begin{remark}
Note that due to the compactness of the admissible set $\mathcal{A}[\ep_0,\sigma^0_\pm,\sigma_b]$, it is difficult to guarantee $\sigma^n\in \mathcal{A}[\ep_0,\sigma^0_\pm,\sigma_b]$. However, Theorem \ref{theorem:convergence} tells us that  $\sigma_n\to \sigma^*$ as $n\to \infty$ in the $C^1$ sense.
\end{remark}

\section{Numerical implementations}
To validate our proposed scheme with convergence analysis, we present some numerical experiments using three different models. To test the convergence property shown in Theorem \ref{theorem:convergence} for target conductivities of different smoothness characterized by $\|\na\ln\sigma^*\|_{C(\overline{\Om})}$, in each model we compare the convergence behavior of the iteration process between the cases of $\sigma_i$ and their blurred versions $\widehat \sigma_i$ for $i=1,2,3$. To be precise, for a given $\sigma$ in $\Omega$, we extend it periodically into $\R^2$ and then take the convolution using the kernel
\begin{equation}\label{def:kernel}
  \mathcal{K}(x,y) = \f{1}{c}e^{-\f{x^2+y^2}{2\nu^2}}, \qquad (x,y)\in \R^2
\end{equation}
with $c = \int_{\R^2}e^{-\f{x^2+y^2}{2\nu^2}} dx dy$ to yield the smooth function
\begin{equation}\label{convolution}
  \widehat\sigma(x,y) = \int_{\R^2}\mathcal{K}(x-x',y-y')\sigma(x',y')dx'dy'.
\end{equation}
With a chosen $\nu$ we can obtain a blurred $\widehat \sigma$ corresponding to each original $\sigma$.

For simplicity, we evaluate the performance of the convergence behavior in terms of the relative errors defined by
\begin{equation*}\label{RE}
  RE_i(n) = \f{\|\ln\sigma_i^n-\ln\sigma_i\|_{C(\Om)}}{\|\ln\sigma_i\|_{C(\Om)}}
\end{equation*}
and
\begin{equation*}\label{RE_hat}
  \widehat{RE}_i(n) = \f{\|\ln\widehat \sigma_i^n-\ln\widehat\sigma_i\|_{C(\Om)}}{\|\ln\widehat\sigma_i\|_{C(\Om)}}
\end{equation*}
for $i=1,2,3$. Here $\sigma_i^n$ and $\widehat \sigma_i^n$ are respectively the reconstructed conductivity using the single current harmonic $B_z$ algorithm at the $n$-th step corresponding to true conductivities $\sigma_i$ and $\widehat \sigma_i$ ($i=1,2,3$).

\subsection{A toy model}\label{sec:toy_simulation}
We considered a two-dimensional toy model in a square domain $\Om:=[-1,1]\times [-1,1] \subset \R^2$, similar to the method used in \cite{Liu2010}. We defined a pair of electrodes with length $0.3$ and thickness $0.1$ to the left and right boundaries of $\Om$, i.e.,
$\mE^{\pm} = \{(x,y):\;|x\pm 1.05|\leq 0.05~\mbox{and } |y|\leq 0.15\}$. We first set the target conductivity distribution in $\Om$ to be
\begin{equation*}
  \sigma_1(x,y) = \left\{
    \begin{split}
      & 1+\f{1}{2}\left(\cos\sqrt{x^2+y^2}-\f{\sqrt{3}}{2}\right)+1,\qquad 0\leq \sqrt{x^2+y^2}\leq \f{\pi}{8}, \\
      & 1,\qquad \mbox{otherwise}.
    \end{split}
  \right.
\end{equation*}

To obtain $\widehat \sigma_1$, we specified $\nu=5$. The matrix size was $128\times 128$. The convolution \eref{convolution} was computed by a weighted summation with the discretized version of $\mathcal{K}$ in a window with a size of $6\times 6$ pixels. We illustrate $\sigma_1$, $\widehat \sigma_1$ and the configuration of $\Om\cup \mE^\pm$ in Fig.~\ref{fig:smooth_setup}.

\begin{figure}[h]
\onecolumn
\centering
\subfigure[]{
\includegraphics[width=3.7cm,height=3.7cm]{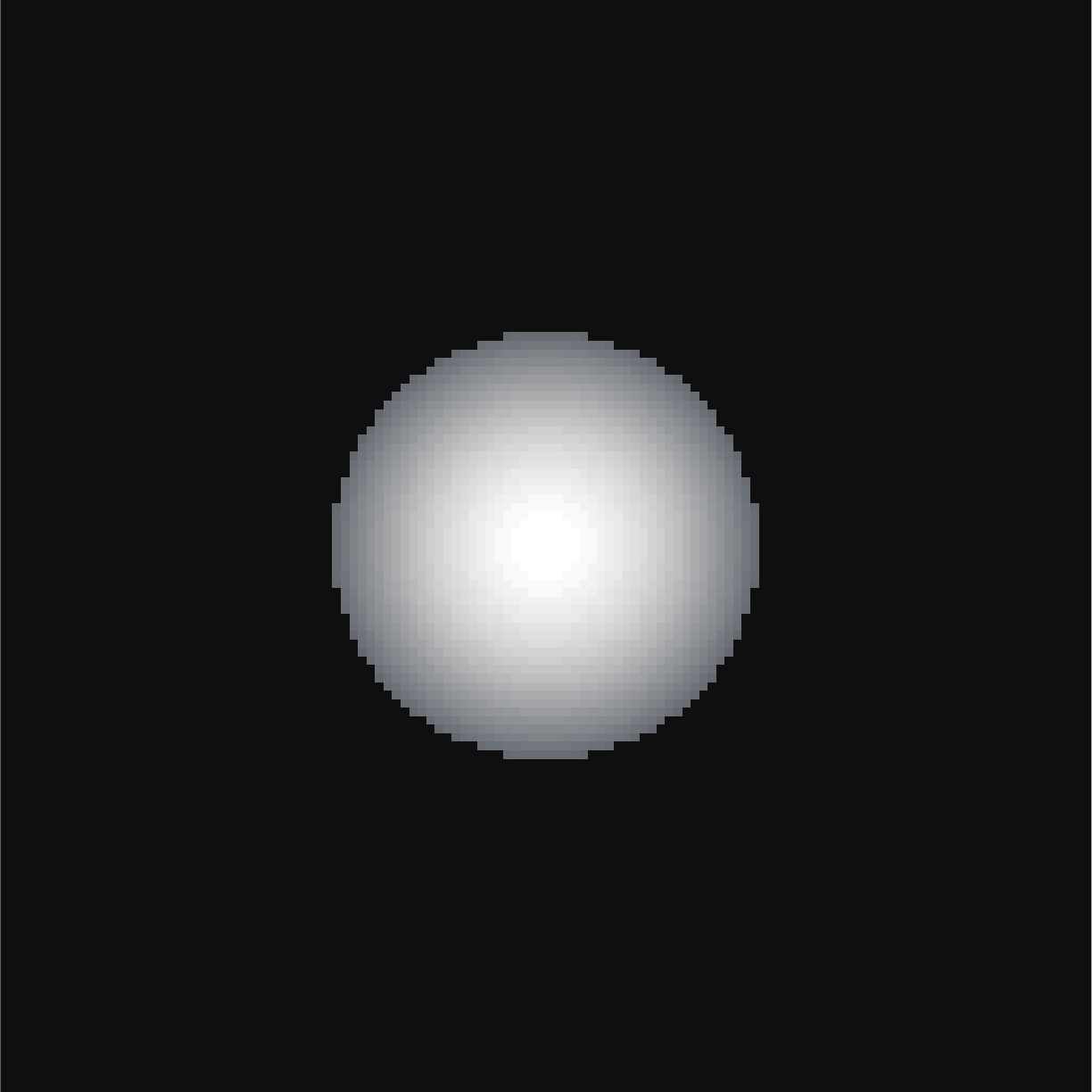}}
~\subfigure[]{
\includegraphics[width=4.3cm,height=3.7cm]{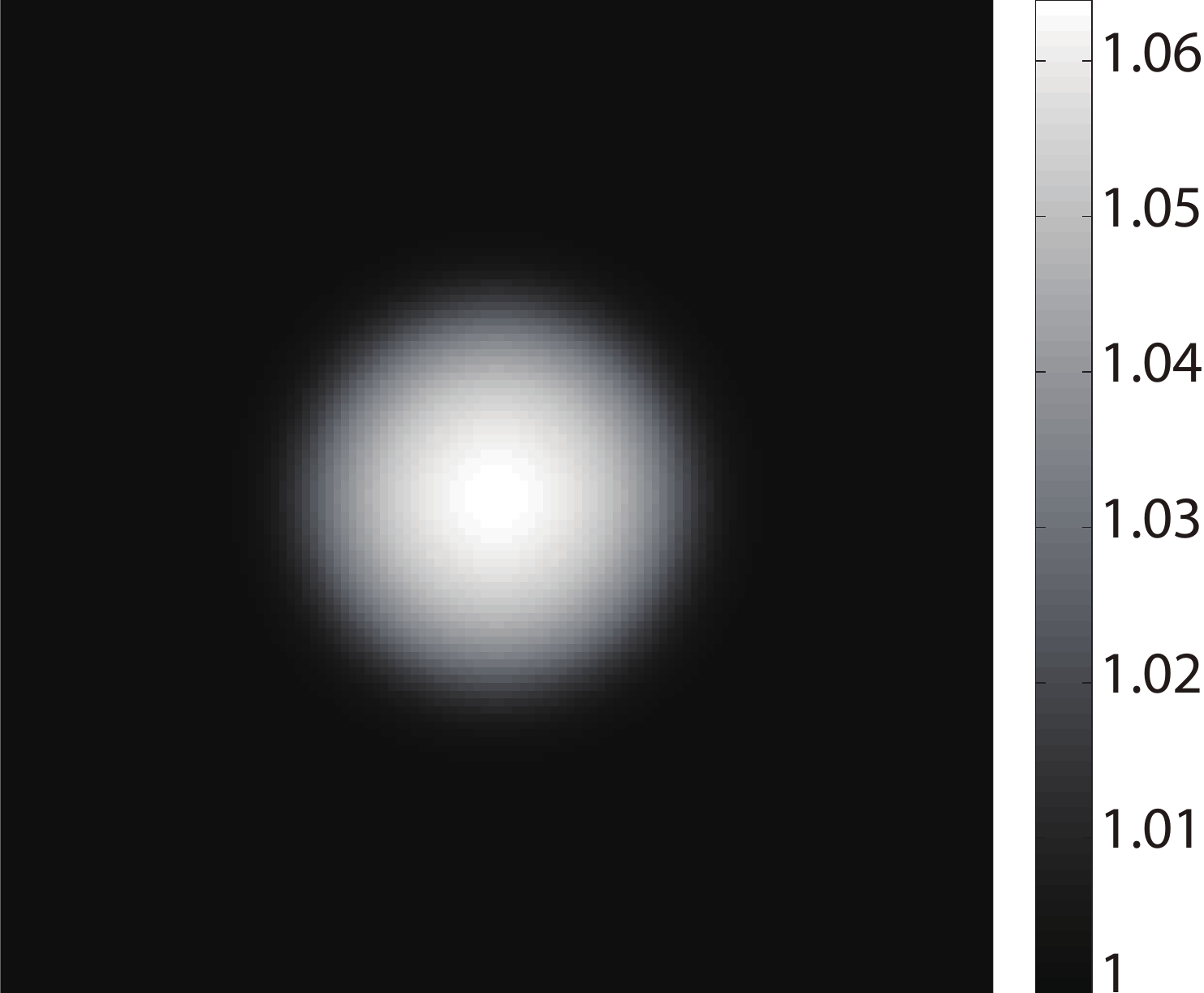}}
~\subfigure[]{
\includegraphics[width=4.0cm,height=3.7cm]{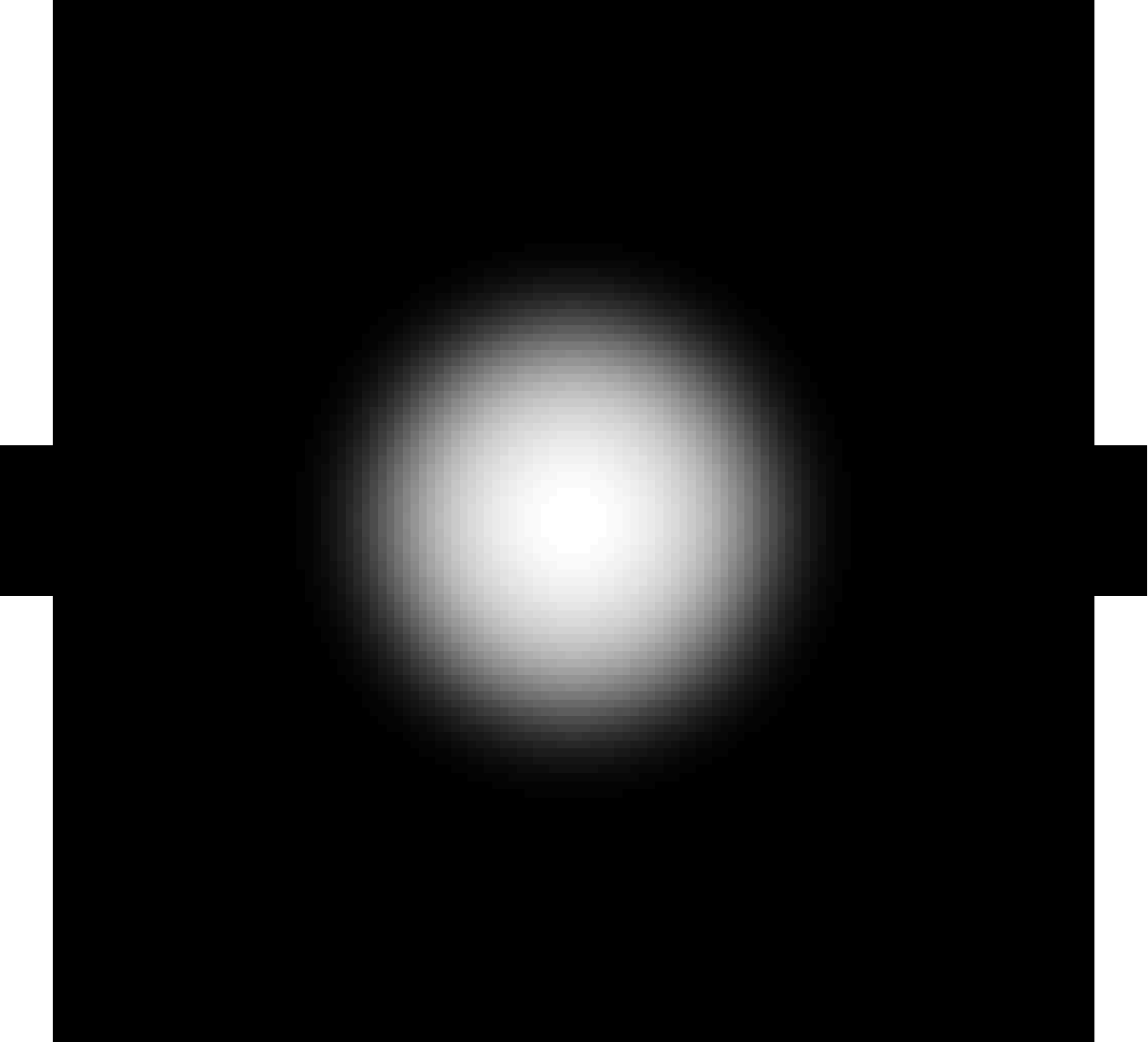}}
\caption{The geometry $\Om$ and conductivity distributions: (a) $\sigma_1(x,y)$, (b) $\widehat \sigma_1(x,y)$ and (c) $\Om\cup \mE^\pm$.}
\label{fig:smooth_setup}
\end{figure}

We simulated injection of a current with amplitude $10$ mA through the electrodes. The finite element method, was used to solve the equation \eref{eq:forwardPDE} for $\sigma_1$ and $\widehat \sigma_1$ to obtain the solutions $u_1$ and $\widehat u_1$, respectively. We then obtained the distributions $B_{z,1}$ and $\widehat{B}_{z,1}$ using the formula \eref{biot-savart law} and the FFT method \cite{Yazdanian2020}. Images of $B_{z,1}$ and $\widehat{B}_{z,1}$ are depicted in Figure \ref{fig:Circle_Bz} (a) and (b) respectively. It can be seen that the magnetic fields are not sensitive to the smoothness of the conductivity, illustrating the ill-posedness of this conductivity imaging problem.
\begin{figure}[h]
\onecolumn
\centering
\subfigure[]{
\includegraphics[width=3.7cm,height=3.7cm]{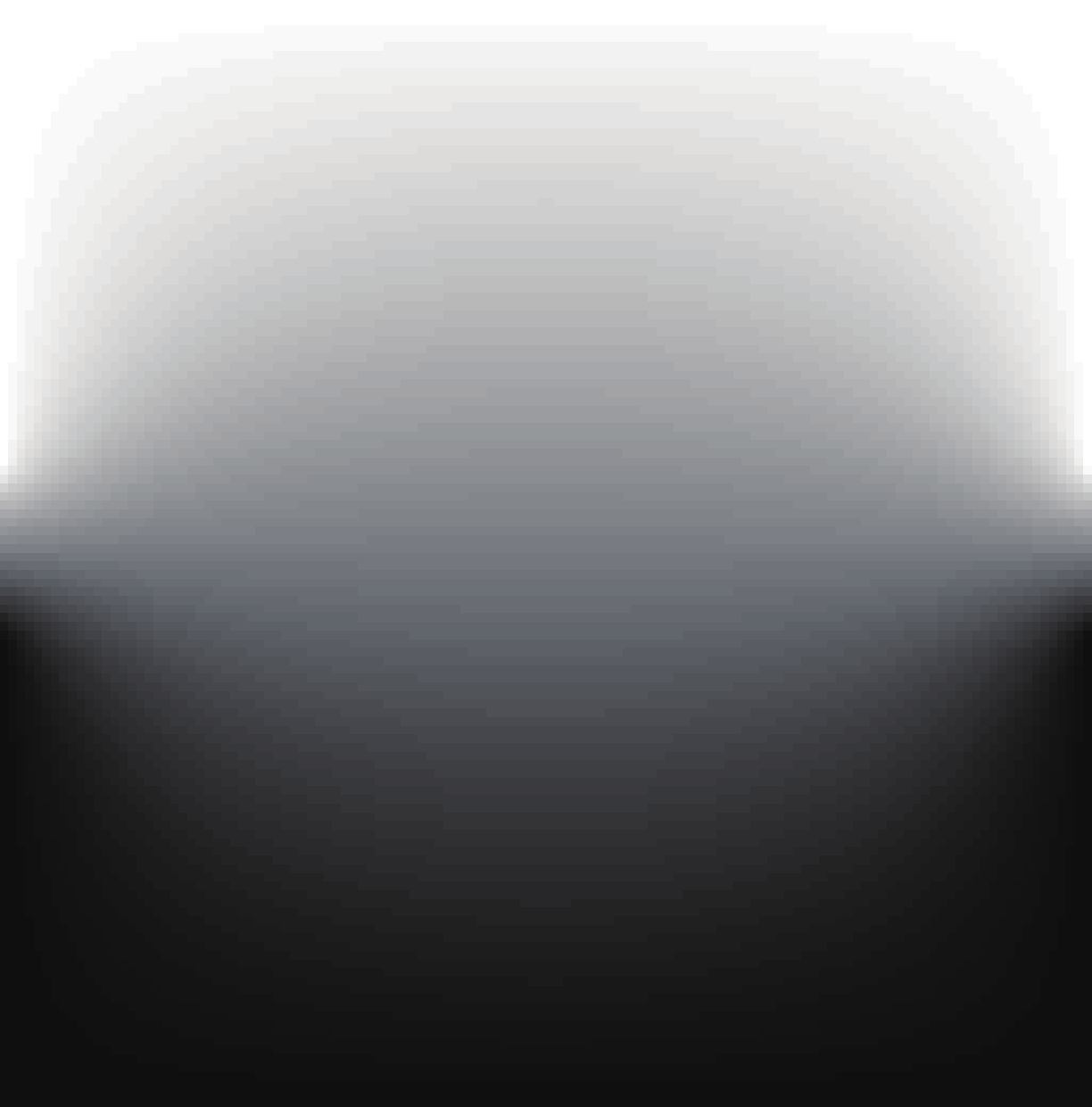}}
~\subfigure[]{
\includegraphics[width=4.5cm,height=3.7cm]{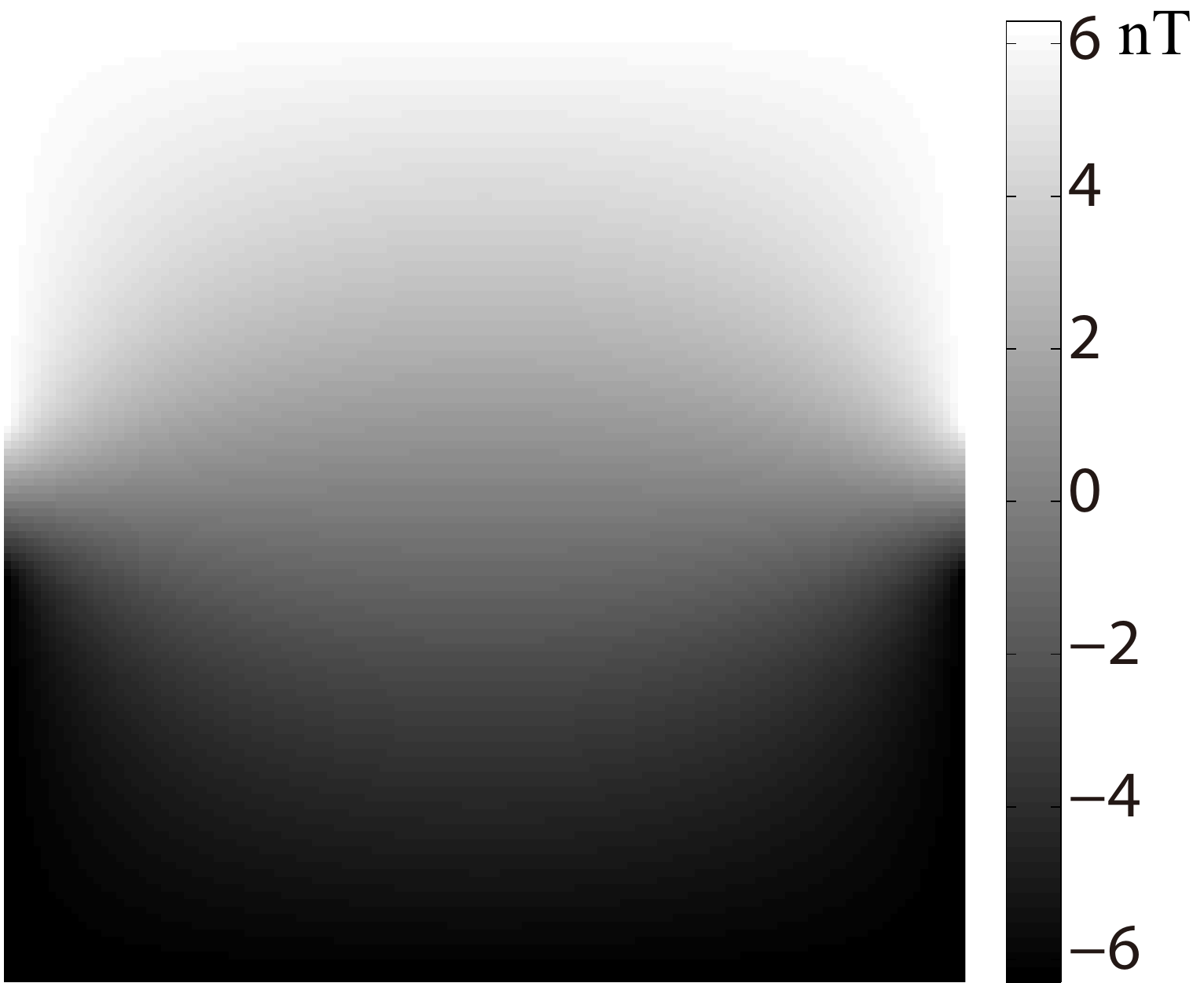}}
\caption{(a) illustrates $B_{z,1}$ while (b) is the image of $\widehat{B}_{z,1}$.}
\label{fig:Circle_Bz}
\end{figure}

Using these magnetic flux density data and the single current harmonic $B_z$ algorithm, we then obtained the reconstructed conductivity $\sigma_1^{n}$ and $\widehat \sigma_1^n$. Figure \ref{fig:circle_result} (a), (b) and (c) show $\sigma_1^n$ when $n=1,~20$ and $50$ respectively. Figure \ref{fig:circle_result} (d), (e) and (f) show  $\widehat\sigma_1^n$ results corresponding to $n=1,~20$ and $50$ respectively.

\begin{figure}[h]
\onecolumn
\centering
\subfigure[]{
\includegraphics[width=3.7cm,height=3.7cm]{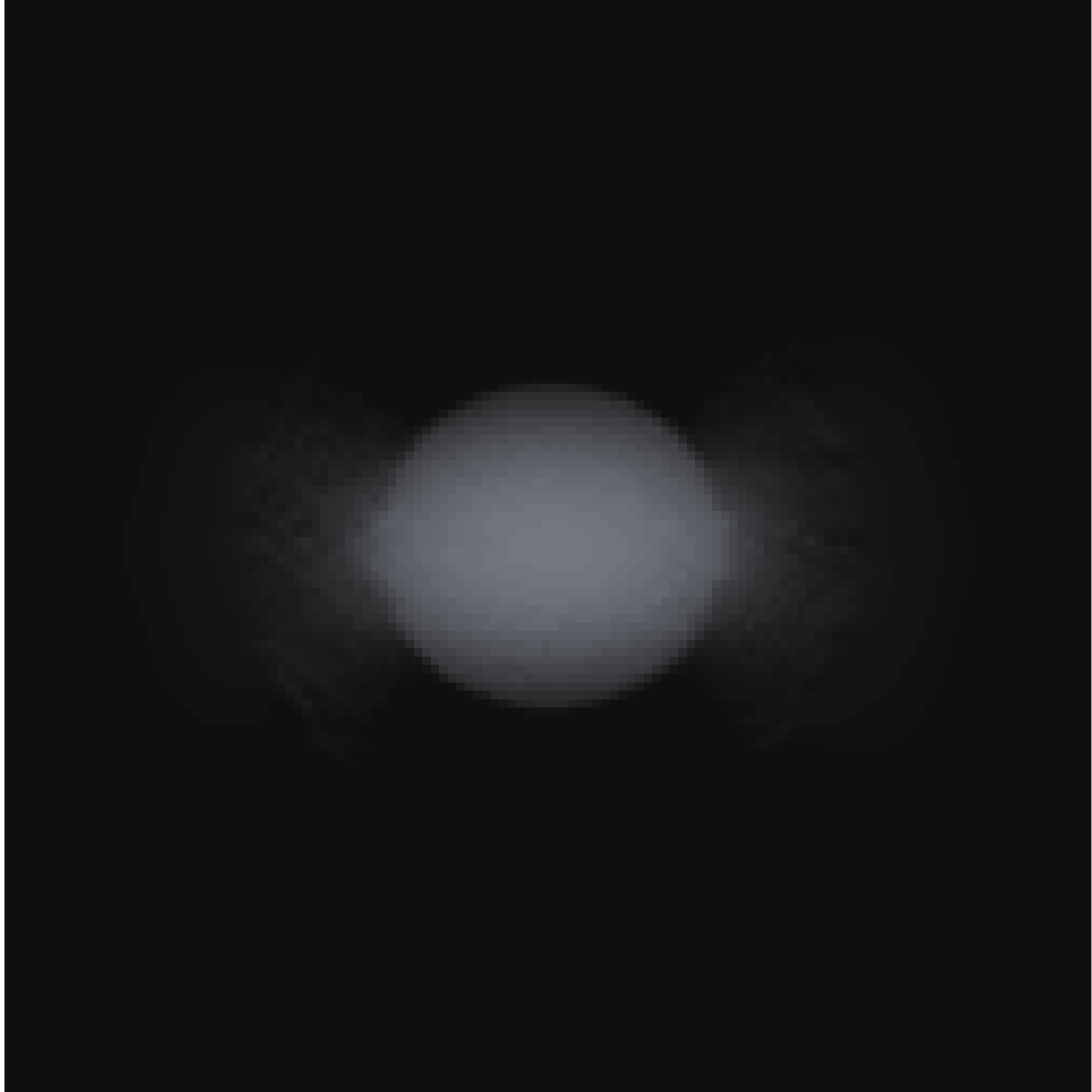}}
~\subfigure[]{
\includegraphics[width=3.7cm,height=3.7cm]{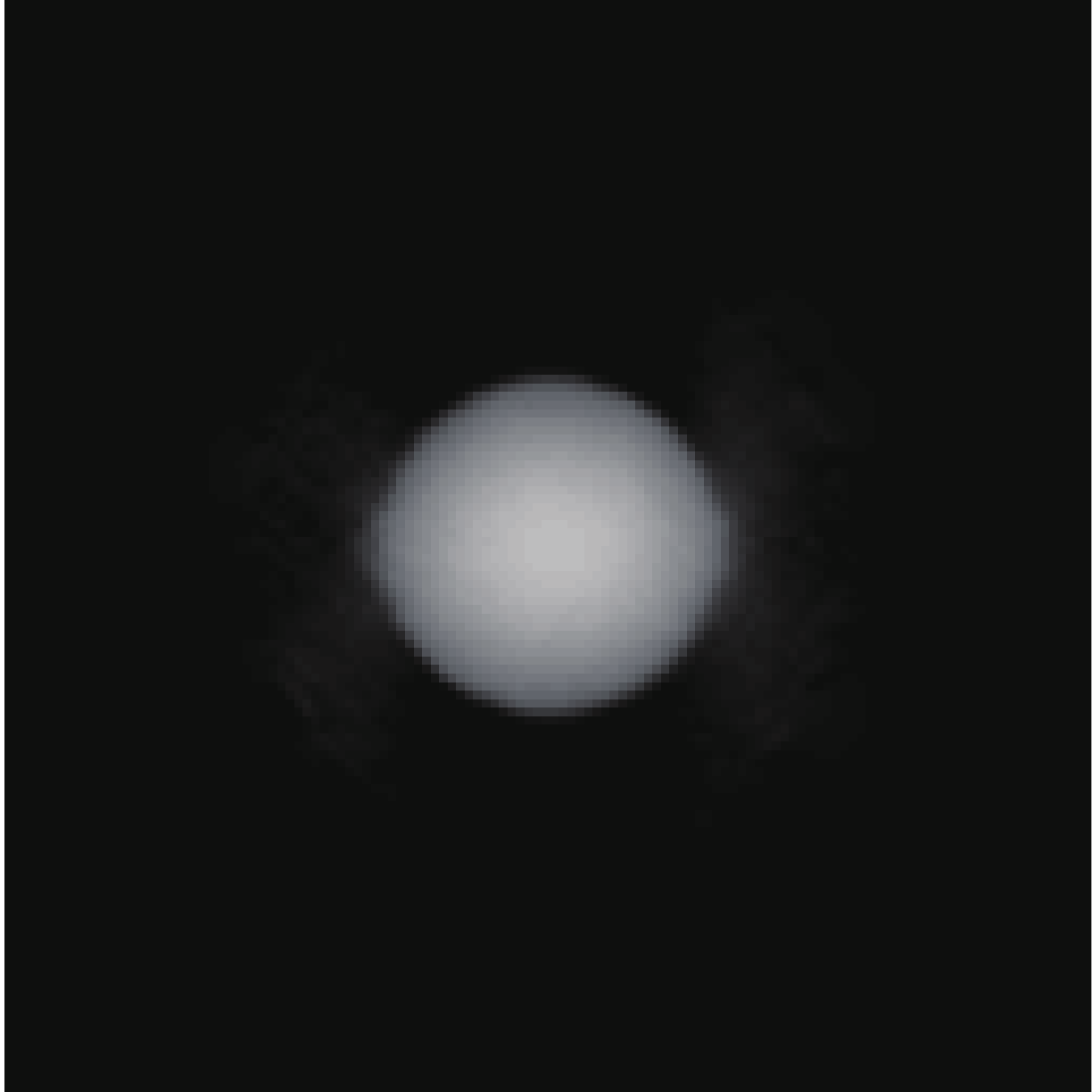}}
~\subfigure[]{
\includegraphics[width=4.3cm,height=3.7cm]{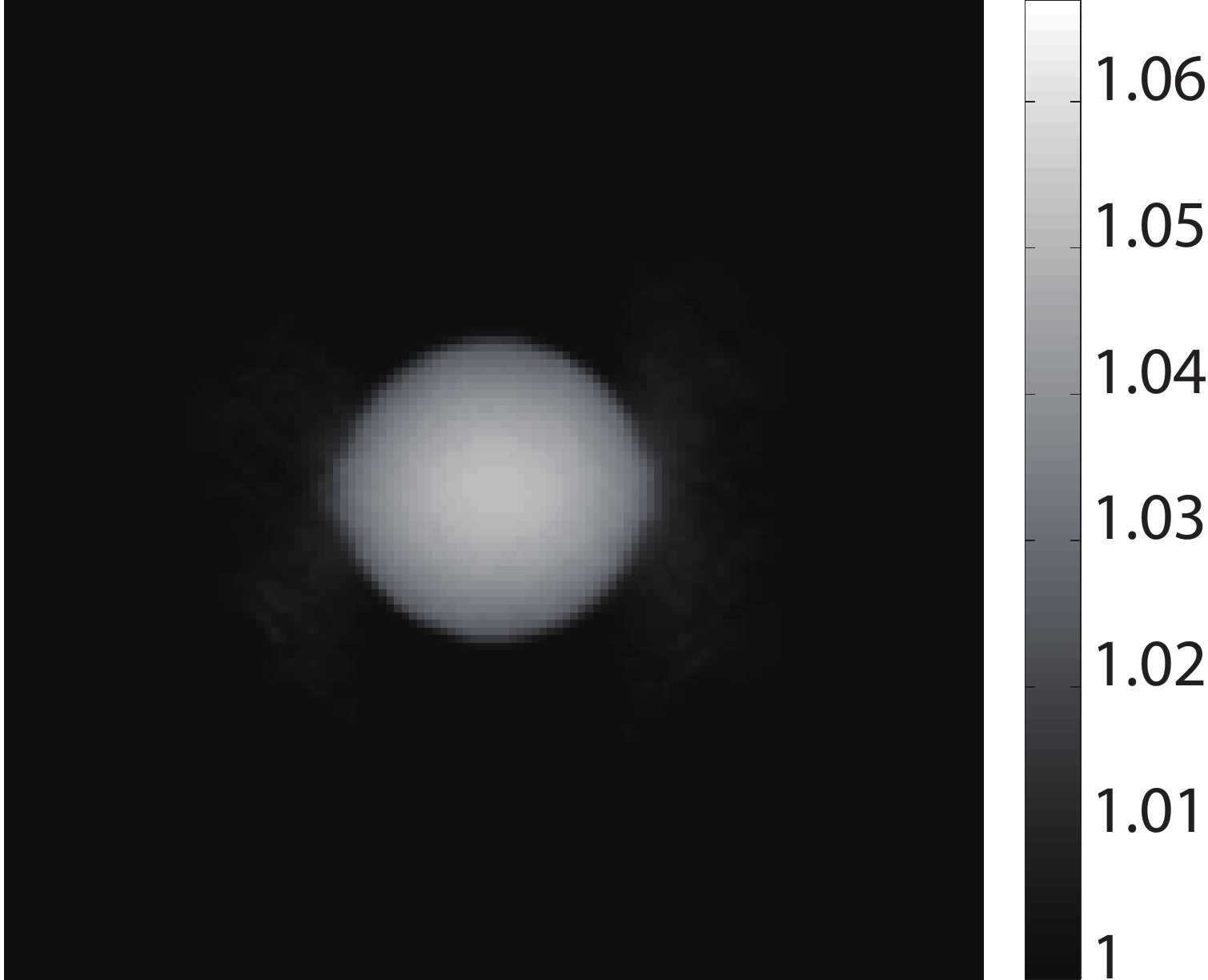}}
~\subfigure[]{
\includegraphics[width=3.7cm,height=3.7cm]{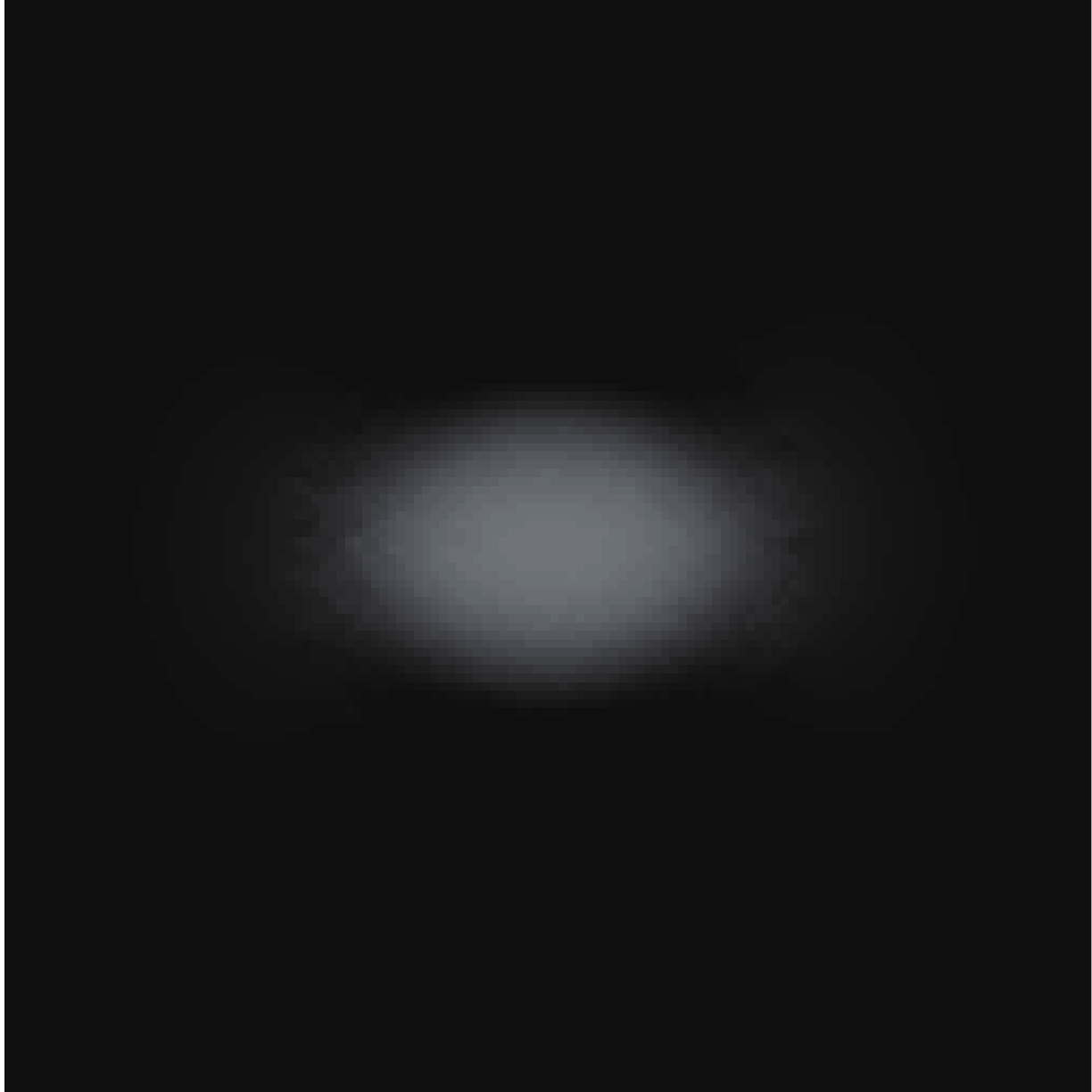}}
~\subfigure[]{
\includegraphics[width=3.7cm,height=3.7cm]{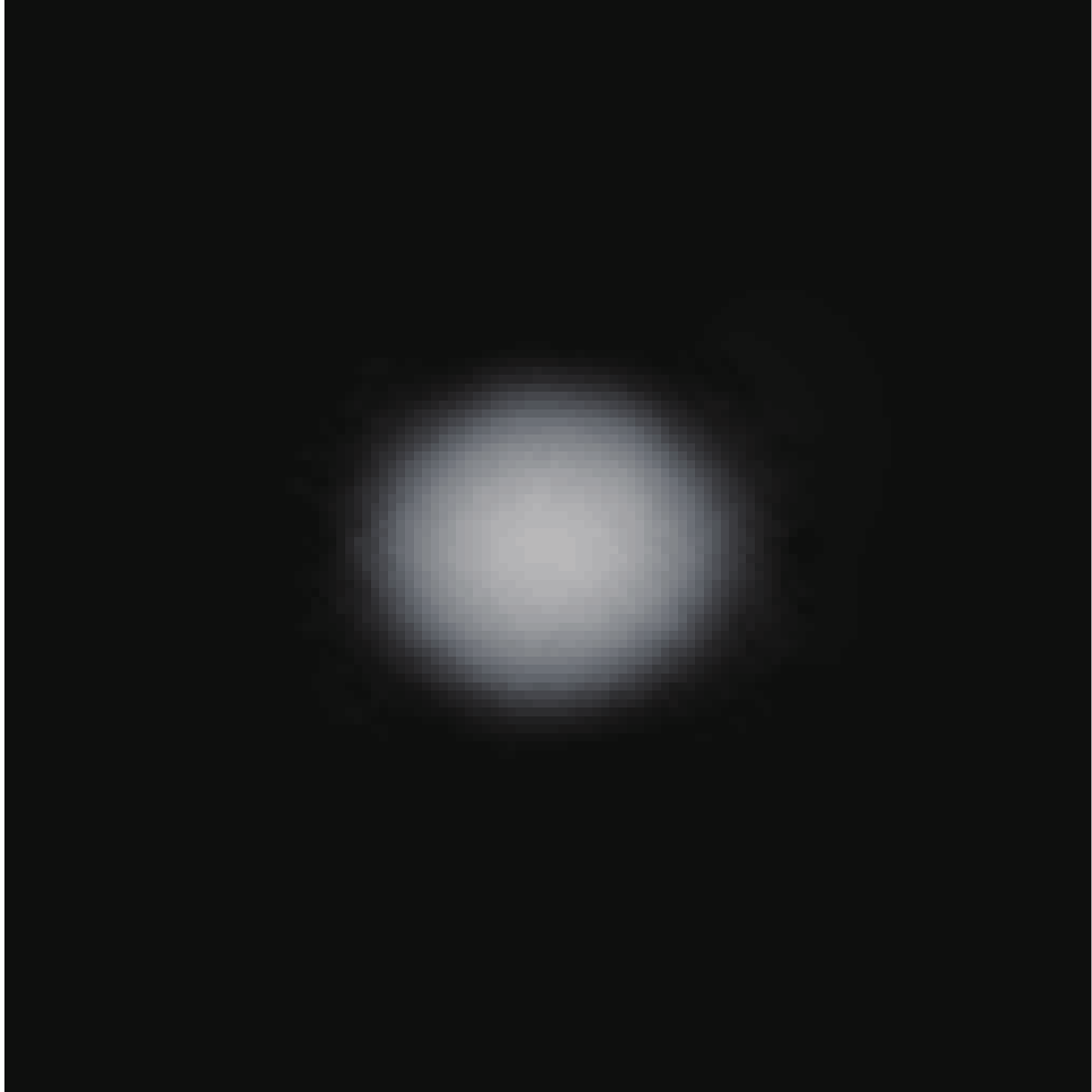}}
~\subfigure[]{
\includegraphics[width=4.3cm,height=3.7cm]{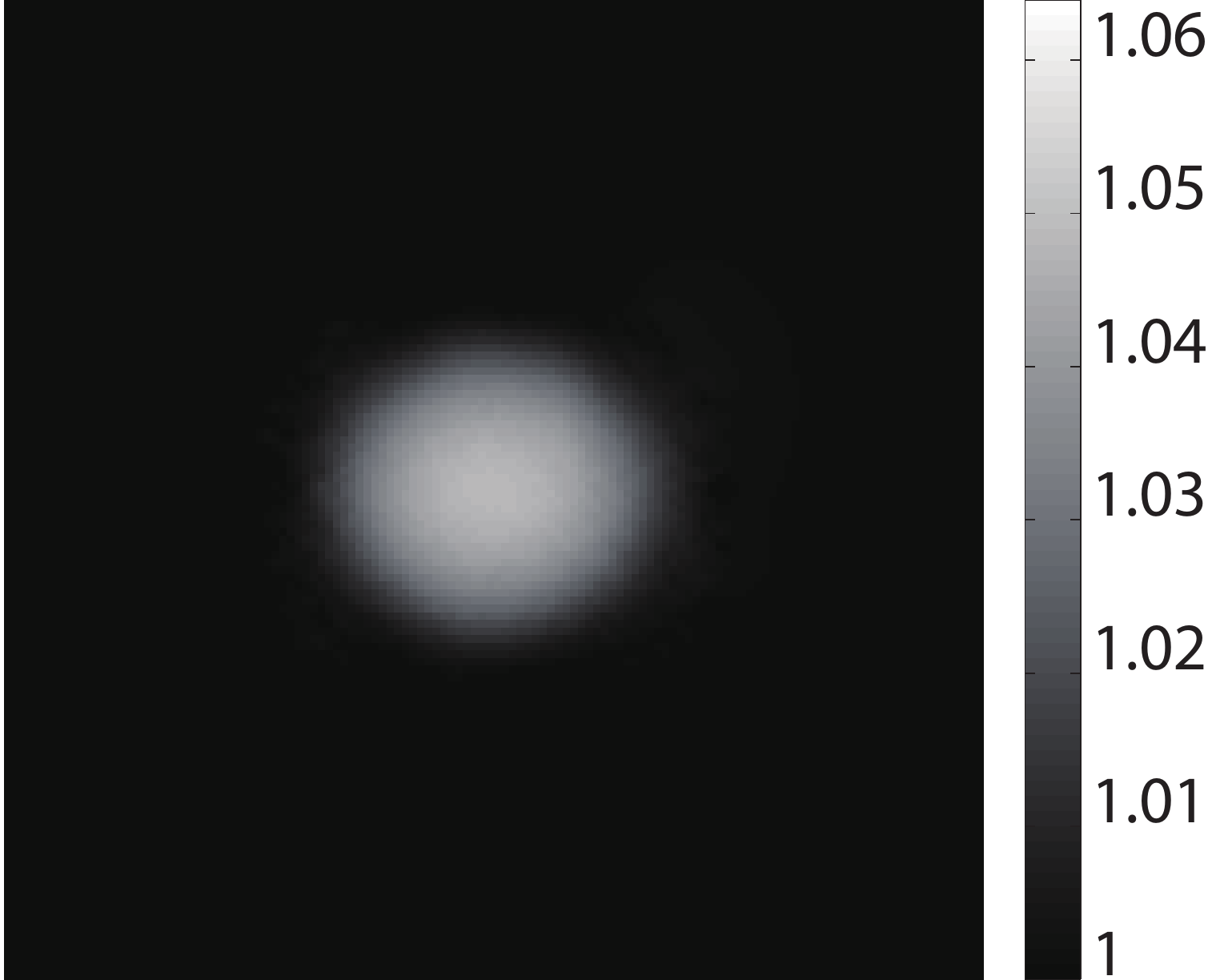}}
\caption{(a) $\sigma_1^1$, (b) $\sigma_1^{20}$, (c) $\sigma_1^{50}$, (d) $\widehat\sigma_1^1$, (e) $\widehat\sigma_1^{20}$ and (f) $\widehat\sigma_1^{50}$.}
\label{fig:circle_result}
\end{figure}

In Figure \ref{Fig:convergence_circle} we compare the asymptotic behaviors of $RE_1(n)$ and $\widehat{RE}_1(n)$ with respect to the iteration process for $n=1,\cdots, 50$. In Figure \ref{Fig:convergence_circle}, the dashed line with diamond markers shows the asymptotic behavior of $RE_1(n)$ while the solid line with circles depicts the behavior of $\widehat{RE}_1(n)$. As we can see, the relative errors for $\widehat \sigma_1^n$ are much smaller than that for $\sigma_1$. To quantitatively illustrate the behavior of $RE_1(n)$ and $\widehat{RE}_1(n)$ we also provide Table \ref{table:RE} for $n=5,10,15,\cdots,50$. We see here that for steps $20$ and $50$, relative errors are $RE_1(20) = RE_1(50)= 3.9\%$, while  $\widehat{RE}_1(20) = \widehat{RE}_1(50)= 2.43\%$.
\begin{figure}[h]
  \centering
  \includegraphics[width=10cm,height=5cm]{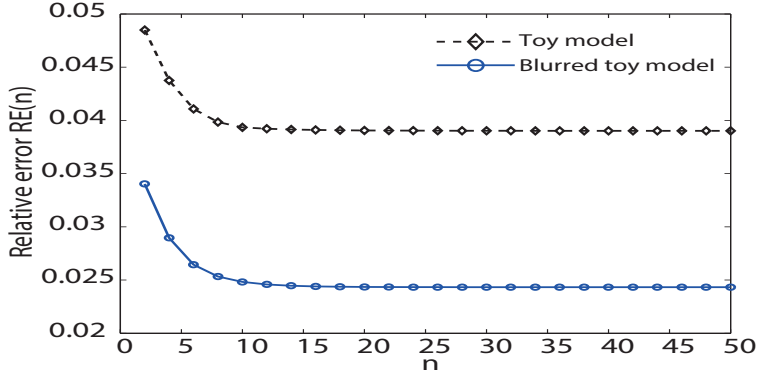}\\
  \caption{Asymptotic behaviors of $RE_1(n)$ and $\widehat{RE}_1(n)$ for $n\leq 50$. The dashed line with diamond markers shows $RE_1(n)$ while the solid line with circles shows $\widehat{RE}_1(n)$. }\label{Fig:convergence_circle}
\end{figure}

\subsection{A modified Shepp-Logan phantom}

Now we define $\Om$ to be the modified Shepp-Logan phantom \cite{SheppLogan1974}. To be precise, we construct $\p\Om$ by a disc centered at $(0,0)$ having a diameter of $0.45$m. We set the matrix size to be $128\times 128$ and the field of view to be $0.6 \mbox{m}\times 0.6 \mbox{m}$. We attached a pair of electrodes on $\p \Om$ with length 0.0972 m. We applied the convolution \eref{def:kernel} with $\nu = 1.2$ and window size to be $6\times 6$ pixels to the classical conductivity distribution $\sigma_2$ shown in Figure \ref{fig:shepp-logan_setup} (a), to obtain a blurred conductivity $\widehat \sigma_2\in C^1(\Omega)$, illustrated in Figure \ref{fig:shepp-logan_setup} (b).
\begin{figure}[h]
\onecolumn
\centering
\subfigure[]{
\includegraphics[width=3.7cm,height=3.7cm]{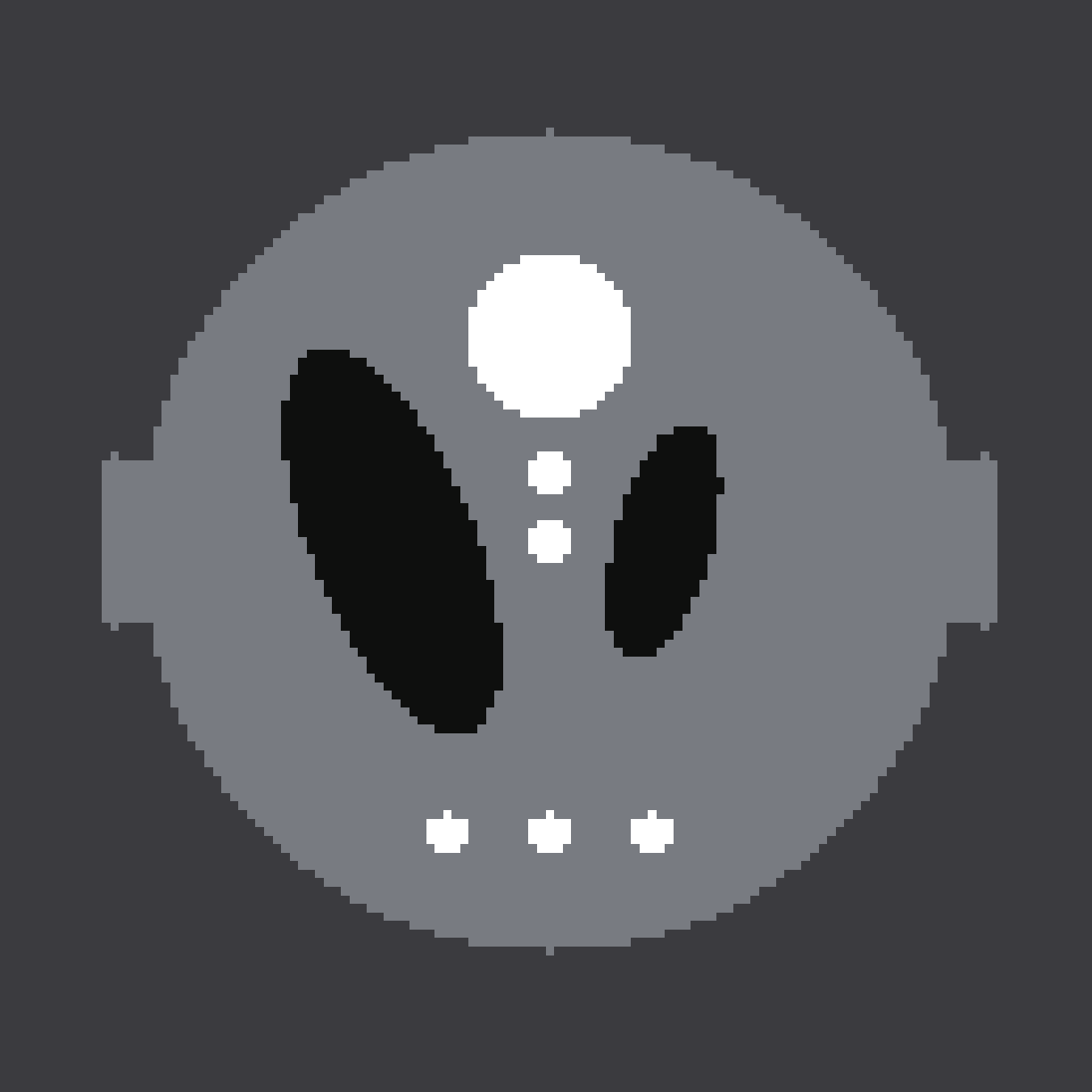}}
~\subfigure[]{
\includegraphics[width=4.5cm,height=3.7cm]{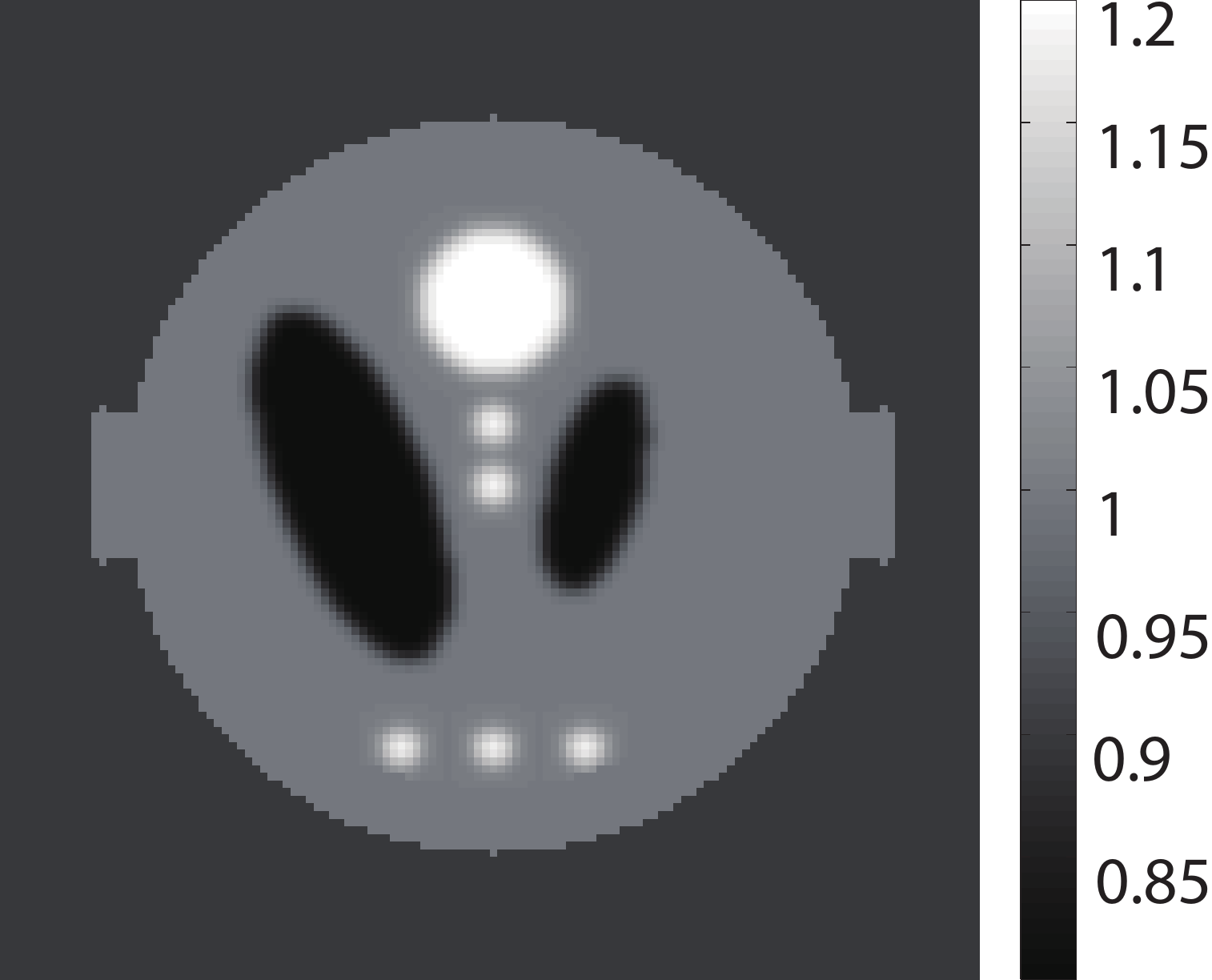}}
\caption{Simulation setup and conductivity distributions used in the modified Shepp-Logan phantom. (a) $\sigma_2$, (b) $\widehat \sigma_2$.}
\label{fig:shepp-logan_setup}
\end{figure}

As in Section \ref{sec:toy_simulation}, a current with amplitude 10 mA was injected, and the resulting $B_{z,2}$ and $\widehat B_{z,2}$ images are shown in Figures \ref{fig:Shepp_Bz} (a) and (b) respectively. The reconstructed conductivities $\sigma_2^n$ and $\widehat \sigma_2^n$ for $n=1,20$ and $50$ found using our iteration scheme are shown in Figure \ref{fig:Shepp-Logan_result}.

\begin{figure}[h]
\onecolumn
\centering
\subfigure[]{
\includegraphics[width=3.7cm,height=3.7cm]{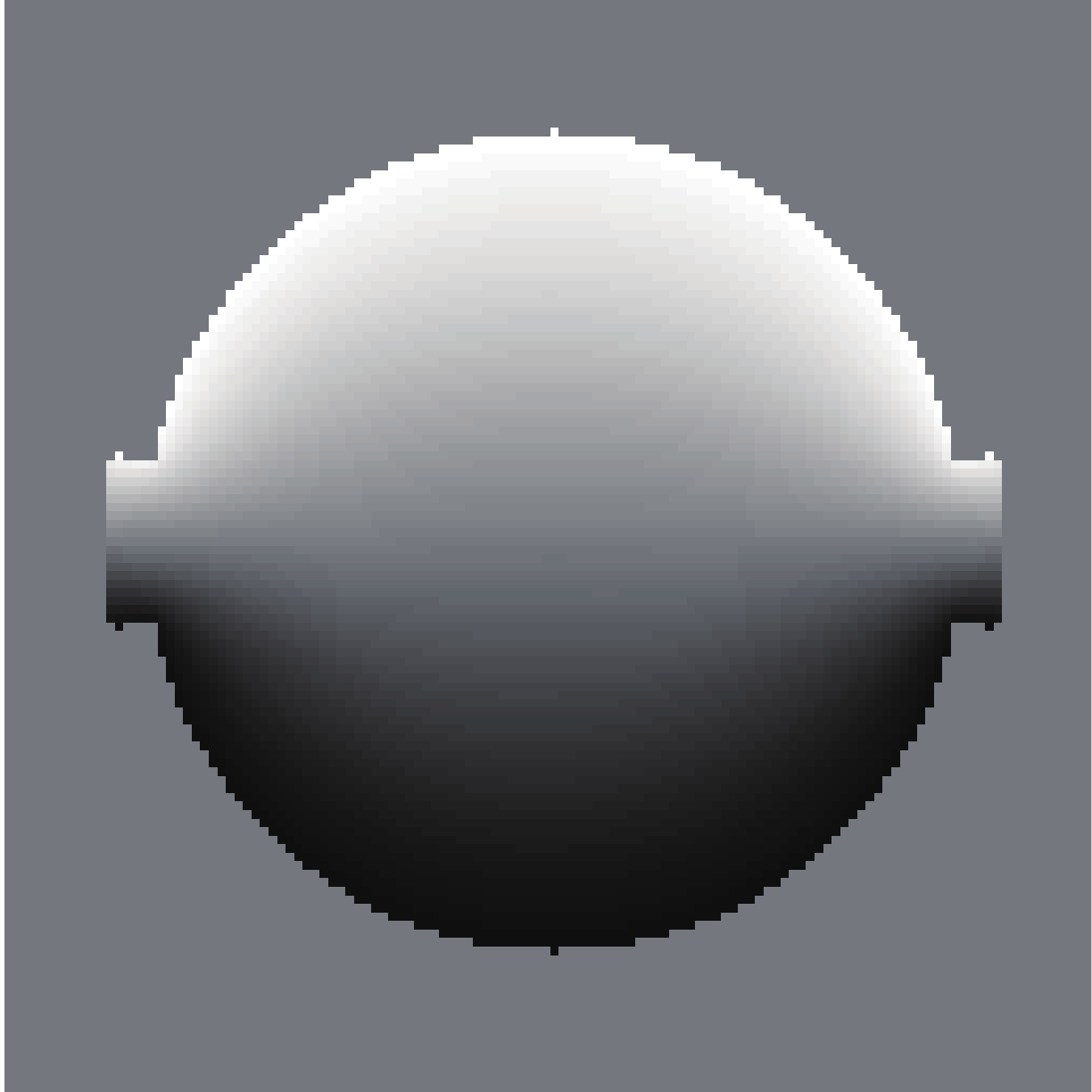}}
~\subfigure[]{
\includegraphics[width=4.5cm,height=3.7cm]{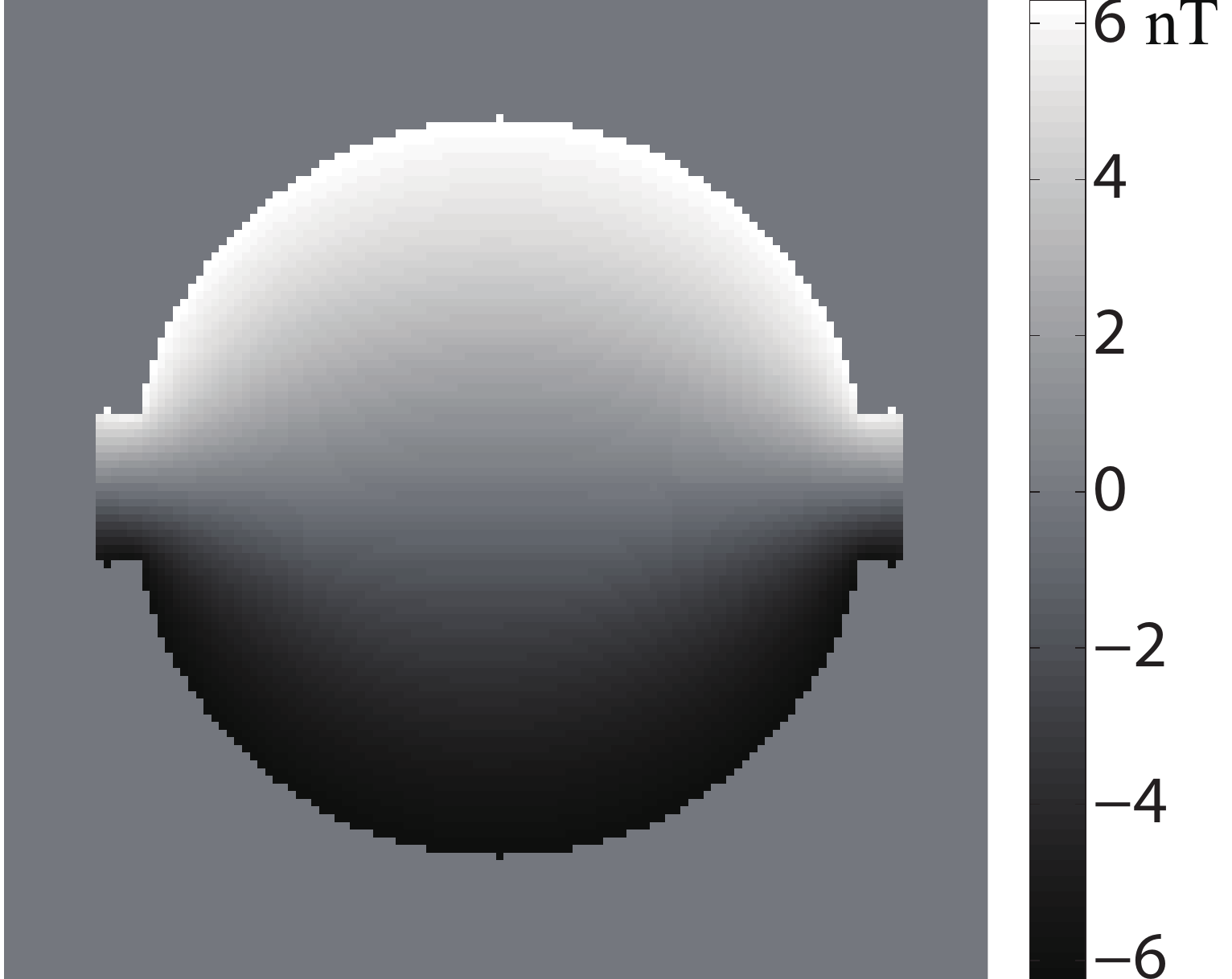}}
\caption{(a) $B_{z,2}$, (b) $\widehat B_{z,2}$.}
\label{fig:Shepp_Bz}
\end{figure}

\begin{figure}[h]
\onecolumn
\centering
\subfigure[]{
\includegraphics[width=3.7cm,height=3.7cm]{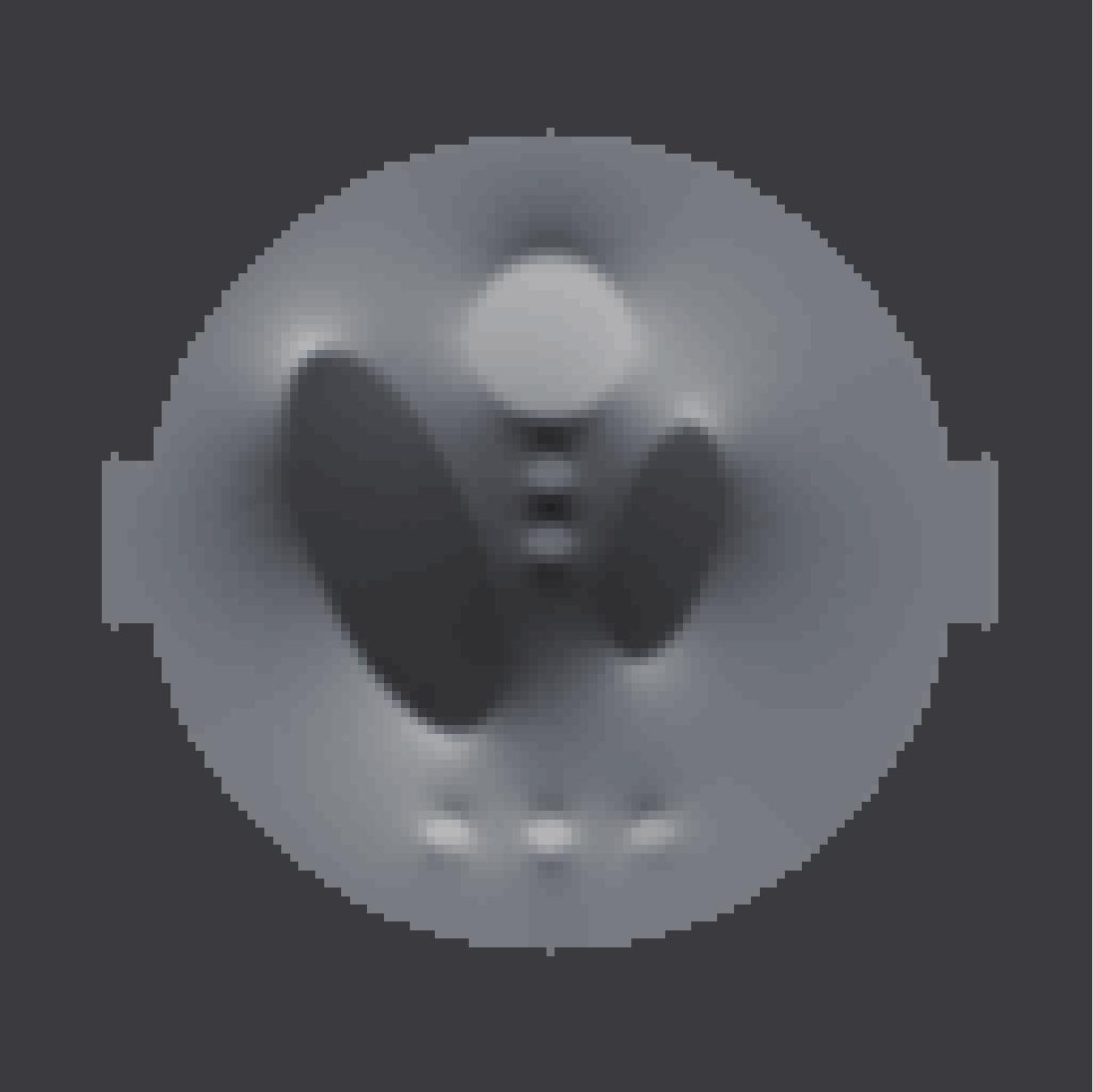}}
~\subfigure[]{
\includegraphics[width=3.7cm,height=3.7cm]{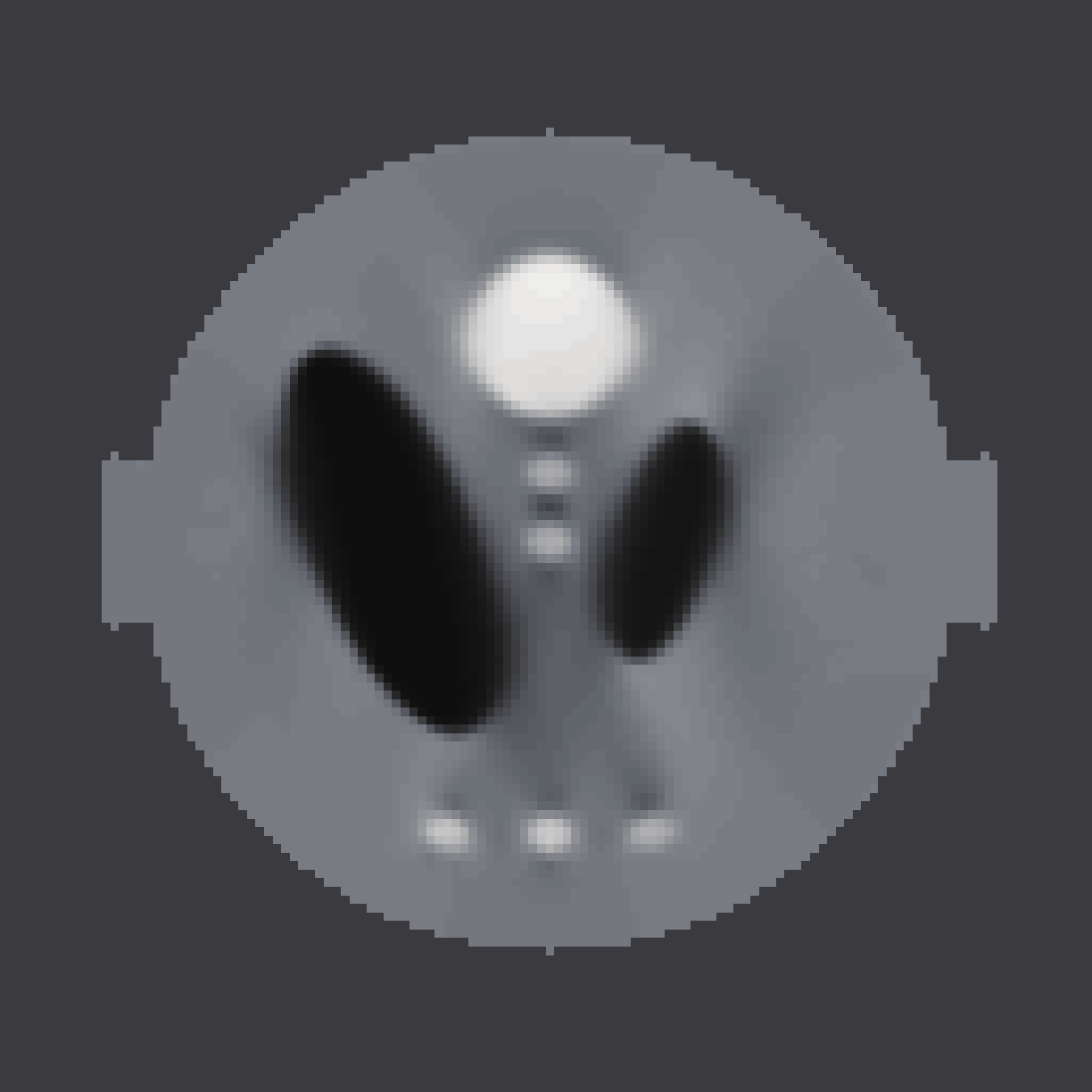}}
~\subfigure[]{
\includegraphics[width=4.5cm,height=3.7cm]{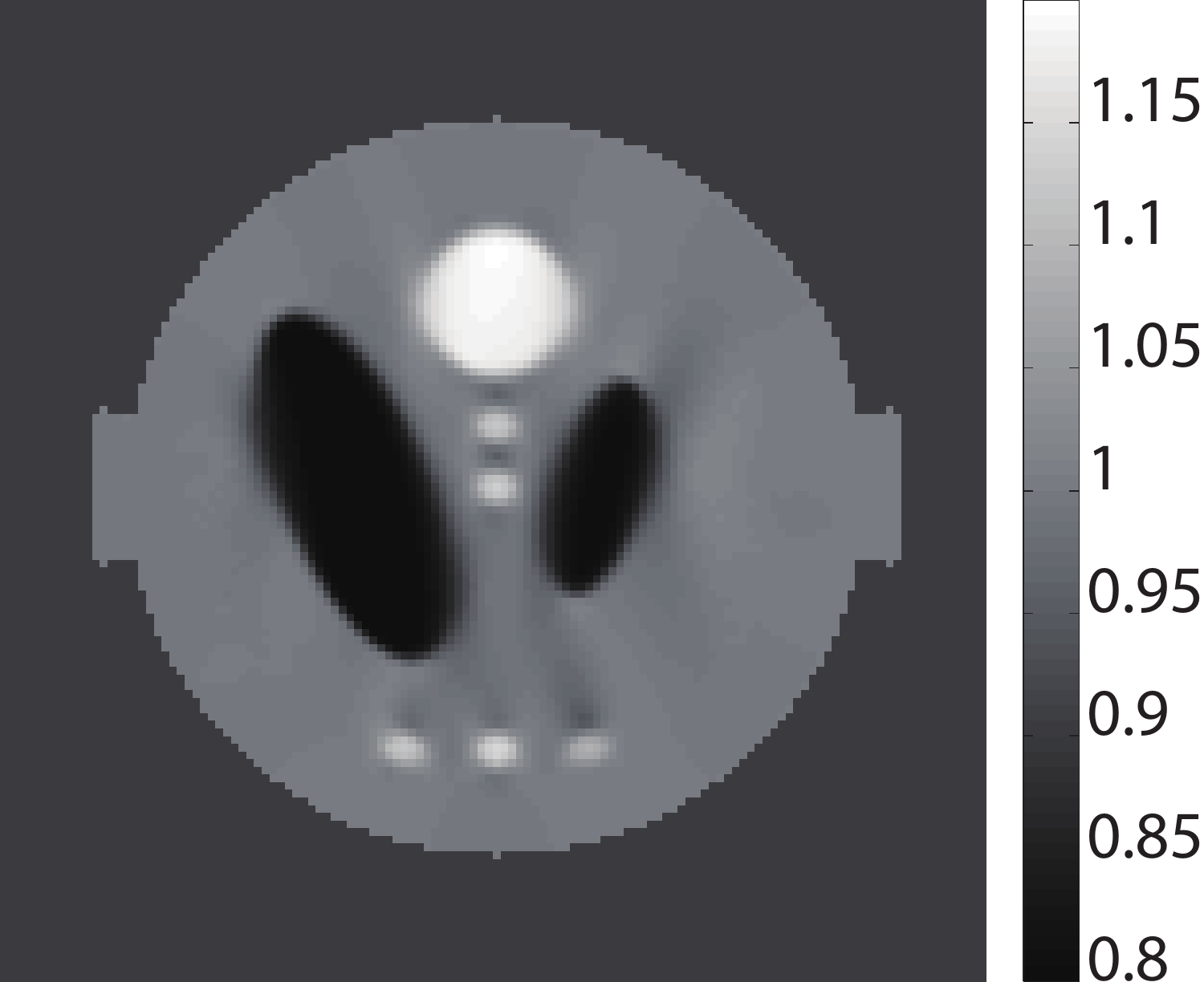}}
~\subfigure[]{
\includegraphics[width=3.7cm,height=3.7cm]{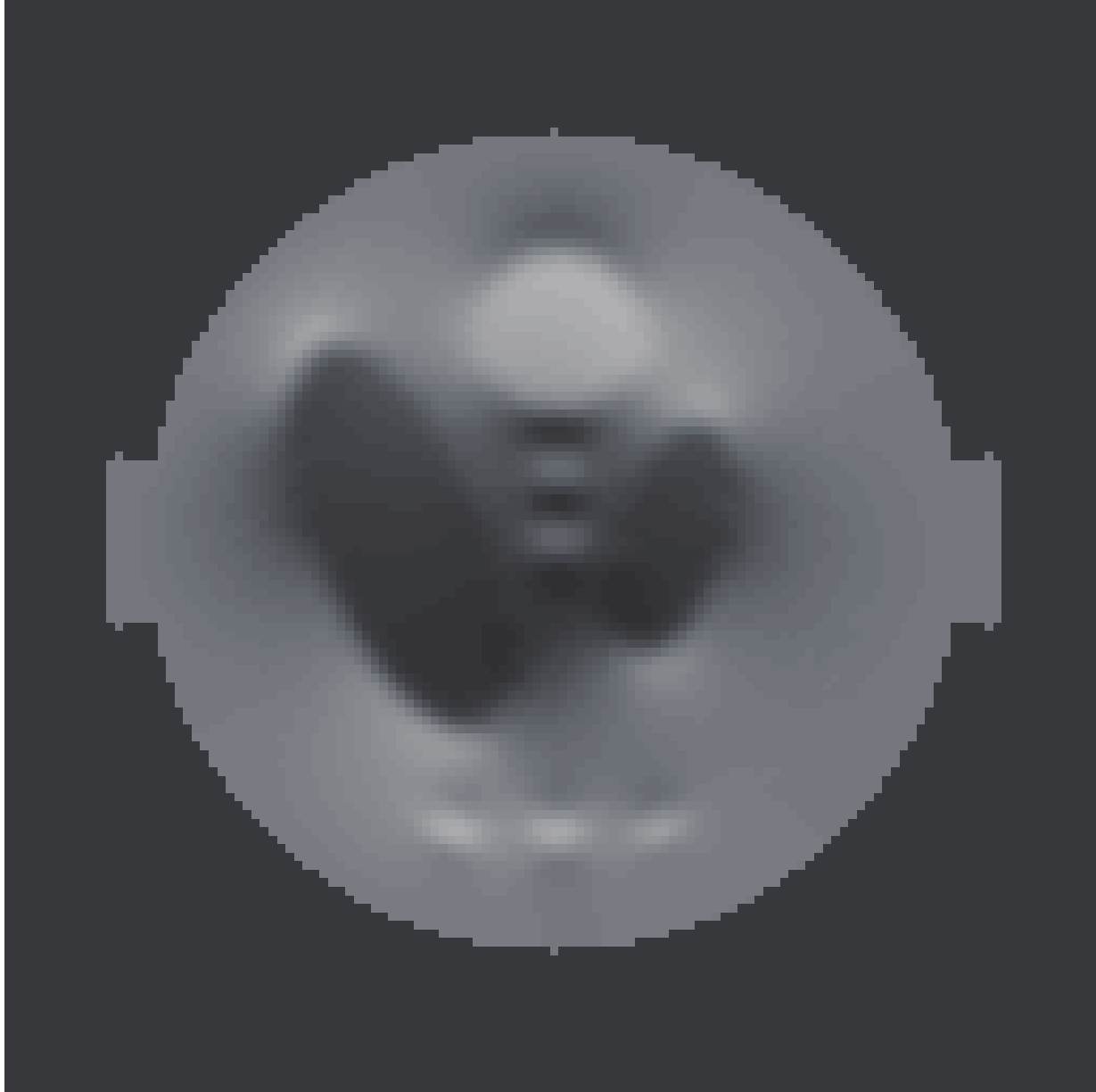}}
~\subfigure[]{
\includegraphics[width=3.7cm,height=3.7cm]{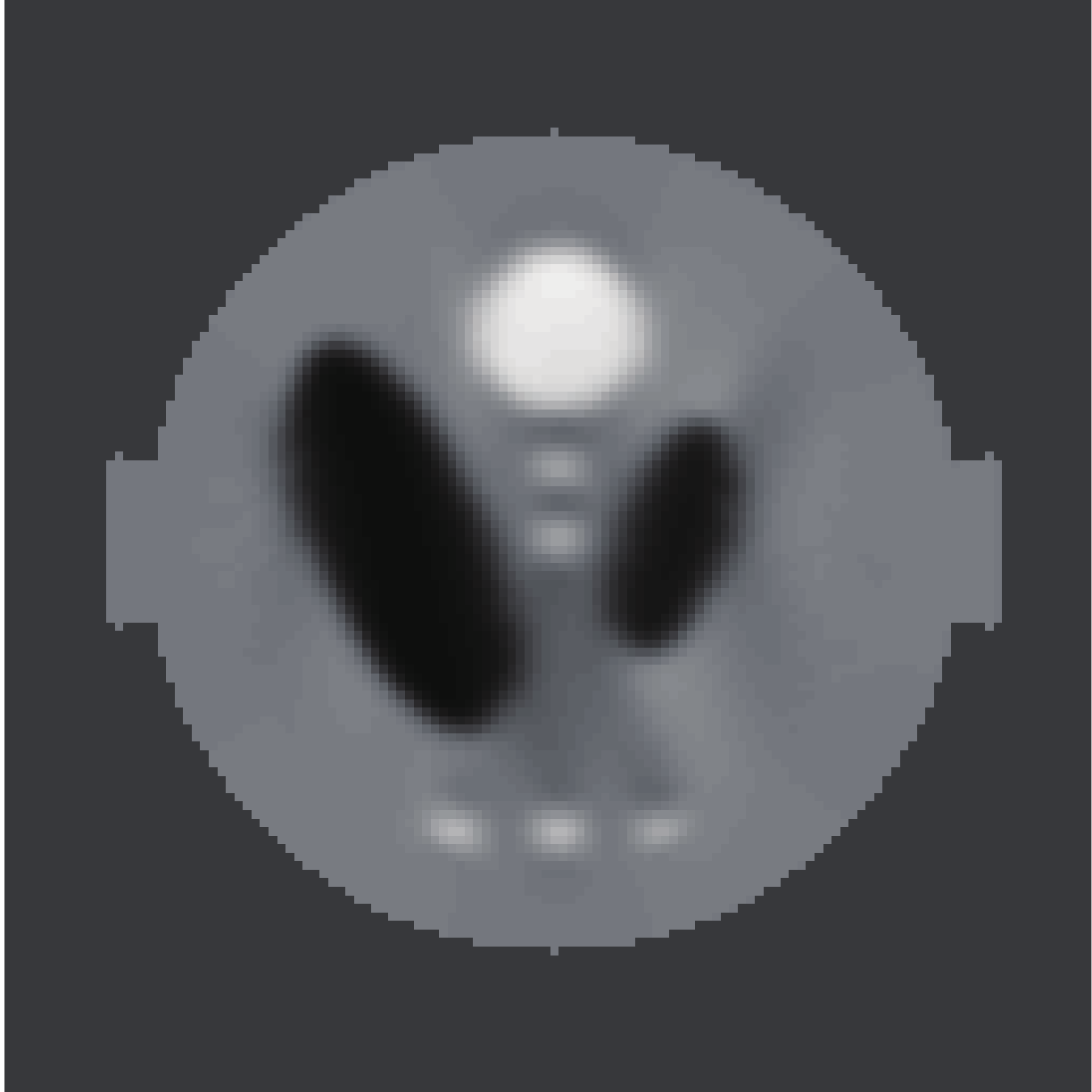}}
~\subfigure[]{
\includegraphics[width=4.5cm,height=3.7cm]{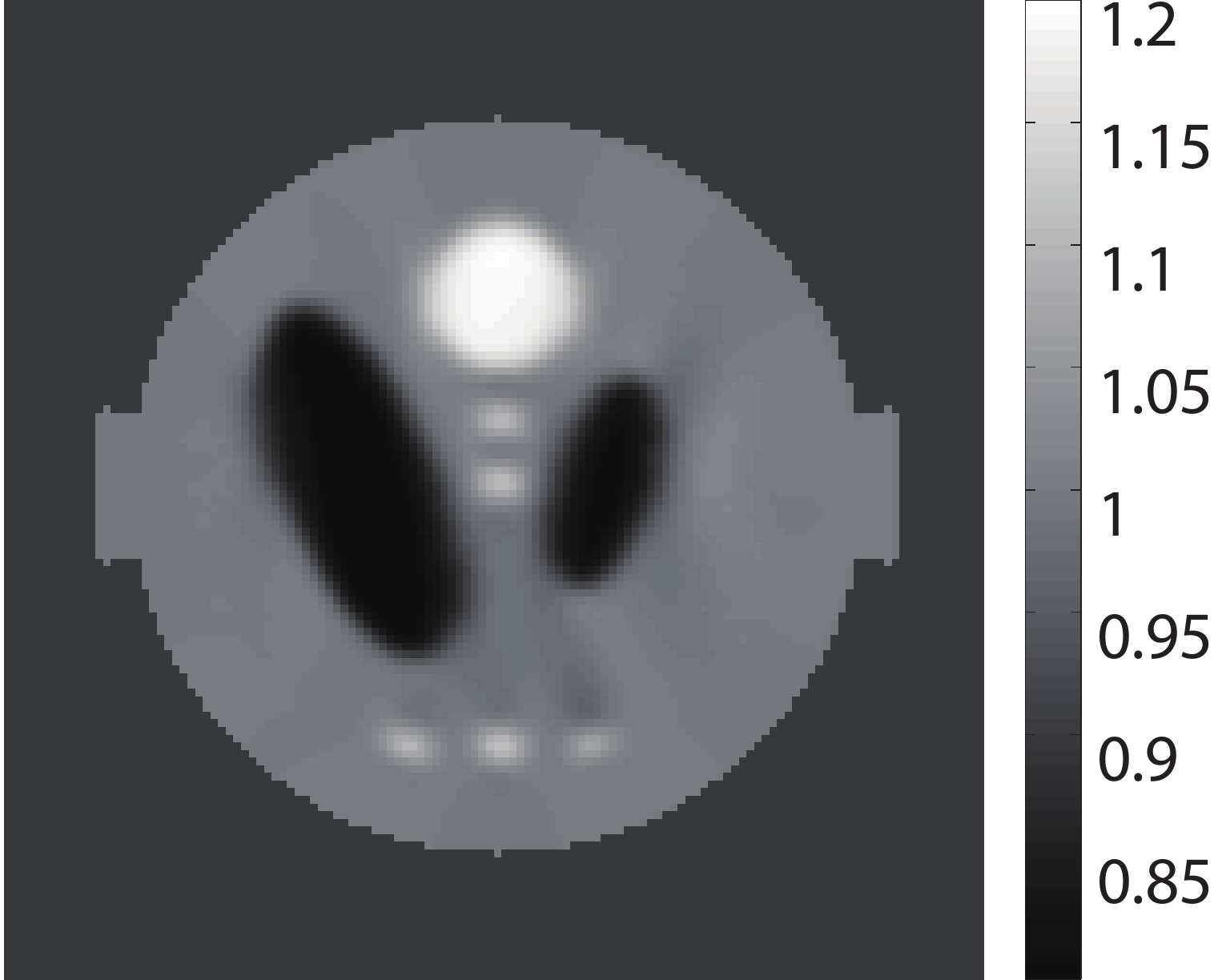}}
\caption{(a) $\sigma_2^1$, (b) $\sigma_2^{20}$, (c) $\sigma_2^{50}$, (d) $\widehat\sigma_2^1$, (e) $\widehat\sigma_2^{20}$ and (f) $\widehat\sigma_2^{50}$.}
\label{fig:Shepp-Logan_result}
\end{figure}

Asymptotic behaviors for $RE_2(n)$ and $\widehat{RE}_2(n)$ are shown for the Shepp-Logan phantom data in Figure \ref{Fig:convergence_shepp}. The values of $RE_2(n)$ and $\widehat{RE}_2(n)$ for $n=5,10,15,\cdots,50$ are also shown in Table \ref{table:RE}. We see that for $20$ and $50$ steps, the relative errors for the original conductivity distribution data are $RE_2(20) =0.1497$ and $RE_2(50)= 0.1448$ respectively, while those for the blurred distribution are $\widehat{RE}_2(20) =0.1038$ and $\widehat{RE}_2(50)= 0.0824$.
All notations used here have the same meaning as those in subsection \ref{sec:toy_simulation}.

\begin{figure}[h]
  \centering
  \includegraphics[width=10cm,height=5.0cm]{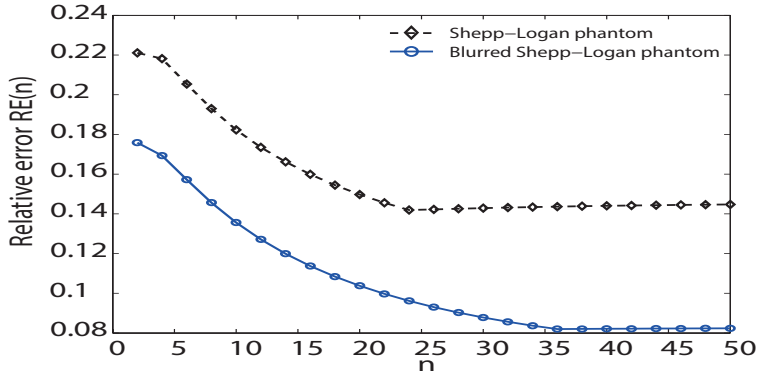}\\
  \caption{Asymptotic behaviors of $RE_2(n)$ and $\widehat{RE}_2(n)$ for $n\leq 50$. Plot legends are similar to those shown in Figure \ref{Fig:convergence_circle}.}\label{Fig:convergence_shepp}
\end{figure}

\subsection{A CT torso model}
In this subsection, we present simulation results using a realistic human torso model via the $640\times 640$ CT image shown in Figure \ref{fig:CT-torso_setup} (a), obtained from the {\it Cancer Imaging Archive} (https://www.cancerimagingarchive.net/). Here, we constructed the simulated conductivity distribution with reference to the CT image, even though there is no relation between the conductivity distribution and the CT image gray scale. We sub-sampled the CT image to a size of $128\times 128$ and attached a pair of electrodes with length of $0.03$. In Figure \ref{fig:CT-torso_setup} (b), we transformed the gray scale of the CT image into the range $[1,2]$ using the formula
\begin{equation*}\label{map_CT2cond}
  \sigma_3(\r) = \f{CT(\r)}{255}+1\qquad \mbox{for }\r\in \Om
\end{equation*}
and use it as the conductivity distribution in S/m, similar to the procedure used in \cite{Liu2010}.

Again, we imposed a blurring effect on the conductivity distribution shown in Figure \ref{fig:CT-torso_setup} (b) using the formula \eref{convolution} to obtain $\widehat \sigma_3$ for the purpose of $C^1$ smoothness. Here we used $\nu=1$ and used a convolution window size of $3\times 3$ pixels. The blurred conductivity image is shown in Figure \ref{fig:CT-torso_setup}(c).

\begin{figure}[h]
\onecolumn
\centering
\subfigure[]{
\includegraphics[width=3.7cm,height=3.7cm]{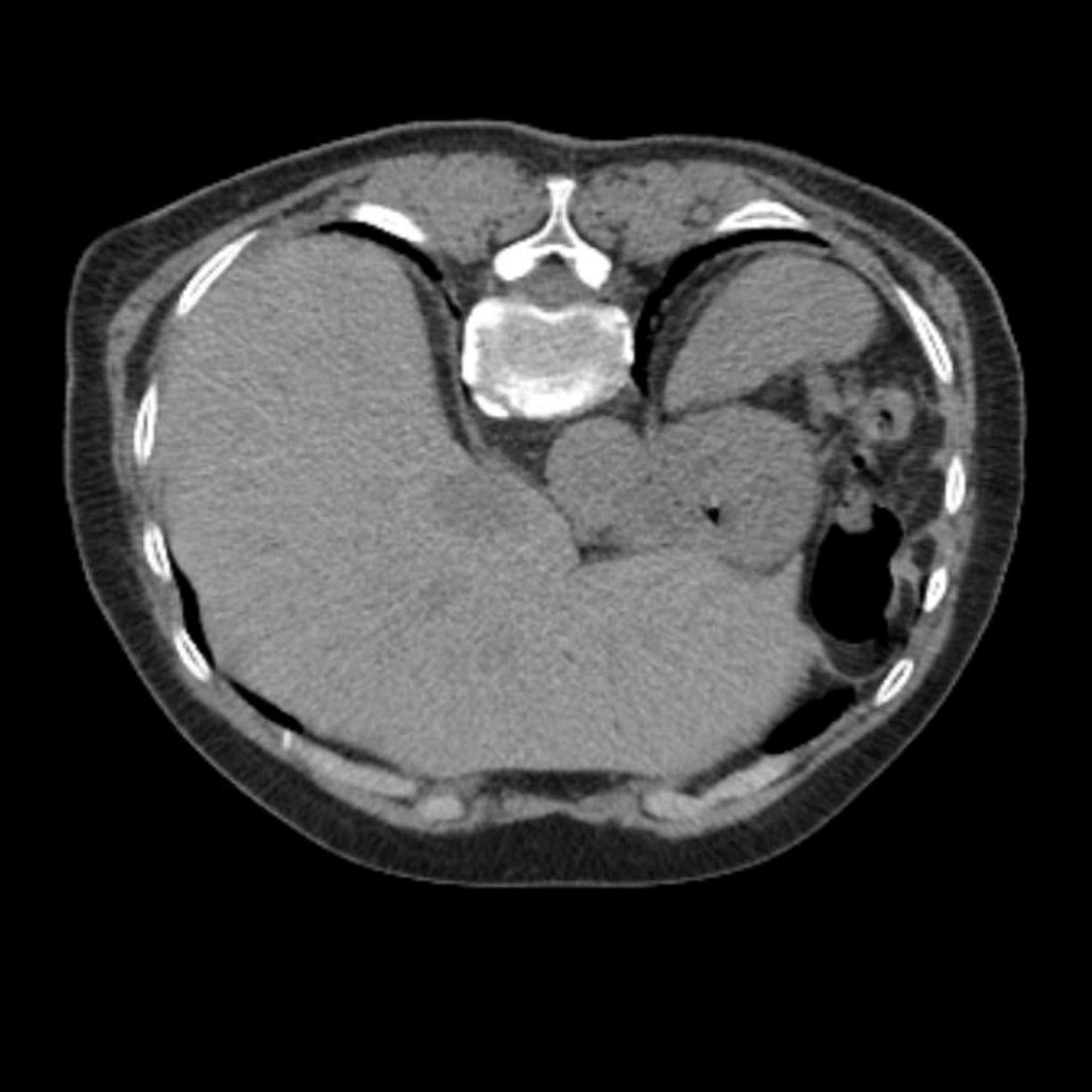}}
~\subfigure[]{
\includegraphics[width=3.7cm,height=3.7cm]{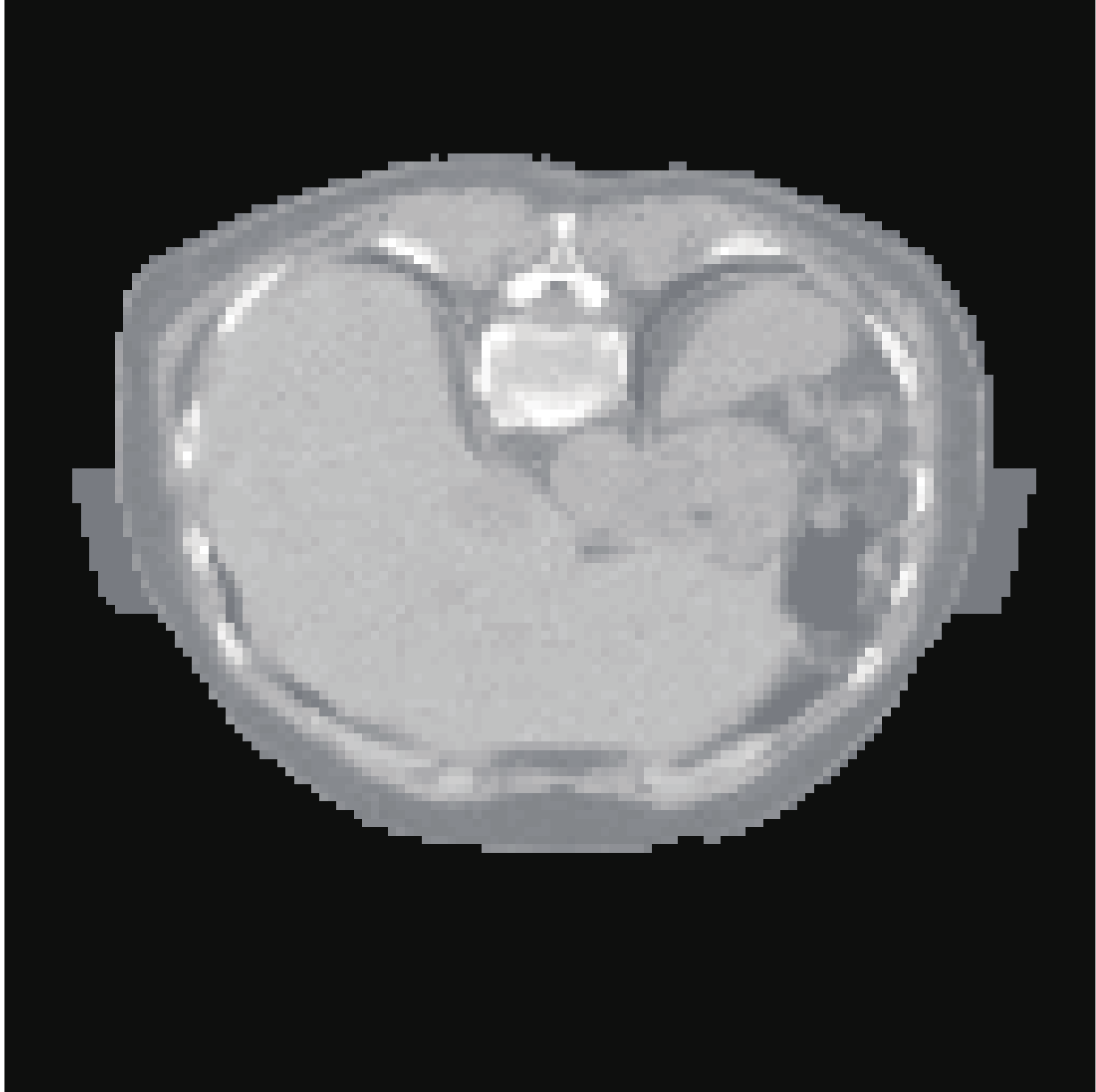}}
~\subfigure[]{
\includegraphics[width=4.5cm,height=3.7cm]{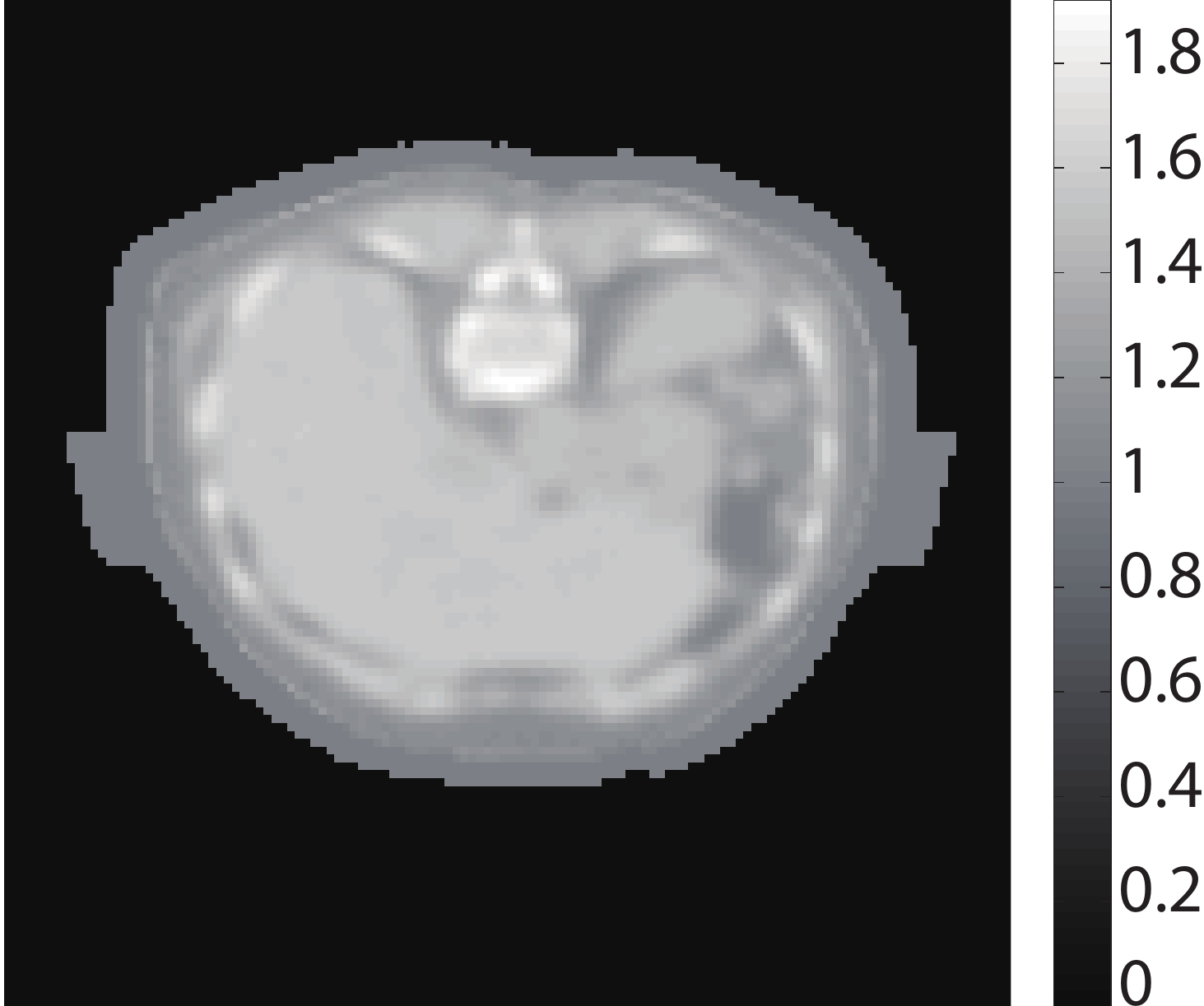}}
\caption{(a) CT image of human torso obtained from Cancer Imaging Archive, (b) $\sigma_3$ and (c) $\widehat \sigma_3$.}
\label{fig:CT-torso_setup}
\end{figure}

We again inject a current with amplitude of 10 mA through the electrodes, and the resulting images of $B_{z,3}$ and $\widehat B_{z,3}$ are shown in Figure \ref{fig:CT-torso_Bz} (a) and (b) for the conductivities $\sigma_3$ and $\widehat \sigma_3$, respectively.
Using these inversion input data, the reconstruction results $\sigma_3^n$ and $\widehat \sigma_3^n$ are shown in Figure \ref{fig:CT-torso_result} for $n=1,20$ and $50$, respectively.

\begin{figure}[h]
\onecolumn
\centering
\subfigure[]{
\includegraphics[width=3.7cm,height=3.7cm]{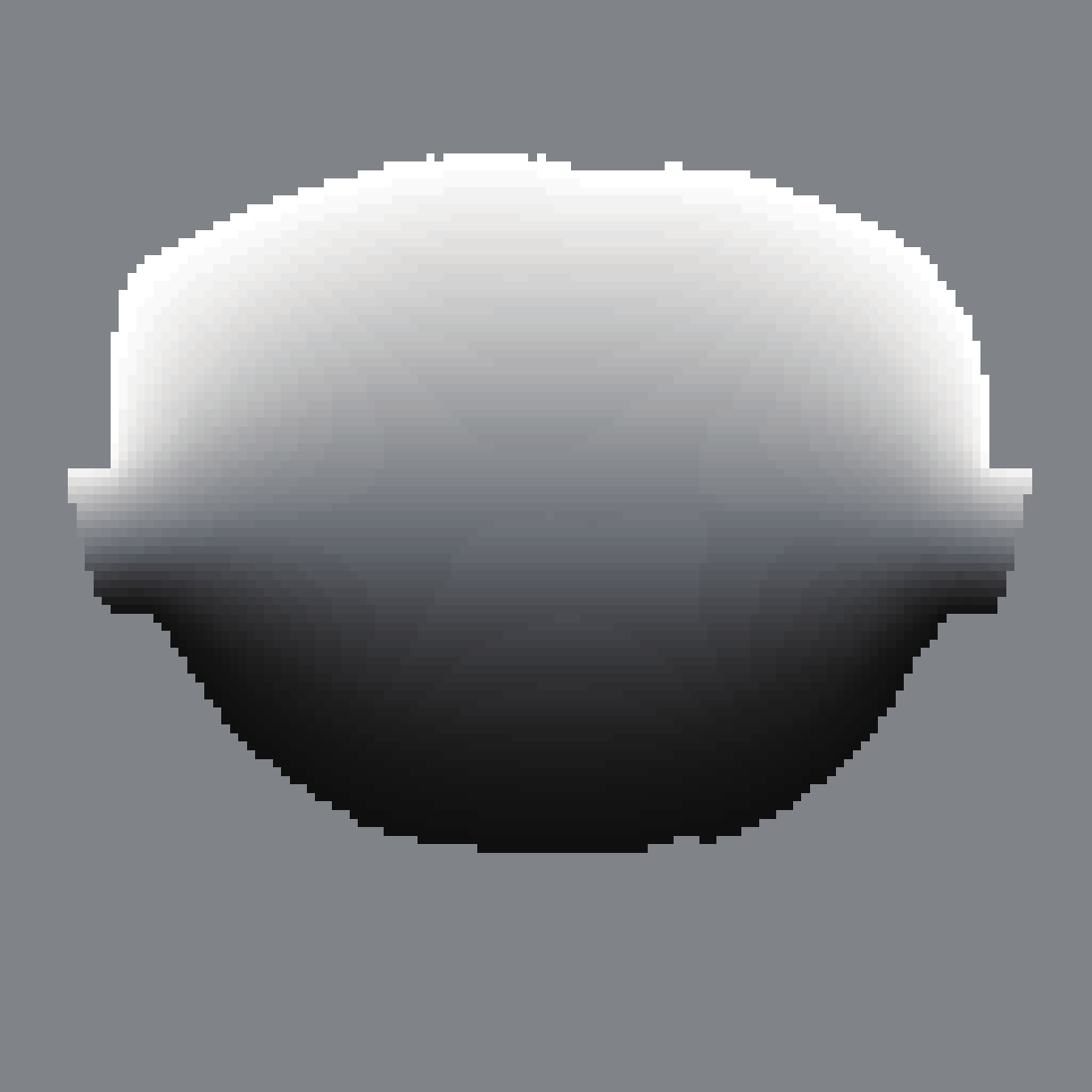}}
~\subfigure[]{
\includegraphics[width=4.5cm,height=3.7cm]{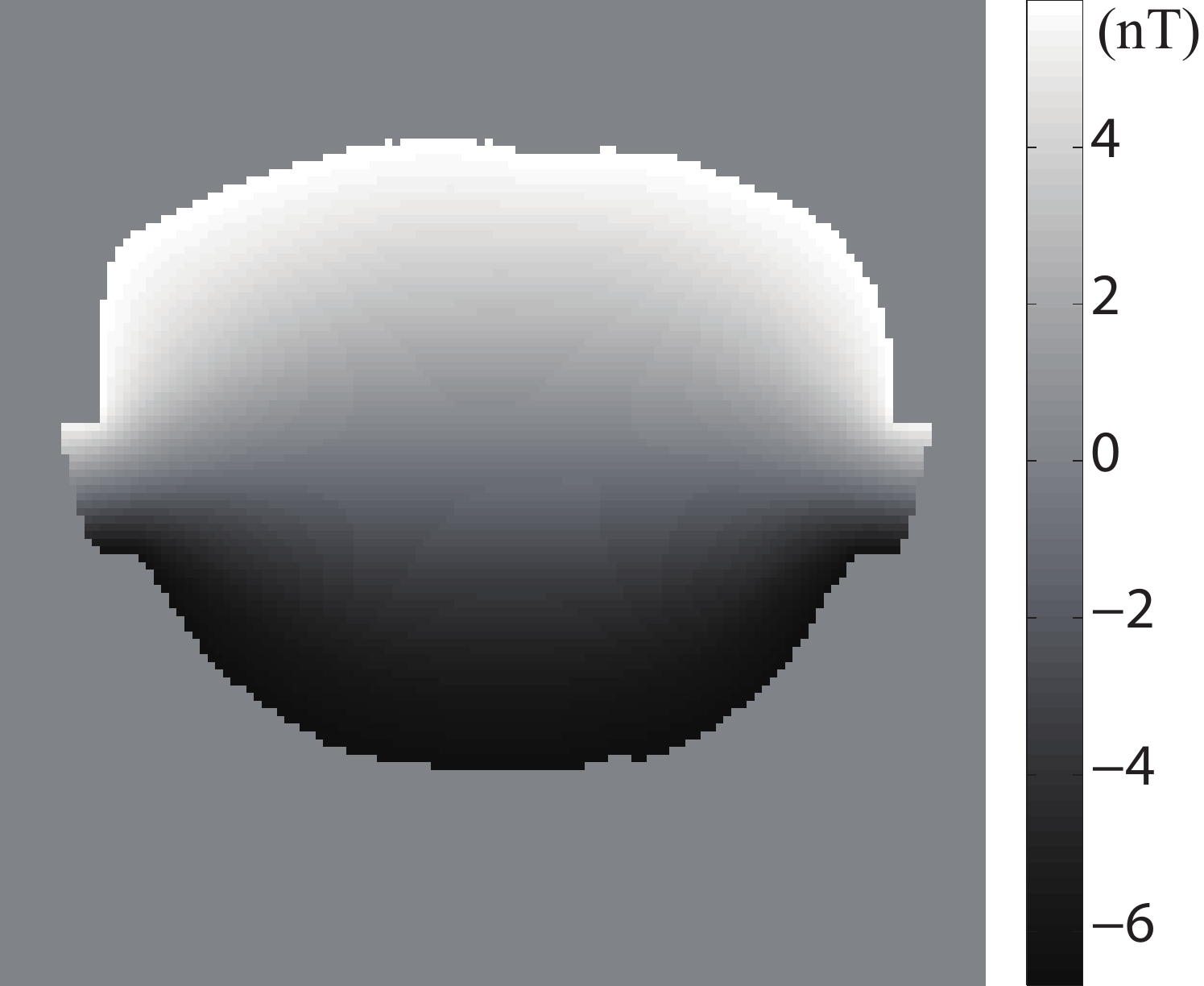}}
\caption{(a) $B_{z,3}$, (b) $\widehat B_{z,3}$.}
\label{fig:CT-torso_Bz}
\end{figure}

\begin{figure}[h]
\onecolumn
\centering
\subfigure[]{
\includegraphics[width=3.7cm,height=3.7cm]{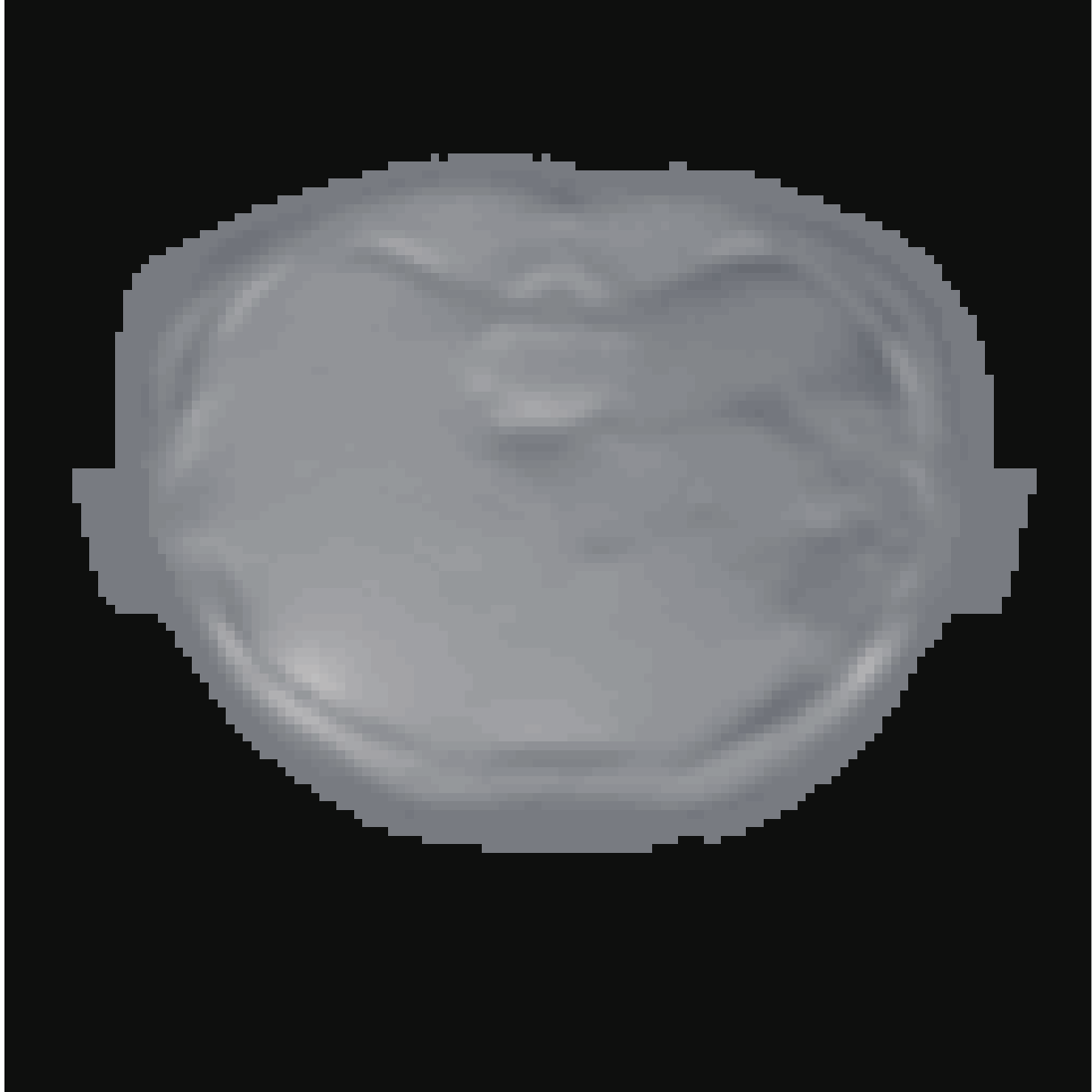}}
~\subfigure[]{
\includegraphics[width=3.7cm,height=3.7cm]{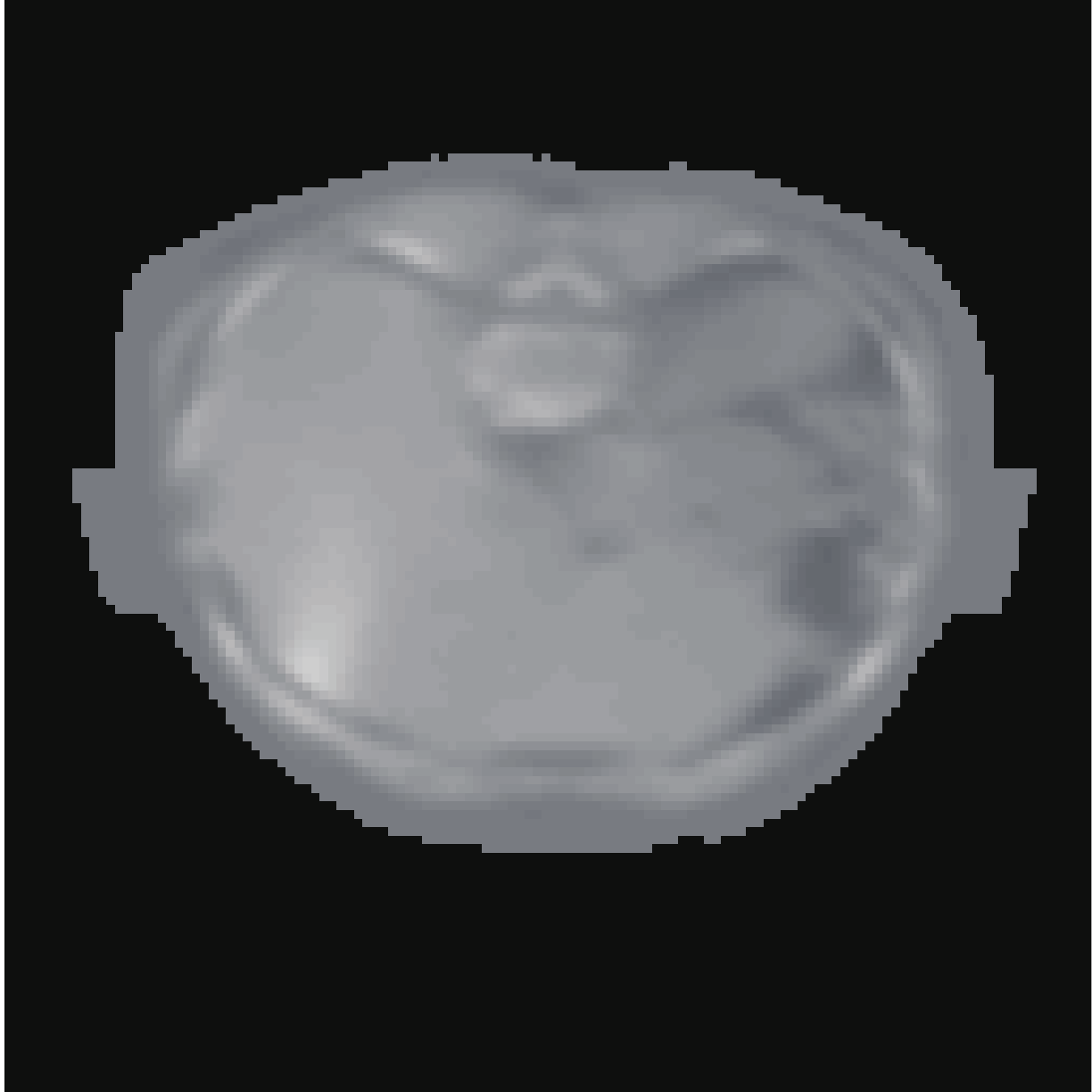}}
~\subfigure[]{
\includegraphics[width=4.5cm,height=3.7cm]{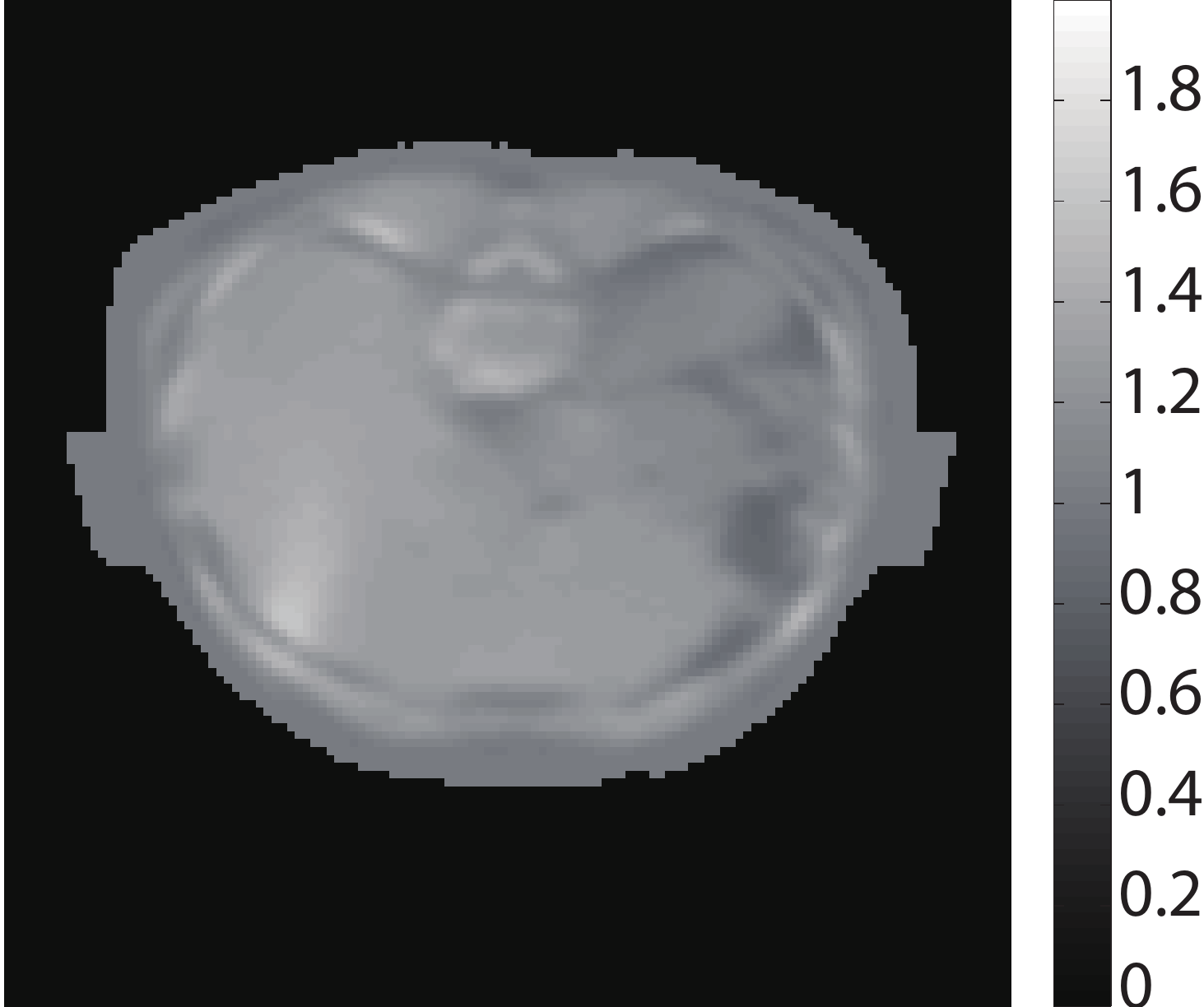}}
~\subfigure[]{
\includegraphics[width=3.7cm,height=3.7cm]{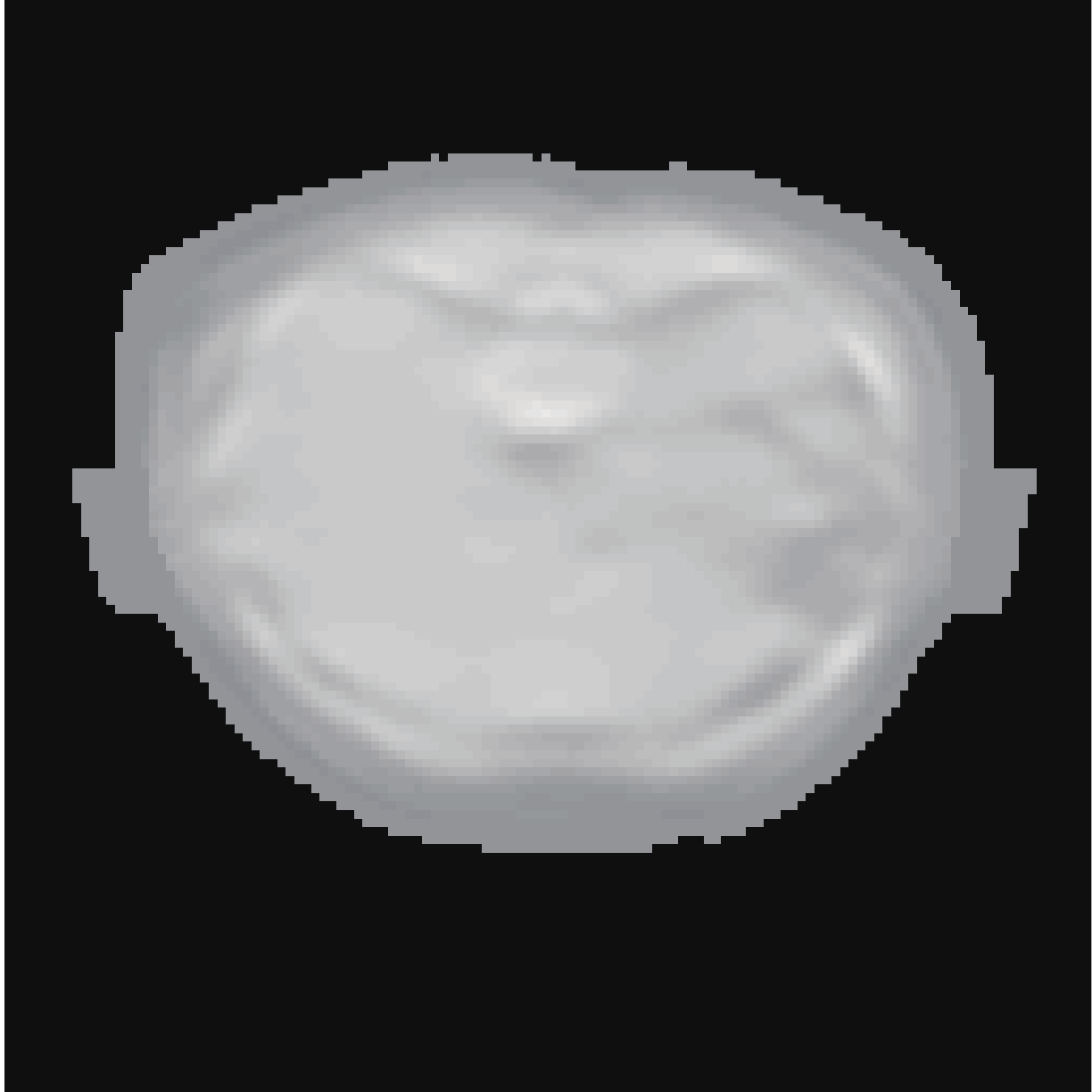}}
~\subfigure[]{
\includegraphics[width=3.7cm,height=3.7cm]{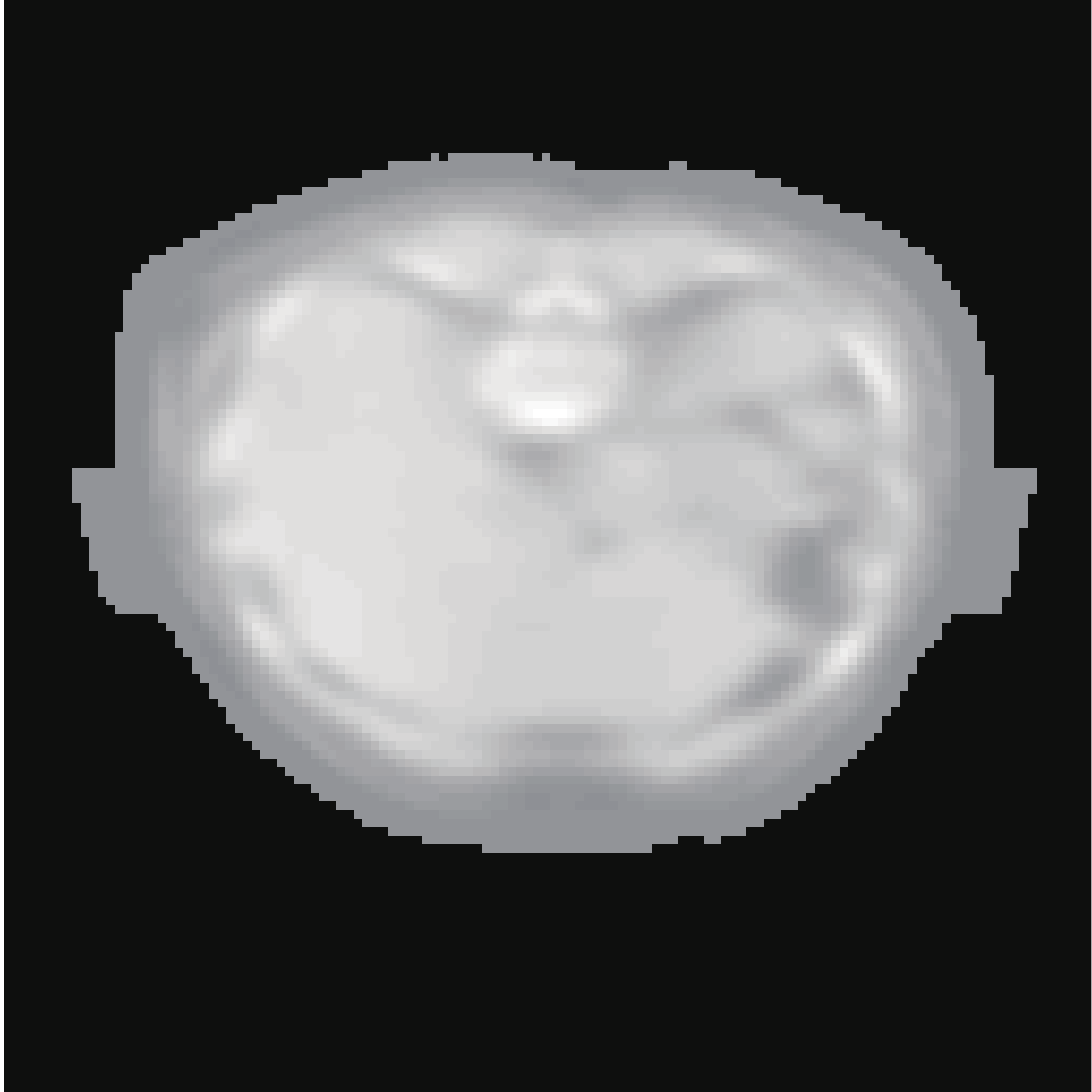}}
~\subfigure[]{
\includegraphics[width=4.5cm,height=3.7cm]{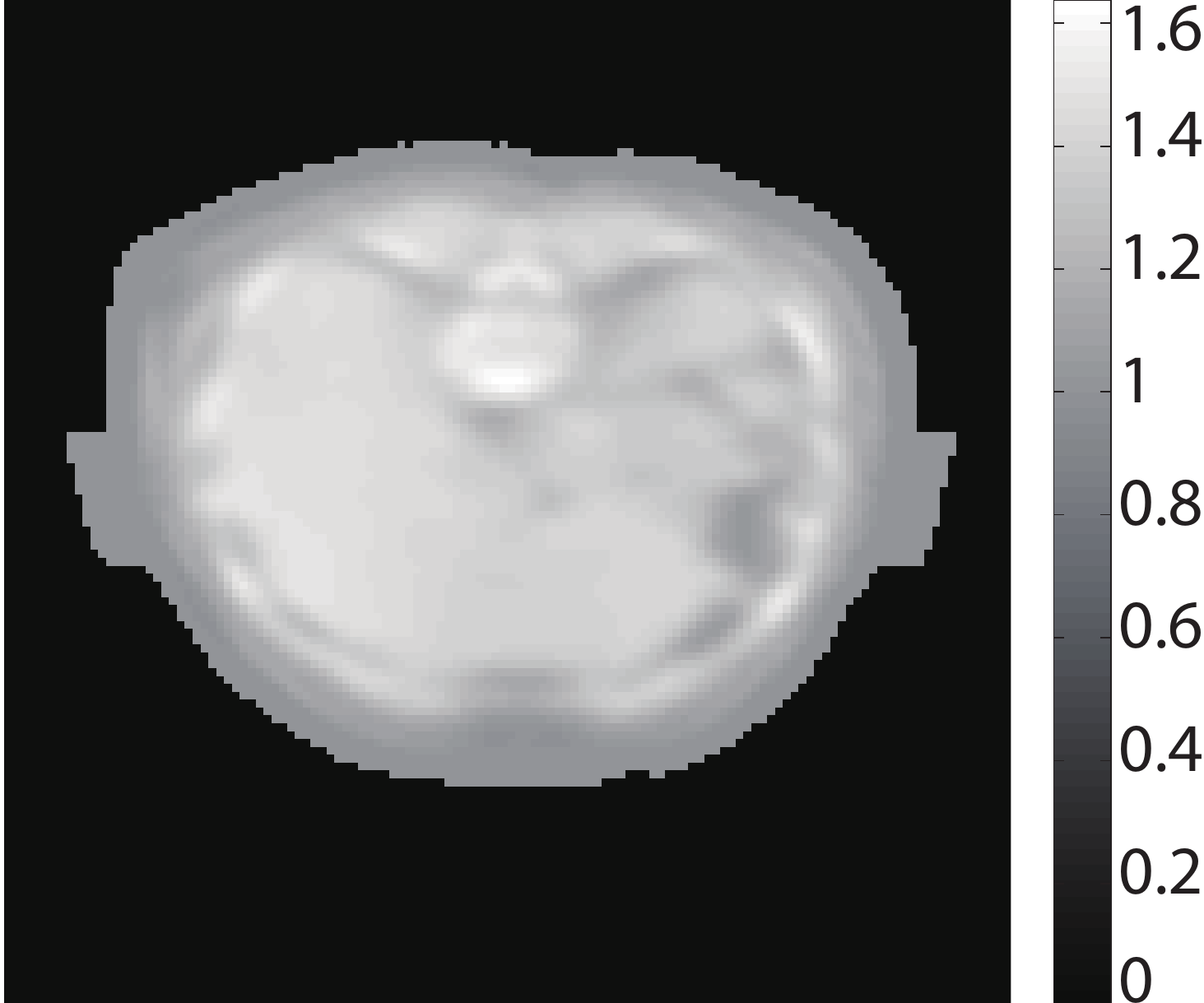}}
\caption{(a) $\sigma_3^1$, (b) $\sigma_3^{20}$, (c) $\sigma_3^{50}$, (d) $\widehat\sigma_3^1$, (e) $\widehat\sigma_3^{20}$ and (f) $\widehat\sigma_3^{50}$.}
\label{fig:CT-torso_result}
\end{figure}

The convergence behaviors for the torso model are shown in Figure \ref{Fig:convergence_CT}. Note that in this example the  relative error $RE_3(n)$ is a zigzag function, showing that convergence fails for this case. The quantitative results for $RE_3(n)$ are illustrated in Table \ref{table:RE} for $n=5,10,\cdots,50$. We believe that the main reason for this is the fact that in this example $\|\na\ln\sigma_3\|_{C^1(\Om)}$ is not small any more. If we set the conductivity to be $\widehat \sigma_3$, we obtain a much better convergence behavior, shown in the solid line with circle markers in Figure \ref{Fig:convergence_CT}. For the quantitative results for $\widehat{RE}_3(n)$ please refer to Table \ref{table:RE}.

\begin{figure}[h]
  \centering
  \includegraphics[width=10cm,height=5cm]{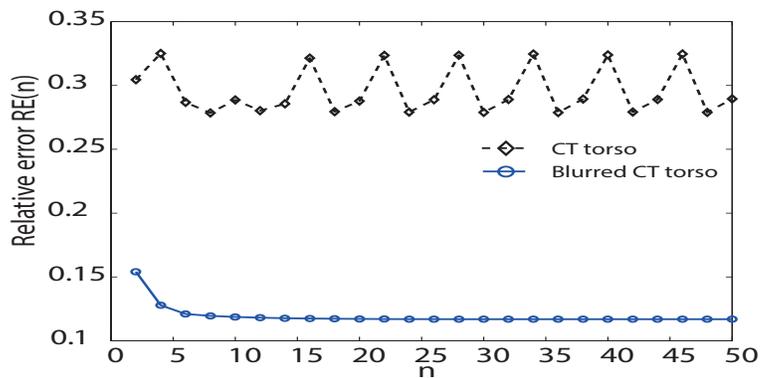}\\
  \caption{Asymptotic behaviors of $RE_3(n)$ and $\widehat{RE}_3(n)$ for $n\leq 50$.}\label{Fig:convergence_CT}
\end{figure}

\begin{table*}[h]
  \centering
  \caption{Values of $RE(n)$ for each model with $\sigma_i$ and $\widehat \sigma_i$ ($i=1,2,3$) and $n=5,10,\cdots,50$.}
\label{table:RE}
  \begin{tabular}{c|c|c|c|c|c|c}

    \hline
    \hline
    Models & \multicolumn{2}{c|}{Circle} & \multicolumn{2}{c|}{Shepp-Logan} & \multicolumn{2}{c}{CT torso} \\
    \hline
    Steps ($n$) & $RE_1(n)$ & $\widehat{RE}_1(n)$ & $RE_2(n)$ & $\widehat{RE}_2(n)$  & $RE_3(n)$ & $\widehat{RE}_3(n)$  \\
    \hline
    5 & 0.0422  & 0.0274  &  0.2120 & 0.1634  & 0.2503  & 0.1224  \\
    \hline
    10 & 0.0393  & 0.0248  & 0.1823  & 0.1356  & 0.2886  & 0.1187  \\
    \hline
    15 & 0.0391  & 0.0244  &  0.1630 & 0.1167  & 0.2259  & 0.1176  \\
    \hline
    20 & 0.0390  & 0.0243  &  0.1497 & 0.1038  & 0.2877  & 0.1172  \\
    \hline
    25 & 0.0390  & 0.0243  & 0.1421  & 0.0946  & 0.3143  & 0.1170  \\
    \hline
    30 & 0.0390  & 0.0243  & 0.1429  & 0.0878  & 0.2789  & 0.1169  \\
    \hline
    35 & 0.0390  & 0.0243  & 0.1435  & 0.0828  & 0.2800  & 0.1169  \\
    \hline
    40 & 0.0390  & 0.0243  & 0.1441  & 0.0821  & 0.3239  & 0.1169  \\
    \hline
    45 & 0.0390  & 0.0243  & 0.1445  & 0.0823  & 0.2256  & 0.1169  \\
    \hline
    50 & 0.0390  & 0.0243  & 0.1448  & 0.0824  & 0.2895  & 0.1169  \\
    \hline
    \hline

  \end{tabular}
\end{table*}

\section{Concluding remarks}

Reducing scanning time, improving signal-to-noise ratio in the $B_z$ data and developing practical methods of mapping current density and conductivity distributions in transcranial electrical stimulation applications are key issues for MREIT. Without adding prior knowledge of the unknown conductivity, a unique determination cannot in general be guaranteed  \cite{Kim2003}. However, by imposing a known boundary conductivity, we can uniquely reconstruct the conductivity distribution, as stated in Theorem \ref{thm:uniqueness} (see also \cite{ParkLeeKwon2007a}). In this paper, we reproved the uniqueness theorem using the theory of first-order hyperbolic PDEs. As a consequence, it is possible to develop reconstruction algorithms based on single current injection. Single-current reconstruction algorithms appear to have promise as a means of reducing MREIT scanning time to at least half that of previous two-current-based algorithms. It should streamline inclusion of MREIT procedures into existing transcranial electrical stimulation protocols using functional MRI and aid in understanding the mechanisms of action of these techniques.

In \cite{SongSadleir2020}, we developed an iterative reconstruction algorithm, the single current harmonic $B_z$ algorithm, to reconstruct the conductivity from a single current administration. In this paper, we provided a strict mathematical analysis of the convergence behavior of the single current harmonic $B_z$ algorithm. The main result of this paper, is to show that if the $C^1$ norm of the unknown conductivity is sufficiently small, the sequence $\{\ln\sigma^n\}$ converges to the true value $\ln\sigma^*$ in the sense of $C^1$. Through numerical experiments we showed that if $\|\na\ln\sigma\|$ is not small, convergence may not occur. However, even with a one step reconstruction we can still obtain the correct geometry of the internal structure and a correct local contrast of the unknown $\sigma$ which is similar to that found in electrical impedance tomography \cite{Harrach2010}. However, the strict theory for this observation is yet to be proven.

As well as reducing scanning time, another advantage of one-current based MREIT reconstruction algorithms lies in the fact that it is possible to analyze the stability and achievable  resolution of the reconstruction algorithm \cite{SeoWoo2009}. To be precise, we do not need to estimate the lower bound of the area of the parallelogram formed by two linearly independent current densities $\J^1$ and $\J^2$ due to two current injections. The only necessity is the lower bound of $\J$, which is guaranteed in the two-dimensional case by Corollary \ref{corollary:uniform_lower_bound} and \cite{Alessandrini2004}. A strict mathematical analysis of the stability and achievable resolution of single-current reconstruction methods are clearly needed in future research.

In spite of the fact that in 2D cases the current density $\J$ is non-zero, in real situations it could be very close to zero in local region such as bubble, bone and airways. In these regions, the term $\f{1}{|\J|^2}$ could amplify the noise in $B_z$. To solve this problem, an appropriate regularization involving the {\it a priori} information about the unknown conductivity and current streamlines should be developed to improve the quality of the reconstructed image. If these approaches prove successful in numerical and phantom testing animal and human experiments will be performed to further verify the proposed algorithm and the convergence theory.

\end{document}